\documentclass[11pt]{article}
\usepackage[T1]{fontenc}
\usepackage{lmodern,mathptmx}
\usepackage{amsmath,amssymb,amsthm,mathtools}
\usepackage[margin=0.95in]{geometry}
\usepackage{microtype,booktabs,longtable,array,enumitem,listings,xcolor}
\usepackage[hidelinks]{hyperref}
\usepackage{bookmark}
\hypersetup{pdftitle={Fuglede's conjecture holds in F\string_p\string^d for p=2,d=8 and p=11,d=3},
 pdfauthor={Samuel J. Ferguson and Joshua Iosevich},linktoc=all}
\newtheorem{theorem}{Theorem}[section]
\newtheorem{lemma}[theorem]{Lemma}
\newtheorem{proposition}[theorem]{Proposition}
\newtheorem{corollary}[theorem]{Corollary}
\theoremstyle{definition}

\theoremstyle{remark}

\newcommand{\F}{\mathbb F}
\newcommand{\Z}{\mathbb Z}

\newcommand{\R}{\mathbb R}

\DeclareMathOperator{\supp}{supp}
\DeclareMathOperator{\rank}{rank}
\DeclareMathOperator{\rad}{rad}
\DeclareMathOperator{\wt}{wt}

\DeclareMathOperator{\PG}{PG}
\newcommand{\D}{\mathcal D}
\newcommand{\U}{\mathcal U}

\newcommand{\pathfile}[1]{\path{#1}}

\numberwithin{equation}{section}
\setlist{itemsep=2pt,topsep=5pt}
\title{Fuglede's conjecture holds in $\mathbb{F}_{p}^{d}$\\
for $p=2,\,d=8$ and $p=11,\,d=3$}
\author{Samuel J. Ferguson\thanks{Independent researcher.
Email: \href{mailto:sjf370@nyu.edu}{\texttt{sjf370@nyu.edu}}.}
\and Joshua Iosevich\thanks{School of Mathematics and Statistics,
Rochester Institute of Technology, 1 Lomb Memorial Drive,
Rochester, NY 14623, USA.
Email: \href{mailto:ji1964@rit.edu}{\texttt{ji1964@rit.edu}}.}}
\date{}
\begin{document}
\maketitle
\begin{abstract}
Fuglede's conjecture in $\F_p^d$ says that a nonempty set tiles by
translation if and only if it is spectral. We prove the conjecture in
$\F_2^8$ and $\F_{11}^3$. Over $\F_2$ we show more generally that
the Fourier support of a nonconstant indicator of density different from
one half contains a plane. Consequently every spectral set of density one
quarter has a subspace tiling complement. This supplies one of the
dimension-eight reductions left by Kadowaki; the others are completed
using classified Hadamard matrices and perfect binary codes. In $\F_2^9$
we prove that every tile is spectral and that every spectral set has
power-of-two cardinality. All spectral cardinalities except $64$ are
shown to tile, and we give further reductions of the remaining case.
We also construct a spectral set of density one eighth whose Fourier
support contains no three-dimensional subspace. For $\F_{11}^3$ we extend
the geometric method of Fallon, Mayeli, and Villano by using moment
identities, polynomial plane profiles, weighted incidence counts, and
exact finite computations. The computations and their verification
programs are described in the appendices.
\end{abstract}
\clearpage
\setcounter{tocdepth}{1}
\tableofcontents
\clearpage

\section{Introduction}

In this paper we prove Fuglede's conjecture for the vector spaces
$\F_2^8$ and $\F_{11}^3$, and reduce the conjecture in $\F_2^9$ to one
spectral cardinality. Here $\F_p$ denotes the field with $p$ elements;
thus $\F_p^d$ is the group also denoted by $\Z_p^d$.

We briefly recall the original problem. A measurable set $E\subset\R^d$
of finite, positive measure is called spectral if $L^2(E)$ has an
orthogonal basis of exponential functions
$x\mapsto e^{2\pi i\lambda\cdot x}$, with $\lambda$ in a set
$\Lambda\subset\R^d$. It tiles by translation if translates of $E$
partition $\R^d$ up to null sets. Fuglede \cite{Fuglede} conjectured
that these two properties are equivalent, in connection with commuting
self-adjoint realizations of the partial differential operators. Tao
\cite{Tao} disproved the conjecture in dimension five and higher using
finite-group constructions. Subsequent work lowered the dimension of
counterexamples and established failure in both directions; see
Matolcsi \cite{Matolcsi}, Kolountzakis and Matolcsi \cite{KM,KMtiles},
Farkas and R\'ev\'esz \cite{FR}, and Farkas, Matolcsi, and M\'ora
\cite{FMM}. These developments make the finite-group version of the
problem of interest in its own right. The relation with factorizations of finite abelian groups and the use of universal spectra were developed further by Lagarias and Wang~\cite{LW} and Lagarias and Szab\'o~\cite{LS}.

Let $G=\F_p^d$. A nonempty subset $E$ tiles $G$ if there is a set
$T\subseteq G$ such that every element of $G$ has exactly one
representation $e+t$, where $e\in E$ and $t\in T$. We write
$E\oplus T=G$. We call $E$ spectral if there is a set
$\Lambda\subseteq G$ for which the characters
\[
 \chi_\lambda(x)=\exp(2\pi i\lambda\cdot x/p),\qquad\lambda\in\Lambda,
\]
form an orthogonal basis for the space of complex functions on $E$,
equipped with counting measure. Such a $\Lambda$ is called a spectrum
of $E$. Necessarily $|\Lambda|=|E|$, and we call $(E,\Lambda)$ a
spectral pair. Fuglede's conjecture in $G$ asserts that the nonempty
spectral sets are exactly the translational tiles.

Iosevich, Mayeli, and Pakianathan \cite{IMP} proved the conjecture in
$\F_p^2$ for every prime $p$. Another approach appears in Kiss,
Malikiosis, Somlai, and Vizer \cite{KMSV}. Aten et al.\ \cite{Aten}
developed the finite-field theory, including the correspondence with
log-Hadamard matrices, cardinality restrictions, and the tiling-to-spectral
implication in dimension three. They also constructed counterexamples in
dimension five for every odd prime and in dimension four for primes
congruent to three modulo four. Ferguson and Sothanaphan \cite{FS}
extended the latter construction to every odd prime, proved the case over $\F_2$
in dimension four, and gave computational proofs in dimensions five
and six, with the computations explained further in \cite{FSadd}.
Their dephasing argument also places a counterexample over $\F_2$ in
dimension ten.

Two positive results immediately precede the cases considered here.
Fallon, Mayeli, and Villano \cite{FMV} proved the conjecture in
$\F_p^3$ for $p=5,7$. Their proof turns Fourier zeros into equal
occupancies of parallel planes. Bounds for directions determined by
small planar sets then force a spectral set into a small number of
parallel planes, where additional counting arguments give a
contradiction. For $p=7$ the three intermediate cardinalities are
$14,21,28$. The $21$-point case requires a more detailed analysis of
projected multiplicities. We retain this geometric starting point at
$p=11$, but some of the planar lemmas used at seven are false at eleven.
Our replacement arguments use polynomial moments and the incidence
geometry of the projective plane. Malikiosis's application of Delsarte's
linear programming method \cite{Delsarte,Malikiosis} first removes the three largest intermediate
cardinalities, leaving six cases for us to consider.

Kadowaki \cite{Kadowaki} proved the conjecture in $\F_2^7$. His proof
provides three particularly useful ingredients for the next dimension:
a common-spectrum construction for the tiling complements of sets of
size at most eight; a Smith normal form obstruction to non-power-of-two
spectral cardinalities; and a projection onto canonical point sets of
dephased Hadamard matrices. The Smith obstruction already settles the
cardinality question in dimension eight. His Remark~7.2 identifies the
remaining dimension-eight cases as $16$-point tiles and spectral sets of
sizes $32$ and $64$, and records the $166$ extended Hadamard classes
of order $32$ and dephased rank at most eight. Our proof over $\F_2$ starts
from precisely these reductions. The $16$-point tile case is completed
by passing to the translation kernels of perfect codes. The $32$-point
spectral case is completed by subgroup-complement certificates for both
orientations of the $166$ classes. The $64$-point spectral case follows
from a short convolution identity, without a classification of
Hadamard matrices of order $64$.

We state our principal conclusions as follows.

\begin{theorem}\label{main:binary8}
For a nonempty subset $E$ of $\F_2^8$, the following are equivalent:
$E$ is spectral; $E$ tiles $\F_2^8$; and $E$ has a tiling complement
which is a linear subspace.
\end{theorem}

\begin{theorem}\label{main:binary9}
Every tile of $\F_2^9$ is spectral. Every spectral subset of $\F_2^9$
has power-of-two cardinality. If a spectral set has cardinality
different from $64$, it tiles $\F_2^9$ by a linear subspace.
\end{theorem}

\begin{theorem}\label{main:p11}
A nonempty subset of $\F_{11}^3$ is spectral if and only if it tiles
$\F_{11}^3$ by translation.
\end{theorem}

The analytic observation behind the $64$-point case of
Theorem~\ref{main:binary8} applies in every dimension over $\F_2$. If $B$
is a nonempty proper subset of $\F_2^d$ of density other than one
half, the support of its Fourier transform contains a two-dimensional
linear subspace. Applying this to a spectrum shows that every spectral
set of density one quarter tiles by a plane. One might hope that the
support of the Fourier transform of a spectral set of density one eighth
must similarly contain a three-dimensional subspace. We give an
explicit example showing that this is false in $\F_2^9$. The example
and its spectrum both tile; what fails is the asserted structure of
the Fourier support.

The remaining $64$-point problem in dimension nine is not settled
here. We prove tiling in the periodic case, give projection criteria
that lower the ambient dimension, and settle the seven unequal
four-fiber patterns in a normalized cubic-rank-three branch. The
uniform four-fiber pattern and the other possible cubic ranks remain
outside that argument. Thus the additional reductions do not assert
the full conjecture in $\F_2^9$.

Several of the proofs use finite classifications. We use the published
classifications of Hadamard matrices of orders $16$ and $32$, perfect
binary codes of length $15$, maximal triply even codes of length $48$,
and small binary self-dual codes. Each use is separated into a
mathematical reduction, a finite statement, and an exact verification
procedure. The odd-prime proof also uses explicit enumerations of
small planar configurations and one mixed-parity height calculation.
The appendices give the computational specifications and executable
code; the accompanying source package contains the complete programs,
data, and verification records. No numerical tolerance is used in
these finite verifications.

We begin with the Fourier lemmas over $\F_2$ and the proof in dimension
eight. We then treat dimension nine, including the failure of the
eighth-density support assertion. The last part proves the result
over $\F_{11}$ and compares the arguments with those at $p=7$.


\section{\texorpdfstring{Fourier transforms over $\F_2$ and the quarter-density lemma}{Fourier transforms over F2 and the quarter-density lemma}}
\label{b8:preliminaries}

We first consider vector spaces over $\F_2$. Put $G=\F_2^d$ and $N=2^d$, and
identify $G$ with its character group by
$\chi_u(x)=(-1)^{u\cdot x}$. For $A\subseteq G$, we use the unnormalized
Walsh transform
\[
 W_A(u)=\sum_{x\in A}(-1)^{u\cdot x},
 \qquad \supp W_A=\{u\in G:W_A(u)\ne0\}.
\]
These characters give the finite form over $\F_2$ of Walsh's
orthogonal system~\cite{Walsh}.
All sums of points in a vector space over $\F_2$ are taken in that space. Thus the
difference set of $A$ is $A+A$.

If $(E,\Lambda)$ is a spectral pair, its character matrix $H$ satisfies
$HH^{\mathsf T}=|E|I$. Since $H$ is square, also
$H^{\mathsf T}H=|E|I$. In particular, $(\Lambda,E)$ is a spectral pair,
and
\begin{equation}\label{b8:orthogonality}
 (E+E)\setminus\{0\}\subseteq\{u:W_\Lambda(u)=0\},\qquad
 (\Lambda+\Lambda)\setminus\{0\}\subseteq\{u:W_E(u)=0\}.
\end{equation}
For a tiling $E\oplus T=G$, the identity $1_E*1_T=1_G$ gives
\begin{equation}\label{b8:tilingfourier}
 W_E(u)W_T(u)=0\qquad(u\ne0).
\end{equation}
We will use these two observations in the following form.

\begin{lemma}\label{b8:certificate}
Let $E,T,\Lambda,B\subseteq G$.
\begin{enumerate}
\item If $(E,\Lambda)$ is a spectral pair,
$|T|=N/|E|$, and $T+T\subseteq\supp W_\Lambda$, then $E\oplus T=G$.
\item If $E\oplus T=G$, $|B|=|E|$, and
$B+B\subseteq\supp W_T$, then $B$ is a spectrum of $E$.
\end{enumerate}
\end{lemma}
\begin{proof}
In the first assertion, a repeated sum $e+t=e'+t'$ with $e\ne e'$
would give a nonzero element of $(E+E)\cap(T+T)$, contrary to
\eqref{b8:orthogonality}. Hence all $|E||T|=N$ sums are distinct.
In the second assertion, \eqref{b8:tilingfourier} makes the characters
indexed by distinct elements of $B$ orthogonal on $E$. Their number is
$\dim L^2(E)$, so they form a basis.
\end{proof}

The second assertion is the universal-spectrum sufficient condition of
Lagarias and Szab\'o~\cite{LS}, in the form of
Farkas, Matolcsi, and M\'ora~\cite[Proposition~4]{FMM}.
Neither of the sets $T$ and $B$ in this lemma is required to be a
subspace. This distinction will be useful in dimension nine.

\begin{lemma}[Projection and lifting]\label{b8:lifting}
Let $q:G\to V$ be a surjective linear map whose restriction to $E$ is
injective. If $q(E)\oplus K=V$, then $E\oplus q^{-1}(K)=G$.
In particular, a linear-subspace tiling complement of $q(E)$ lifts to
one of $E$. If $E$ has a linear-subspace tiling complement $K$ in $G$,
then $K^\perp$ is a spectrum of $E$.
\end{lemma}
\begin{proof}
For $x\in G$, the decomposition of $q(x)$ specifies a unique point
$q(e)\in q(E)$, and hence a unique $e\in E$, such that
$q(x-e)\in K$. This gives the required unique decomposition of $x$.
For the last assertion, $E$ contains exactly one representative of each
coset of $K$. Characters in $K^\perp$ are constant on these cosets, and
their restrictions to $E$ have the same orthogonality relations as the
characters of $G/K$.
\end{proof}

The projection argument appears in Szegedy's work~\cite{Szegedy} and in the finite
Fuglede literature; we use it in the form of
\cite[Lemma~2.2]{Kadowaki}. Our first new ingredient is a restriction on
the nonzero Fourier support of a Boolean function.

\begin{lemma}[The support-plane lemma]\label{b8:supportplane}
If $\varnothing\ne A\ne G$ and $|A|\ne N/2$, then $\supp W_A$
contains a two-dimensional linear subspace.
\end{lemma}
\begin{proof}
Put $a(u)=N^{-1}W_A(u)$ and $\delta=a(0)=|A|/N$.
Fourier transformation of $1_A^2=1_A$ gives $a*a=a$, where convolution
on the frequency group is an unnormalized sum. For $s\ne0$, separating
the terms $t=0,s$ gives
\begin{equation}\label{b8:idempotence}
 (1-2\delta)a(s)=\sum_{t\notin\{0,s\}}a(t)a(s+t).
\end{equation}
By Fourier inversion, $a(s)\ne0$ for some $s\ne0$; otherwise $1_A$
would be constant. The left side of \eqref{b8:idempotence} is nonzero,
so at least one summand on the right is nonzero. For this $t$, all four
points $0,s,t,s+t$ belong to $\supp W_A$. Since $t\notin\{0,s\}$,
they form a plane.
\end{proof}

\begin{theorem}[The quarter-density lemma]\label{b8:quarter}
Let $E\subseteq\F_2^d$ be spectral. Then
\[
 |E|\le 2^{d-2},\qquad |E|=2^{d-1},\qquad\text{or}\qquad E=\F_2^d.
\]
If $|E|=2^{d-2}$, then $E$ has a two-dimensional linear-subspace
tiling complement. If $|E|=2^{d-1}$, it has a one-dimensional
linear-subspace tiling complement.
\end{theorem}
\begin{proof}
Let $\Lambda$ be a spectrum of $E$. Suppose first that $\Lambda$ is
proper and $|\Lambda|\ne N/2$. By Lemma~\ref{b8:supportplane}, there
is a plane $K\subseteq\supp W_\Lambda$. Equation
\eqref{b8:orthogonality} implies
$K\cap(E+E)=\{0\}$, so the four translates $E+k$, $k\in K$, are
disjoint. Thus $4|E|\le N$, with equality giving a tiling.
If $|\Lambda|=N/2$, Fourier inversion supplies a nonzero $s$ with
$W_\Lambda(s)\ne0$. The same argument with $K=\{0,s\}$ gives a
tiling. Finally, $|\Lambda|=N$ implies $|E|=N$.
\end{proof}

For $d=8$, the quarter-density case has $64$ points. The theorem
therefore settles this cardinality without a classification of
Hadamard matrices of order $64$. This is the reason the Fourier
identity above is particularly useful here. Kadowaki's proof in
dimension seven already reduces the dimension-eight problem to
sixteen-point tiles and spectral sets of sizes $32$ and $64$
\cite[Remark~7.2]{Kadowaki}. We shall settle the first two cases by
finite classifications and the last one by Theorem~\ref{b8:quarter}.

\section{\texorpdfstring{Fuglede's conjecture in $\F_2^8$}{Fuglede's conjecture in F2 to the eighth power}}\label{b8:section}

\begin{theorem}\label{b8:main}
For a nonempty subset $E$ of $\F_2^8$, the following are equivalent:
\begin{enumerate}
\item $E$ is spectral;
\item $E$ tiles $\F_2^8$ by translations;
\item $E$ has a linear-subspace tiling complement.
\end{enumerate}
\end{theorem}

We separate the proof into cardinality, projection, and tiling
arguments. The cardinality bound and the canonical-column reduction
are due to Kadowaki \cite{Kadowaki}. We include their proofs because
they also organize the work in dimension nine.

\subsection{Spectral cardinalities}

Let $S_d=((-1)^{x\cdot y})_{x,y\in\F_2^d}$. For $1\le m\le 2^d$,
let $s_d(m)$ be the sum of the $m$ smallest elements of the multiset
in which $j$ occurs $\binom dj$ times, for $0\le j\le d$.

\begin{lemma}[Kadowaki's Smith-form bound]\label{b8:smith}
If $E\subseteq\F_2^d$ is spectral of size $m\ge2$, then
\begin{equation}\label{b8:smithbound}
 \frac m2\nu_2(m)\ge s_d(m),
\end{equation}
where $\nu_2$ is the $2$-adic valuation.
\end{lemma}
\begin{proof}
The integral Smith form of $S_1$ has diagonal entries $1,2$.
Taking $d$-fold tensor products of its unimodular equivalences shows
that the Smith form of $S_d$ has entries $2^j$ with multiplicity
$\binom dj$. For Smith normal form and determinantal divisors, see
Smith~\cite{Smith} and Newman~\cite{Newman}.
By Cauchy--Binet, every $m\times m$ minor of $S_d$ is
divisible by $2^{s_d(m)}$.

A spectral pair of size $m$ gives an $m\times m$ submatrix $H$ of
$S_d$ with $HH^{\mathsf T}=mI$. Consequently
$(\det H)^2=m^m$, so
$\nu_2(\det H)=\tfrac m2\nu_2(m)$. This proves the inequality.
\end{proof}

\begin{corollary}\label{b8:cardinalities}
Every spectral subset of $\F_2^8$ has power-of-two cardinality.
\end{corollary}
\begin{proof}
A real Hadamard matrix has order $1,2$, or a multiple of four.
Indeed, after normalizing its first row, a second row splits the
columns into equal halves. Orthogonality of a third row to both of
these rows makes that third row balanced on each half. The order is
therefore divisible by four.

By Theorem~\ref{b8:quarter}, only cardinalities at most $64$, and
$128,256$, remain. For the non-power-of-two multiples of four at
most $64$, the two sides of \eqref{b8:smithbound} are as follows:
\begin{center}
\small
\begin{tabular}{c|rrrrrrrrrrr}
$m$&12&20&24&28&36&40&44&48&52&56&60\\ \hline
$\tfrac m2\nu_2(m)$&12&20&36&28&36&60&44&96&52&84&60\\
$s_8(m)$&14&30&38&46&62&73&85&97&109&121&133
\end{tabular}
\end{center}
Each column contradicts \eqref{b8:smithbound}.
\end{proof}

Kadowaki obtained the same cardinality conclusion by checking the
Smith bound up to $128$ \cite[Remark~4.5]{Kadowaki}. The quarter-density
lemma shortens that check, but the cardinality conclusion itself is
already part of his work.

\subsection{Canonical point sets and small Hadamard matrices}

A matrix $L$ over $\F_2$ is log-Hadamard if $((-1)^{L_{ij}})$ is a real
Hadamard matrix, and it is dephased if its first row and first column
are zero. A canonical column set of a dephased matrix $L$ of rank $r$
is the set of columns of a matrix $R$ whose rows form a basis of the
row space of $L$. Changing this basis applies an invertible linear
map to the column set. The columns are distinct, since equal
columns would contradict Hadamard orthogonality.

\begin{lemma}[Kadowaki]\label{b8:canonical}
Let $(E,\Lambda)$ be a spectral pair in $\F_2^d$. If its dephased
log-Hadamard matrix has rank $r$, then $r\le d$, and $E$ projects
bijectively under a surjective linear map onto a canonical column
set in $\F_2^r$.
\end{lemma}
\begin{proof}
Translate both sets to contain zero. The matrix
$L=(\lambda\cdot e)_{\lambda\in\Lambda,e\in E}$ is dephased and has
rank at most $d$. Choose $\lambda_1,\ldots,\lambda_r$ whose rows
form a row-space basis, and put
\[
 q(x)=(\lambda_1\cdot x,\ldots,\lambda_r\cdot x).
\]
The $\lambda_i$ are linearly independent: any relation between them
would give the same relation between their rows. Hence $q$ is
surjective. Its values on $E$ are precisely the distinct columns
of the chosen row-basis matrix.
\end{proof}

\begin{lemma}[Equivalence of canonical point sets]\label{b8:equivalence}
Equivalent real Hadamard matrices have the same dephased rank
over $\F_2$, and their canonical column sets are affinely equivalent.
When equivalence classes are identified with their transposes,
both orientations must be considered.
\end{lemma}
\begin{proof}
Row and column sign changes cancel in dephasing. Permutations
replace the dephasing base by some $(i_0,j_0)$, giving
\[
 L'_{ij}=L_{ij}+L_{ij_0}+L_{i_0j}+L_{i_0j_0}.
\]
Write $L=PR$, with both factors of rank $r$, and let $p_i$ be the
rows of $P$ and $c_j$ the columns of $R$. Since $L$ is dephased,
$p_0=c_0=0$. Thus
\[
 L'_{ij}=(p_i+p_{i_0})\cdot(c_j+c_{j_0}).
\]
The translated $p_i$ span $\F_2^r$, as do the translated $c_j$:
the translated family contains the translation vector and hence
spans the original family. Both new factors therefore have rank
$r$, as does their product. A basis of the new row space differs
from the translated column factor by an invertible change of
coordinates. This proves affine equivalence.

Transposition exchanges the two factors. It preserves the rank
but can exchange inequivalent point sets, which proves the last
assertion. The rank assertion also follows from the dephasing
theorem of Ferguson and Sothanaphan \cite[Proposition~4.3]{FS}; the
point-set assertion is \cite[Lemma~5.3]{Kadowaki}.
\end{proof}

\begin{lemma}\label{b8:h16}
Every spectral set of size $1,2,4,8$, or $16$ in a vector space over $\F_2$
has a linear-subspace tiling complement.
\end{lemma}
\begin{proof}
For orders $2,4,8$, the unique Hadamard equivalence class is the
Sylvester class. Its canonical point set is the whole space of
dimension $1,2,3$, respectively, so Lemmas~\ref{b8:canonical} and
\ref{b8:lifting} apply. The singleton case is immediate.

For order $16$, use Hall's five equivalence classes \cite{Hall}.
For each representative, there are five column indices
$c_0,c_1,c_2,c_3,c_4$ for which the sixteen relative sign patterns
\[
 (H_{i,c_j}H_{i,c_0})_{j=1}^4\qquad(0\le i<16)
\]
are distinct. For the representatives used in Appendix~\ref{b8:appendix},
these five-tuples, in order, are
\[
 (0,1,2,4,8),\quad(0,1,2,8,12),\quad
 (0,1,8,10,12),\quad(0,1,8,10,12),\quad(0,8,9,10,12).
\]
The check is a direct evaluation of sixteen patterns for each
matrix; its code and the five input matrices are included there.
Taking rows as points and columns as frequencies, the four
frequency differences give a linear projection bijective on the
point set. Its kernel is a tiling complement. Hall's classes are
ordinary equivalence classes, so transposes are covered as well.
Lemma~\ref{b8:equivalence} transfers this conclusion to every
matrix in the five classes, and Lemma~\ref{b8:lifting} transfers it
to the original spectral set. This is the subgroup form of
Kadowaki's order-$16$ argument \cite[Proposition~6.1]{Kadowaki}.
\end{proof}

\begin{lemma}[The order-$32$ finite calculation]\label{b8:h32}
Every canonical column set of an order-$32$ Hadamard matrix with
dephased rank over $\F_2$ at most eight has a linear-subspace tiling
complement.
\end{lemma}
\begin{proof}
We use the complete classification of Kharaghani and
Tayfeh-Rezaie \cite{KTR}, in McKay's representation \cite{McKay}.
The input consists of $6{,}857{,}010$ classes up to equivalence
and transposition. Dephase each representative, row-reduce over
$\F_2$, and retain ranks at most eight. The resulting counts are
\begin{center}
\begin{tabular}{c|rrrr|r}
Rank&5&6&7&8&Total\\ \hline
Classes&1&2&13&150&166.
\end{tabular}
\end{center}
These counts, including the $166$ classes, already occur in
\cite[Table~6.2]{Kadowaki}. The additional calculation here
certifies both orientations of all $166$ classes.

For each of the $332$ oriented canonical sets
$A\subseteq\F_2^r$, the supplied certificate is a list of $r-5$
independent vectors spanning a subspace $K$ such that
\begin{equation}\label{b8:directsumcertificate}
 \bigl|\{a+k:a\in A,\ k\in K\}\bigr|=2^r.
\end{equation}
The verifier checks the Hadamard equations, reconstructs the
canonical set from the selected matrix, and checks
\eqref{b8:directsumcertificate} by forming every sum. All $332$
certificates satisfy the equality. Since $|A|=32$ and
$|K|=2^{r-5}$, it proves $A\oplus K=\F_2^r$.

For completeness, the search for a certificate enumerates
every subspace of dimension $r-5$ by its unique reduced
row-echelon basis. It accepts a subspace precisely when its
nonzero elements avoid $A+A$. The full-source extraction and
the independent extraction described in Appendix~\ref{b8:appendix}
ensure that the selected records cover every class allowed by
the rank bound. Lemma~\ref{b8:equivalence} then covers every
order-$32$ matrix in the assertion.
\end{proof}

\begin{corollary}\label{b8:spectraltosubgroup}
Every spectral subset of $\F_2^8$ has a linear-subspace tiling
complement.
\end{corollary}
\begin{proof}
Its size is a power of two by Corollary~\ref{b8:cardinalities}.
Sizes at most $16$ follow from Lemma~\ref{b8:h16}. Size $32$
follows from Lemmas~\ref{b8:canonical}, \ref{b8:h32}, and
\ref{b8:lifting}. Theorem~\ref{b8:quarter} gives sizes $64$ and
$128$, and the whole-group case is immediate.
\end{proof}

\subsection{Sixteen-point tiles and perfect codes}

The remaining tiling case is naturally expressed in terms of
perfect codes. We do not first assume that the tile is a graph:
the graph property will follow from the quotient certificates.

\begin{lemma}\label{b8:tile16}
Every sixteen-point tile of affine dimension at most eight in a
vector space over $\F_2$ has a linear-subspace tiling complement.
\end{lemma}
\begin{proof}
Translate the tile $E$ to contain zero, and put
$V=\langle E\rangle$, $r=\dim V\le8$. If $E\oplus T=G$, then
$E\oplus(T\cap V)=V$: in the unique representation of a point
of $V$, the summand from $E$ lies in $V$, so the other summand
does also. Set $T_V=T\cap V$.

Let $H:\F_2^{15}\to V$ have as its columns the fifteen nonzero
points of $E$, and put $S=\{0,e_1,\ldots,e_{15}\}$. The map $H$
is surjective and sends $S$ bijectively onto $E$. We claim that
\begin{equation}\label{b8:perfectlift}
 S\oplus C=\F_2^{15},\qquad C=H^{-1}(T_V).
\end{equation}
Indeed, the unique decomposition of $Hx$ in $E\oplus T_V$
specifies a unique $s\in S$ with $x+s\in C$. Thus $C$ is a
perfect binary one-error-correcting code of length $15$: its
radius-one Hamming balls partition the space. It has $2048$
words, and distinct words are at distance at least three.
Translate $T_V$, if necessary, so that $0\in C$.

Let $K_0=\ker H$ and let
\[
 M=\{u\in\F_2^{15}:C+u=C\}
\]
be the translation kernel of $C$. It is a linear subspace,
$K_0\subseteq M$, and
$\dim M\ge\dim K_0=15-r\ge7$. Since $0\in C$, we have
$M\subseteq C$; its nonzero vectors therefore have weight at
least three. In particular, the quotient map is injective on
$S$. Under the identification $V=\F_2^{15}/K_0$, the induced
map
\begin{equation}\label{b8:kernelquotient}
 V\longrightarrow\F_2^{15}/M
\end{equation}
is injective on $E$ and sends $E$ onto $S/M$.

The classification of \"Osterg\aa rd and Pottonen \cite{OP}
contains $5983$ equivalence classes of perfect length-$15$
codes. Its kernel distribution, also recorded in \cite{OPP},
has exactly the following classes with kernel dimension at
least seven:
\begin{center}
\begin{tabular}{c|rrrr|r}
$\dim M$&7&8&9&11&Total\\ \hline
Classes&48&17&3&1&69.
\end{tabular}
\end{center}
For each of these $69$ codes, the supplied certificate gives
a subgroup complement of $S/M$ in $\F_2^{15}/M$. The verifier
checks all $32768$ radius-one-ball sums, computes the kernel
by testing its translations, constructs the quotient map
from a basis of $M^\perp$, and checks every sum in the
certified quotient tiling. The complement dimension is
$11-\dim M$ in every case.

These certificates cover the code in \eqref{b8:perfectlift}.
Indeed, translating a code leaves its kernel unchanged,
and permuting coordinates sends its kernel to the permuted
kernel while preserving $S$. Equivalent codes therefore
give linearly equivalent quotient point sets $S/M$.
Lift the subgroup certificate through
\eqref{b8:kernelquotient} by Lemma~\ref{b8:lifting}. It gives
a subgroup complement of $E$ in $V$. Adding any linear
complement of $V$ in $G$ gives one in $G$ as well.
\end{proof}

\subsection{Small tiling partners}

The following construction is Kadowaki's universal-spectrum
argument \cite[Theorem~3.1 and Lemma~3.3]{Kadowaki}. We give
the proof to indicate exactly which part of his dimension-seven
result continues to apply in higher dimensions.

\begin{lemma}\label{b8:universal}
If $T\subseteq\F_2^d$ has $1,2,4$, or $8$ points, there is a
set $B$ of size $2^d/|T|$ with $B+B\subseteq\supp W_T$.
Consequently every tiling complement of $T$ is spectral.
\end{lemma}
\begin{proof}
Suppose $|T|=8$. Translate $T$ to contain zero and write
$T=\{0,t_1,\ldots,t_7\}$, $V=\langle T\rangle$,
$r=\dim V$. The evaluation map
\[
 \phi:V^*\longrightarrow\F_2^7,\qquad
 \phi(\alpha)=(\alpha(t_1),\ldots,\alpha(t_7))
\]
is injective. Let $U$ be its image. For $\alpha\in V^*$,
$W_T(\alpha)=8-2\operatorname{wt}(\phi(\alpha))$.
Put $S_7=\{0,e_1,\ldots,e_7\}$ and $j=(1,\ldots,1)$.
The sixteen-point set $D=S_7\cup(j+S_7)$ has no difference
of weight four: within either part the nonzero weights are
one or two, and between parts they are five, six, or seven.

There are $2^{7-r}$ cosets of $U$, so some coset contains
at least $16/2^{7-r}=2^{r-3}$ points of $D$. Choose that many
and translate them into $U$. Their differences still avoid
weight four. Apply $\phi^{-1}$ and take the full preimage
under restriction $G^*\to V^*$. This gives $2^{d-3}$
frequencies with the asserted nonvanishing property.
Differences whose restriction is zero have Walsh value $8$.

If $|T|=4$, its normalized span has dimension two or
three. In dimension two, $T=V$ and the singleton intrinsic
frequency set suffices. In dimension three, $T$ is an
affine basis; use zero and the functional taking value one
on its three nonzero points. Their Walsh difference is
$4-2\cdot3=-2$. Lift by restriction as above. For sizes one
and two the intrinsic spectrum-support set is again a
singleton, and restriction gives the required cardinality.
The last assertion is Lemma~\ref{b8:certificate}(2).
\end{proof}

\begin{proof}[Proof of Theorem~\ref{b8:main}]
Corollary~\ref{b8:spectraltosubgroup} proves that spectral
sets have subgroup complements. Lemma~\ref{b8:lifting}
shows that sets with subgroup complements are spectral,
and they are plainly tiles.

It remains to show that every tile is spectral. Its size
divides $256$. Tiles over $\F_2$ of sizes at most eight have
subgroup complements by the small-tile theorem of
Cohen, Litsyn, Vardy, and Z\'emor \cite{CLVZ}. A
sixteen-point tile has one by Lemma~\ref{b8:tile16}.
For sizes $32,64,128,256$, the tiling partner has at most
eight points, so Lemma~\ref{b8:universal} applies.
\end{proof}

We may now describe the extension of Kadowaki's result precisely.
His Smith-form argument excludes all non-power-of-two spectral
cardinalities in dimension eight, his order-$16$ argument treats
all sixteen-point spectral sets, and his universal spectra treat
all tiles of size at least $32$. His rank scan also identifies
the $166$ relevant extended order-$32$ classes. Our additions are
the $332$ subgroup certificates for those classes, the perfect-code
kernel reduction and its $69$ quotient certificates for sixteen-point
tiles, and the quarter-density lemma for sixty-four-point spectral
sets. Together they settle exactly the cases left in
\cite[Remark~7.2]{Kadowaki}.


\section{\texorpdfstring{Spectral cardinalities in $\F_2^9$}{Spectral cardinalities in F2 to the ninth power}}\label{b9:cardinalities}
We now consider dimension nine. The quarter-density argument still settles
spectral sets of size $128$ and larger. Below quarter density, the
Smith bound of Lemma~\ref{b8:smith} leaves just two possible cardinalities
which are not powers of two, namely $48$ and $96$. We exclude both.
The exclusion of $96$ is elementary. The exclusion of $48$ uses the
classification of triply even codes and a finite collection of positive
obstruction certificates.

\begin{theorem}\label{b9:main}
Every spectral subset of $\F_2^9$ has power-of-two cardinality. Every
spectral subset whose cardinality is different from $64$ has a linear
subspace tiling complement. Every translational tile in $\F_2^9$ is spectral.
\end{theorem}
The last assertion is proved in Section~\ref{b9:tiles}. Additional sufficient
conditions for a spectral set of size $64$ to tile are proved in
Section~\ref{b9:size64}. The general $64$-point spectral-to-tiling implication
is not asserted.

\begin{lemma}\label{b9:smithremaining}
If a spectral subset of $\F_2^9$ does not have power-of-two cardinality,
its cardinality is $48$ or $96$.
\end{lemma}
\begin{proof}
A real Hadamard matrix has order $1$, $2$, or a multiple of four:
after normalizing its first row, orthogonality forces half the entries
of each other row to be negative; comparing two such rows gives four
equal sign classes. Theorem~\ref{b8:quarter} reduces the order to at most
$128$, apart from the powers $256$ and $512$.
For $m\le128$, the $m$ smallest Smith exponents of the dimension-nine
Walsh matrix have sum
\[
 s_9(m)=\begin{cases}
 m-1,&1\le m\le10,\\
 2m-11,&10\le m\le46,\\
 3m-57,&46\le m\le128.
 \end{cases}
\]
Substitution of $m=4,8,\ldots,128$ in
$(m/2)\nu_2(m)\ge s_9(m)$ leaves only powers of two and $48,96$.
This is the Smith bound with multiplicities $1,9,36,84$ for exponents
$0,1,2,3$; thus the substitution is a finite integer inequality, not a
matrix-classification assumption.
\end{proof}

\subsection{Smith factors and quadratic evaluation}
Let $H$ be a real Hadamard matrix of order $n$, with Smith factors
$s_1\mid\cdots\mid s_n$. Since $H^{-1}=H^T/n$, comparison of the
Smith forms gives the standard invariant-factor relation
$s_is_{n+1-i}=n$; see Newman~\cite[p.~127]{Newman1971}.
Write $\alpha_i=\nu_2(s_i)$.
For a dephased phase matrix $L$ over $\F_2$, deleting its zero row and column
gives $L'$ and integer row and column operations give
\[
 H\sim 1\oplus(-2L').
\]
Thus the multiplicities of exponents zero and one are $1$ and
$r=\operatorname{rank}_{\F_2}L$. In orders $48$ and $96$ the multiplicities
are respectively
\[
 (1,r,46-2r,r,1),\qquad(1,r,47-r,47-r,r,1).
\]
Here only the powers of two in the Smith factors are displayed. Write $c_2$ for the multiplicity of exponent two.

\begin{lemma}\label{b9:96}
There is no spectral set of cardinality $96$ in $\F_2^9$.
\end{lemma}
\begin{proof}
For vectors $x,y$ over $\F_2$ the parity expansion gives
\[
 (-1)^{x\cdot y}\equiv
 1-2\sum_{i=1}^9x_iy_i+4\sum_{i<j}x_ix_jy_iy_j\pmod8.
\]
Thus every submatrix of the Walsh matrix factors, over $\Z/8\Z$,
through $1+9+\binom92=46$ features. It has at most $46$ nonzero Smith
components modulo eight: put the left factor in Smith form, leaving at
most $46$ nonzero rows, and multiply by the right factor.
For a Hadamard matrix of order $96$, however,
$\alpha_i+\alpha_{97-i}=5$. Exactly one exponent in each of the $48$
pairs is at most two, so there are $48$ nonzero components modulo eight.
This is a contradiction.
\end{proof}

For a set $X\subset\F_2^r$ let $h_X(2)$ be the dimension of the
restrictions to $X$ of Boolean polynomials of degree at most two.
When $r=9$, put $d_X=46-h_X(2)$. All polynomial spaces and code ranks
in the following argument are over $\F_2$.

\subsection{The quadratic Schur factorization}
Factor the dephased matrix over $\F_2$ as
\[
 L_{x,y}=x\cdot y\pmod2,\qquad x\in X\subseteq\F_2^r,\ y\in Y\subseteq\F_2^r.
\]
For a spectral pair in $\F_2^9$, one has $r\le9$. Both canonical sets have
$n$ distinct points, contain zero, and span $\F_2^r$. Select $r$ basis
points of $X$ and change dual coordinates so they become $e_1,\ldots,e_r$.
Choose $r$ basis points of $Y$. Write
\[
 B=[y_1\ \cdots\ y_r],
\]
an invertible matrix over $\F_2$. It is invertible over the ring of rationals
with odd denominator as well. Order the nonzero points of each set with
these basis points first. The integer $0/1$ matrix $L'$ then has block form
\[
 L'=\begin{pmatrix}B&Y_*^T\\ C&D\end{pmatrix}.
\]
Its Schur complement is $S=D-CB^{-1}Y_*^T$. It is entrywise even, since
$L'$ has rank $r$ over $\F_2$. Eliminating $B$ shows
\begin{equation}\label{b9:quad:smith}
 c_2=\rank_2\bigl((S/2)\bmod2\bigr).                 
\end{equation}

Let $q=\binom r2$ and $Q(x)=(x_ix_j)_{i<j}\in\F_2^q$. Lifting the dot product over $\F_2$ to the integers gives
\begin{equation}\label{b9:quad:parity}
 L_{x,y}\equiv\sum_i x_i y_i
       -2\sum_{i<j}x_ix_jy_iy_j\pmod4.             
\end{equation}
Let $Q_X$ be the quadratic-feature matrix on the nonbasis points of $X$,
$Q_Y$ the same matrix on nonbasis points of $Y$, and $Q_B$ the quadratic
features on the $r$ chosen basis points of $Y$. Equation~\eqref{b9:quad:parity} gives
\begin{equation}\label{b9:quad:schur}
 \frac S2\equiv
 -Q_X\underbrace{\bigl(Q_Y^T-Q_B^TB^{-1}Y_*^T\bigr)}_{R_Y}
 \pmod2.                                          
\end{equation}
The sign disappears modulo two. All inverses in~\eqref{b9:quad:schur} may therefore be
taken over $\F_2$.

The constant and linear evaluation columns have rank $r+1$ on either
set. Quadratic features vanish on the zero and unit-vector basis of $X$.
On $Y$, the subtraction in $R_Y$ removes the unique linear interpolant
at zero and the selected basis points. Therefore, with $h_X(2),h_Y(2)$
defined in dimension $r$,
\[
 u:=\rank Q_X=h_X(2)-r-1,\qquad
 v:=\rank R_Y=h_Y(2)-r-1.
\]
Combining~\eqref{b9:quad:smith}, \eqref{b9:quad:schur}, and Sylvester's rank inequality yields
\begin{equation}\label{b9:quad:key}
 u+v-q\ \le\ c_2\ \le\ \min(u,v)\ \le\ q.
\end{equation}
This is the central restriction; it is stronger than merely summing
Smith exponents to estimate a determinant.

\subsection{Deductions for 48 and 96}
For order 48, (\ref{b9:quad:key}) implies
\[
 46-2r\le\binom r2.
\]
For $r\le8$ this fails (at eight, $30>28$). Hence in dimension nine one
must have $r=9$, and then $q=36$ and $c_2=28$. Both actual sets span
dimension nine, since their bilinear phase matrix has rank nine.
Now $u=36-d_E$, $v=36-d_\Lambda$, so (\ref{b9:quad:key}) becomes
\[
 36-d_E-d_\Lambda\le28\le36-d_E,\ 36-d_\Lambda.
\]
We have proved $0\le d_E,d_\Lambda\le8$ and $d_E+d_\Lambda\ge8$.

For order 96, the rightmost inequality in (\ref{b9:quad:key}) gives
\[
 47-r\le\binom r2,
 \quad\text{or}\quad r(r+1)/2\ge47.
\]
This requires $r\ge10$, excluding every ambient dimension at most nine.
This also gives a second proof of Lemma~\ref{b9:96}.

\subsection{Two geometric obstructions}
\begin{lemma}[Odd quotient obstruction]\label{b9:cubic:flat}
If a spectral set $E$ of size $m$ properly contains an affine flat
$a+V$ of size $k$, then $m/k$ is an even integer.
\end{lemma}
\begin{proof}
Let $\Lambda$ be a spectrum. Orthogonality between points of $a+V$
gives $\sum_{\lambda\in\Lambda}(-1)^{\lambda\cdot v}=0$ for all
$0\ne v\in V$. Fourier inversion on $V$ says that every restriction
fiber $\Lambda_b=\{\lambda:\lambda|_V=b\}$ has size $m/k$.
Choose $x\in E\setminus(a+V)$. Orthogonality of $x$ against every
$a+v$ gives
\[
 0=\sum_{b\in V^*}(-1)^{b(v)}
       \sum_{\lambda\in\Lambda_b}(-1)^{\lambda\cdot(x+a)}
 \qquad(v\in V).
\]
Fourier inversion again makes each inner sum zero. A zero sum of
$m/k$ signs requires $m/k$ even.
\end{proof}
Thus a spectral set of size 48 cannot contain an affine four-flat:
its size would be 16 and its quotient would be 3.

\begin{lemma}\label{b9:cubic:period}
A 48-point spectral set in $\F_2^9$ has no nonzero translation period.
\end{lemma}
\begin{proof}
If $E+p=E$, quotienting by $\langle p\rangle$ produces a 24-point set
in $\F_2^8$. Orthogonality between $x$ and $x+p$ splits its spectrum
into two restriction fibers of size 24. Either fiber, after translating
one frequency to zero, supplies a spectrum of the quotient set:
the original orthogonality sums are twice the quotient sums.

But a normalized order-24 Hadamard matrix has phase rank 11 over $\F_2$.
Indeed its 2-primary Smith exponents pair to sum $v_2(24)=3$.
The multiplicities are $(1,r,r,1)$, since normalization gives
$H\sim1\oplus(-2L')$ and $r=\rank_{\F_2} L'$. Counting 24 factors gives
$2+2r=24$. A phase matrix factored through $\F_2^8$ has rank at most
eight, a contradiction.
\end{proof}

\subsection{The cubic moment and its quadratic derivatives}
For a full-span set $E$ in a spectral pair define
\[
 T_E(u,v,w)=\sum_{x\in E}(u\cdot x)(v\cdot x)(w\cdot x)\in\F_2.
\]
Let $C_E$ be the length-48 binary code formed by evaluating all affine
functions on $E$. It has dimension ten and contains the all-ones vector.
Choose nine linearly independent nonzero frequencies in $\Lambda$.
Their evaluation words have weight 24, and every pair intersects in
12 positions, because their sum also has weight 24. Inclusion-exclusion
shows that their span is doubly even; adding the all-ones word preserves
this property. A doubly even binary code is self-orthogonal, since
$\wt(c+d)=\wt(c)+\wt(d)-2\wt(c*d)$.
Consequently $T_E(u,u,v)=0$, so $T_E$ is an alternating trilinear form.

Write $\rho_E$ for the rank of its contraction map
$u\mapsto T_E(u,\cdot,\cdot)$, a map from a nine-dimensional space
to the 36-dimensional space of alternating bilinear forms.

\begin{lemma}\label{b9:cubic:deriv}
For a spectral pair of size 48,
$\rho_E\le d_\Lambda$ and $\rho_\Lambda\le d_E$.
Also $T_E=0$ if and only if $C_E$ is triply even.
\end{lemma}
\begin{proof}
Use the chosen frequency basis as coordinate units. For $s\in\F_2^9$ let
\[
 P_E(s)=\frac{\wt((s\cdot x)_{x\in E})}{4}\pmod2.
\]
The numerator is divisible by four. Expanding the integer weight of a
sum of binary words by inclusion-exclusion modulo eight, the linear
coefficients of $P_E$ are $24/4=6=0\pmod2$, the quadratic coefficients
are $-12/2=-6=0\pmod2$, and the cubic coefficients are triple
intersection parities. Higher terms vanish. Hence
\[
 P_E(s)=\sum_{i<j<k}T_E(e_i,e_j,e_k)s_i s_j s_k.
\]
For every $\lambda\in\Lambda$, both $P_E(\lambda)$ and
$P_E(\lambda+e_i)$ vanish: the corresponding word has weight zero
or 24. Thus $P_E(s+e_i)+P_E(s)$ is a homogeneous quadratic vanishing
on $\Lambda$. Its coefficient vector is exactly the contraction of
$T_E$ with $e_i$. The span of these nine quadratics has dimension
$\rho_E$, which proves the inequality. Reverse the pair for the other.

If $T_E=0$, then $P_E=0$, so every linear evaluation word has weight
divisible by eight. Its complement has weight $48-\wt(c)$, also
divisible by eight, so $C_E$ is triply even. Conversely if all its
weights are divisible by eight, the zero function $P_E$ has zero
cubic coefficients, giving $T_E=0$.
\end{proof}

\subsection{Excluding the zero cubic by the length-48 classification}
\begin{lemma}\label{b9:cubic:nonzero}
Neither $C_E$ nor $C_\Lambda$ is triply even.
\end{lemma}
\begin{proof}
Every triply even code extends to a maximal one. Theorem 29 and
Example 12 of~\cite{BM} give ten maximal codes of length 48: seven
ordinary extended doublings of indecomposable doubly even self-dual
length-24 codes; one $\mathrm{RM}(1,4)\oplus D(d_{16}^+)$; one
$\mathrm{RM}(1,4)^{\oplus3}$; and one exceptional code of dimension nine.

The dimension-nine code cannot contain $C_E$, of dimension ten.
An ordinary doubling is generated by $(c,c)$ and $(\mathbf1,0)$.
On any subcode, the two coordinate functionals in each paired position
differ by the same functional $p$. Choose a generator basis of $C_E$
starting with its all-ones word. Its 48 columns have form $(1,x)$,
$x\in E$. The functional $p$ has zero first coordinate. If $p=0$,
paired columns coincide, contrary to the distinct points of $E$;
otherwise the paired columns say $E+p=E$, contrary to
Lemma~\ref{b9:cubic:period}.

In either remaining maximal code, select its 16-coordinate
$\mathrm{RM}(1,4)$ component. Restricting the generators of $C_E$ to
these coordinates gives affine functions on $\F_2^4$. Hence the
corresponding 16 columns $(1,x)$ are the image of an affine map from
$\F_2^4$. Distinctness of the points of $E$ makes this map injective.
Its image therefore gives an affine four-flat inside $E$, contrary to
Lemma~\ref{b9:cubic:flat}. Interchange the two sides for $C_\Lambda$.
\end{proof}

It follows that $3\le d_E,d_\Lambda\le8$. Indeed a nonzero
alternating trilinear form has contraction rank at least three, since it
descends to a nonzero alternating three-form on its radical quotient.
We next use both tensors together to obtain a stronger restriction.

\subsection{Pairing of the two cubic moment tensors}
Let $V=\mathbb F_2^r$, and let $X\subseteq V$ and $Y\subseteq V^*$
be a spectral pair of order $n$, containing zero and spanning their
respective ambient spaces. Assume $16\mid n$. Their cubic moments are
the alternating tensors
\[
 T_X\in\bigwedge^3 V,\qquad T_Y\in\bigwedge^3 V^*.
\]
Thus $T_X(u,v,w)=\sum_{x\in X}u(x)v(x)w(x)$, for $u,v,w\in V^*$,
and similarly for $T_Y$. Write $\rho_X,\rho_Y$ for their contraction
ranks. The alternating assertion follows from the doubly even evaluation
codes, exactly as in the preceding cubic-moment lemma.

\begin{lemma}[Orthogonality of the contraction images]\label{b9:pair:pairorth}
Under the natural perfect pairing between $\bigwedge^2V$ and
$\bigwedge^2V^*$, one has
\[
 \langle \iota_uT_X,\iota_vT_Y\rangle=0
 \qquad(u\in V^*,\ v\in V).
\]
\end{lemma}
\begin{proof}
Choose dual coordinate bases. All scalar evaluations inside the integer
sums below are lifted to $\{0,1\}$. Put $L_{xy}=x\cdot y\pmod2$,
$a_x=u(x)$, $b_y=v(y)$, and write
\[
 F=\sum_{x\in X,y\in Y}a_xb_y
             \sum_{i<j}x_ix_jy_iy_j\in\mathbb Z.
\]
The parity of $F$ is the desired tensor pairing. The elementary parity
expansion gives
\[
 2F\equiv
 \sum_i\left(\sum_xa_xx_i\right)
       \left(\sum_yb_yy_i\right)
       -\sum_{x,y}a_xb_yL_{xy}\pmod4.
\]
Every factor in the first sum is even: it is the intersection parity
of two words in a self-orthogonal evaluation code. Thus the first sum
is divisible by four.

It is enough to treat $u=y_0\in Y$ and $v=x_0\in X$, because those
sets span and the tensor pairing is bilinear in $(u,v)$. If either is
zero, the assertion is immediate. Otherwise $a_x=L_{x,y_0}$ and
$b_y=L_{x_0,y}$. Remove the zero row and column of $L$, obtaining $M$.
Its row and column weights are $n/2$, and distinct row or column
intersections are $n/4$. Therefore
\[
 M^TM=\frac n4(J+I),\qquad
 M^TMM^T=\frac{n^2}{8}J+\frac n4M^T.
\]
The last weighted sum is the $(y_0,x_0)$ entry of $M^TMM^T$.
Both coefficients are divisible by four when $16\mid n$. Hence
$2F=0\pmod4$, so $F=0\pmod2$.
\end{proof}

For a set containing zero and of full affine span, let $K_X$ be the
space of quadratic principal parts of its vanishing polynomials of
degree at most two. It lies in $\bigwedge^2V^*$ and has dimension $d_X$.
Likewise $K_Y\subseteq\bigwedge^2V$ has dimension $d_Y$. The principal
part map is injective here because there is no nonzero vanishing affine
polynomial. Under an invertible linear coordinate change, the quadratic
principal part transforms by the appropriate exterior-square map:
terms with repeated indices only change the linear part of a Boolean
polynomial. Thus the pairing of $K_X$ and $K_Y$ is coordinate invariant.

\begin{lemma}[Exact rank of the relation pairing]\label{b9:pair:relationpair}
Let $q=\binom r2$, and let $c_2$ be the multiplicity of 2-primary Smith
exponent two of the corresponding Hadamard matrix. The natural pairing
$K_X\times K_Y\longrightarrow\mathbb F_2$ has rank
\[
 d_X+d_Y-q+c_2.
\]
\end{lemma}
\begin{proof}
Normalize an actual point basis of $X$ to the coordinate units.
In the notation of the preceding quadratic Schur-complement argument,
\[
 c_2=\rank(Q_XR_Y),\quad
 \ker Q_X=K_X,\quad \operatorname{im}R_Y=K_Y^\perp.
\]
For the last equality, a proposed quadratic coefficient vector defines
a vanishing polynomial on $Y$ precisely when its evaluations on the
nonbasis points agree with the unique linear interpolant through the
chosen point basis; this is exactly $\ker R_Y^T=K_Y$.
Since $\dim\operatorname{im}R_Y=q-d_Y$, rank-nullity yields
\[
 \dim(K_X\cap K_Y^\perp)=q-d_Y-c_2.
\]
This intersection is the left radical of the pairing in question.
Subtracting its dimension from $d_X$ proves the formula.
\end{proof}

\begin{theorem}[Joint cubic-rank bound]\label{b9:pair:rhobound}
In the above situation,
\[
 \rho_X+\rho_Y\le q-c_2.
\]
Consequently, for a hypothetical order-48/rank-nine spectral pair,
\[
 \rho_X+\rho_Y\le8.
\]
Combining with the previously proved nonzero cubic condition gives,
up to transposition, only $(\rho_X,\rho_Y)=(3,3)$ or $(3,5)$.
In particular at least one side always has cubic contraction rank three.
\end{theorem}
\begin{proof}
The quadratic-derivative lemma gives
$\operatorname{im}(\iota T_Y)\subseteq K_X$ and
$\operatorname{im}(\iota T_X)\subseteq K_Y$.
This formulation uses principal parts, so the two inclusions remain
valid when the separate point-basis normalizations in that lemma are
undone. By Lemma~\ref{b9:pair:pairorth}, the restriction of the relation pairing
to these two subspaces is zero. Choose bases starting with their
$\rho_Y$ and $\rho_X$ independent elements, respectively. Its matrix
has a zero $\rho_Y\times\rho_X$ block. Its rank is consequently at most
\[
 (d_X-\rho_Y)+(d_Y-\rho_X).
\]
Compare with Lemma~\ref{b9:pair:relationpair} and cancel $d_X+d_Y$.
For order 48 and rank nine, $q=36$ and $c_2=28$.

Each nonzero alternating trilinear form has contraction rank at least
three. Rank four is impossible: the induced form on its radical quotient
would be a nondegenerate alternating 3-form in dimension four; but every
nonzero 3-form in dimension four is decomposable and has a one-dimensional
radical. Thus the only allowed ordered ranks with positive entries and
sum at most eight are $(3,3)$, $(3,5)$ and $(5,3)$.
\end{proof}

All statements below are over $\F_2$. Let $(E,\Lambda)$ be a hypothetical
spectral pair of size 48 in $\F_2^9$. We use the already proved properties
that the affine evaluation codes have dimension ten and are doubly even,
that $3\le d_E,d_\Lambda\le8$, and that the cubic moment tensor $T_E$
has contraction rank at most $d_\Lambda$. Here $d_A=46-h_A(2)$.

\begin{theorem}\label{b9:rank3:hyperplane}
If the contraction rank of $T_E$ is three, then the affine evaluation
code $C_E$ contains a triply even subcode of dimension nine containing
the all-ones word. If additionally $d_\Lambda=3$, then in a basis drawn
from $\Lambda$, the weight cubic $\wt((s\cdot x)_{x\in E})/4\pmod2$
is a product of three independent linear forms with disjoint supports.
\end{theorem}

\subsection{The common derivative zero set}\label{b9:rank3:gradientsection}
Choose a frequency basis $b_1,\ldots,b_9$ from $\Lambda$ and use it as
coordinate units. The earlier cubic identity gives the pure squarefree
cubic
\[
 P(s)=\sum_{i<j<k}T_E(b_i,b_j,b_k)s_i s_j s_k.
\]
Every $\lambda\in\Lambda$ is a common zero of its nine first Boolean
derivatives $\partial_iP(s)=P(s+e_i)+P(s)$.

Since $T_E$ has contraction rank three, its radical quotient has
dimension three and carries the unique nonzero alternating volume form.
Thus there is a surjection $A:\F_2^9\to\F_2^3$ such that
\[
 T_E(u,v,w)=\det(Au,Av,Aw).
\]
Put $p_i=Ae_i$, let $\mathcal P$ be the set of distinct nonzero values
among these nine columns, and put $m=|\mathcal P|$. The columns span
$\F_2^3$, so $3\le m\le7$. Write
\[
 q(w)=(w_2w_3,w_1w_3,w_1w_2),\qquad
 Bs=\sum_i q(p_i)s_i.
\]
Expansion of $q(As)$ shows
\begin{equation}\label{b9:rank3:gradient}
 \partial_iP(s)=p_i\cdot\bigl(q(As)+Bs\bigr).
\end{equation}
Indeed the diagonal terms of $q(As)$ are exactly $Bs$, and the
coefficient of $s_js_k$ after pairing with $p_i$ is the indicated
determinant. Because the $p_i$ span, the common derivative zero set is
\[
 S=\{s:Bs=q(As)\}.
\]
For each $p\in\mathcal P$ introduce the group-parity variable
$g_p=\sum_{i:p_i=p}s_i$. The map $s\mapsto(g_p)_{p\in\mathcal P}$ is
surjective and has kernel dimension $9-m$. Consequently $S$ is linearly
equivalent to $K_{\mathcal P}\times\F_2^{9-m}$, where
\[
 K_{\mathcal P}=\left\{g:
   \sum_{p\in\mathcal P}q(p)g_p
    =q\left(\sum_{p\in\mathcal P}pg_p\right)\right\}.
\]

\begin{lemma}\label{b9:rank3:circuit}
The seven vectors $(p,q(p))$, $0\ne p\in\F_2^3$, have rank six and
their unique linear dependence is their sum. In particular every six
are independent.
\end{lemma}
\begin{proof}
Their sum is zero: each linear coordinate takes the value one four
times, and each squarefree quadratic takes the value one twice.
The six columns indexed by $e_1,e_2,e_3,e_1+e_2,e_1+e_3,e_2+e_3$
are independent by first eliminating the three quadratic coordinates,
then the three linear coordinates. Hence the full kernel is generated
by the all-ones vector, which has no zero coordinates.
\end{proof}
For all seven labels, the solutions to
$\sum_p(p,q(p))g_p=(w,q(w))$ are the singleton $g=e_w$ (or $g=0$
when $w=0$) and its complement. Restricting to $\mathcal P$ means
that all omitted coordinates of $g$ must vanish. Therefore, in their
natural $m$ coordinates,
\[
 K_{\mathcal P}=\begin{cases}
 \{0,e_1,\ldots,e_m\},&3\le m\le5,\\
 \{0,e_1,\ldots,e_6,\mathbf1\},&m=6,\\
 \{0,e_1,\ldots,e_7,\mathbf1,
       \mathbf1+e_1,\ldots,\mathbf1+e_7\},&m=7.
 \end{cases}
\]
No assumption about the particular spanning subset of labels is needed
for this zero-set description.

\subsection{The quadratic defects}
For any finite set $K$ and $k\ge0$, separating the free variables into
degrees zero, one, and two gives
\[
 h_{K\times\F_2^k}(2)=h_K(2)+k\,h_K(1)+\binom{k}{2}.
\]
These spaces are independent by their distinct free-variable monomials.
For $m\le5$, the set $K$ is an affine basis, so both $h_K(1)$ and
$h_K(2)$ equal $m+1$. For $m=6$ the respective values are seven and
eight. For $m=7$ they are eight and fifteen. For the last assertion,
the affine functions give dimension eight. Quadratic monomials vanish
at $0,e_1,\ldots,e_7$ and their values at $\mathbf1,\mathbf1+e_i$
are $(\sigma,\sigma+\deg_i)$, where $\sigma$ and $\deg_i$ are the
edge sum and vertex degrees of the corresponding edge labeling of
$K_7$. The degree map has rank six and an odd triangle makes the edge
sum independent of that map. Thus the quadratics add seven dimensions.

We obtain the entirely explicit table
\begin{center}\begin{tabular}{rrrr}
\toprule
$m$&$|S|$&$h_S(2)$&$d_S=46-h_S(2)$\\\midrule
3&256&43&3\\
4&160&40&6\\
5&96&36&10\\
6&64&32&14\\
7&64&32&14\\\bottomrule
\end{tabular}\end{center}
Since $\Lambda\subseteq S$, every quadratic vanishing on $S$ vanishes
on $\Lambda$, so $d_\Lambda\ge d_S$. Thus $d_\Lambda\le8$ forces
$m\le4$, and $d_\Lambda=3$ forces $m=3$.

\subsection{A linear factor and a triply even hyperplane}
The pure cubic depends only on the group-parity variables $g_p$:
triples with repeated labels have determinant zero, while for three
distinct labels the coordinate sums factor as the product of the three
group parities. If $m=3$, the labels are independent, and hence
$P=g_1g_2g_3$.

If $m=4$, choose three independent labels as $e_1,e_2,e_3$. Up to
permuting these, the remaining label is either $e_1+e_2$ or
$e_1+e_2+e_3$. The respective cubics are
\[
 g_3(g_1g_2+g_1g_4+g_2g_4)
\]
and
\[
 g_1g_2g_3+g_1g_2g_4+g_1g_3g_4+g_2g_3g_4
  =(g_1+g_2+g_3+g_4)(g_1g_2+g_3g_4).
\]
The equality is in the Boolean polynomial ring, $g_i^2=g_i$.
In every allowed case $P$ vanishes on a linear hyperplane $H$ of
$\F_2^9$. Its eight-dimensional linear evaluation code has all weights
divisible by eight. Adding the all-ones word yields a triply even
dimension-nine subcode of $C_E$. The all-ones word is independent
because $0\in E$. This proves the assertion.

In fact the two cases give precisely two linear-equivalence types for
the weight cubic:
\[
 P=xyz\qquad\hbox{or}\qquad P=x(yz+w),
\]
with the displayed variables independent. For the first $m=4$ case,
use $ab+ac+bc=(a+c)(b+c)+c$. For the second case set
$x=g_1+g_2+g_3+g_4$; direct expansion gives
$P=x(g_1g_2+g_1g_3+g_2g_3)$, followed by the same substitution.
The respective numbers of frequencies $s\in\F_2^9$ with $P(s)=1$
are 64 and 128. The first type has four hyperplanes on which $P$
vanishes (the three factor hyperplanes and $x+y+z=0$); the second
type has only the hyperplane $x=0$.

\subsection{A projective triply even hyperplane is unavoidable}
Call a binary code containing the all-ones word projective when its
generator columns are distinct. For an affine evaluation code this is
equivalent to its point set having no repetitions.

\begin{theorem}\label{b9:rank3:projective}
If $T_E$ has contraction rank three, then $C_E$ contains a projective
triply even dimension-nine subcode containing the all-ones word.
\end{theorem}
\begin{proof}
Consider a hyperplane $H$ on which the weight cubic vanishes, and the
corresponding triply even subcode $D_H$. If $D_H$ is not projective,
two points of $E$ restrict identically to $H$, so their difference is
the unique nonzero vector $p\in H^\perp$. Orthogonality of these
points against the spectrum says
\[
 \sum_{\lambda\in\Lambda}(-1)^{\lambda\cdot p}=0,
 \quad\hbox{hence}\quad
 |\Lambda\cap H|=|\Lambda\setminus H|=24.
\]

In the $m=4$ case the common gradient zero set is an affine basis
$\{0,e_1,e_2,e_3,e_4\}$ times $\F_2^5$. In the three-circuit case the
factor hyperplane is $g_3=0$; its complement meets this zero set in
just the coset indexed by $e_3$. In the four-circuit case the factor
hyperplane is $g_1+g_2+g_3+g_4=0$; it meets the zero set in just the
coset indexed by zero. Thus one of the two 24-point halves of
$\Lambda$ would lie in a single affine five-flat. A 24-point subset
of a 32-point affine five-flat has the entire direction five-space
as its difference set, by the pigeonhole principle. Orthogonality
then makes $E$ uniformly distributed on the 32 characters of this
five-space, impossible because $32\nmid48$. Therefore this
triply even subcode is projective.

In the $m=3$ case the four factor hyperplanes have equations
$g_1=0,g_2=0,g_3=0,g_1+g_2+g_3=0$. The gradient zero set consists
of four parallel six-flats indexed by $0,e_1,e_2,e_3$. Balancing
across each factor hyperplane forces a different one of these four
fibers to have size 24. All four fibers are nonempty, since
$\Lambda$ has full affine span. Thus at most one factor hyperplane
can be balanced: two fibers of size 24 would leave the other two
empty. At least three of the four subcodes are consequently
projective.
\end{proof}

\subsection{Finite extension constraints}
Suppose $D=\langle\mathbf1,h_1,\ldots,h_8\rangle$ is such a triply
even dimension-nine subcode and $C=D+\langle g\rangle$. The code $C$
is doubly even precisely when $g\in D^\perp$ and $\wt(g)\equiv0\pmod4$.
Define the alternating bilinear form on $D/\langle\mathbf1\rangle$
by
\[
 B_g(i,j)=\sum_{a=1}^{48}g_a(h_i)_a(h_j)_a.
\]
Alternation follows from $g\perp D$. On $C/\langle\mathbf1\rangle$,
the cubic moment is $\epsilon\wedge B_g$, where $\epsilon(g)=1$ and
$\epsilon(D)=0$. If $B_g\ne0$, its contraction rank is
$1+\rank B_g$: a radical vector must have zero $g$ coefficient and
its component in $D/\langle\mathbf1\rangle$ must belong to $\rad B_g$.
Thus rank-three cubic moment is equivalent to $\rank B_g=2$.

There are $\binom{8}{2}_2=(255\cdot254)/6=10,795$ such forms.
For each one, its entries impose linear equations on the 48 bits of
$g$. Coordinate pairs with equal columns in $D$ must have opposite
$g$ bits, and a column multiplicity exceeding two makes projective
extension impossible. We use $C*D$ for the linear span of componentwise
products of codewords; see~\cite{Randriam} for this formalism. Moreover
\[
 h_C(2)=\dim(C*C)\le\dim(D*D)+9.
\]
Here $C*C=D*D+gD$ and $g^2=g$ is included because $\mathbf1\in D$.
The bound $d_E\le8$ therefore forces $\dim(D*D)\ge29$.

A direct scan of the 1,960 dimension-nine triply even codes in~\cite{BMdata}
found 519 with column multiplicities at most two. Only 112 also satisfy
$\dim(D*D)\ge29$: their square-code dimension counts are
$54,31,18,5,3,1$ in dimensions $29,30,31,32,33,34$, respectively.
By Theorem~\ref{b9:rank3:projective} it suffices instead to use the 26 projective
base codes with square-code dimension at least 29: the counts in
dimensions $29,30,31,32,33,34$ are $7,6,9,2,1,1$, respectively.
These counts are checks of the published list, not a reproof of its
completeness. For fixed $B_g$, the solution kernel contains $D$, so
enumerating $g$ modulo $D$ removes nine redundant bits. Before the
extra duplicate-column equations, there are at most
$2^{48-29-9}=1024$ extensions per feasible prescribed bilinear form.
The exhaustive test of these extension families is given next.

\subsection{A positive obstruction certificate}
\begin{lemma}\label{b9:finite:five-obstruction}
There is no 48-point spectral set $E\subseteq\F_2^9$ for which
$E+E$ contains a five-dimensional linear subspace.
\end{lemma}
\begin{proof}
If $U\subseteq E+E$ has dimension five, orthogonality gives
$\sum_{\lambda\in\Lambda}(-1)^{u\cdot\lambda}=0$ for every
$0\ne u\in U$. Fourier inversion on $U$ makes every restriction
class of $\Lambda$ in $U^*$ have cardinality $48/32$, impossible.
\end{proof}
The certificate is just five vectors: check that their 32 subset sums
are distinct and all belong to $E+E$. No negative clique-search result
is needed for this obstruction.

\subsection{Complete enumeration and independent checks}
For each of the 26 classified base codes, the program lists all
10,795 nonzero alternating rank-two bilinear forms on an
eight-dimensional space. It prescribes their 28 upper-triangle entries
as the parities $g\cdot(h_i*h_j)$, together with the nine equations
$g\perp D$. Exact Gaussian elimination over $\F_2$ gives either inconsistency or an
affine solution space. Its homogeneous kernel is $(D*D)^\perp$ and
contains $D$. Choosing a complement to $D$ in that kernel enumerates
one representative of every extension word modulo $D$. Adding a word
of $D$ leaves the extended code unchanged and preserves all required
conditions. A nonzero prescribed bilinear form ensures $g\notin D$.
On column points, $g\mapsto g+d$ is an invertible affine shearing of
the last coordinate by an affine function of the first eight. Thus it
preserves spectrality and the five-space obstruction as well.

Across the 26 bases there are 26,582 feasible prescribed forms and
10,294,816 extension representatives. The only rejection filters in
the final pass are $\wt(g)\equiv0\pmod4$ and
$38\le\dim(C*C)\le43$, the latter being exactly $3\le d_E\le8$.
Every retained code has 48 distinct columns because its subcode $D$
is already projective. Extracting the column points $E$ requires
only reading the nine generator bits after the all-ones row.
\begin{center}\begin{tabular}{lr}
\toprule
Stage & Complete count\\\midrule
Projective triply even starting codes & 26\\
Extension representatives modulo $D$ & 10,294,816\\
$\wt(g)$ divisible by four & 5,748,704\\
Quadratic Hilbert dimension between 38 and 43 & 490,816\\
Explicit five-space obstruction certificates & 490,816\\
Unresolved retained candidates & 0\\\bottomrule
\end{tabular}\end{center}
The separate certificate verifier reconstructs every 48-point set,
forms its difference set directly, and checks every one of the 32 span
vectors of its recorded basis. It also requires exact one-to-one
coverage of the retained candidate file, with no duplicate or extra
certificates. All 490,816 pass. The subspace finder may stop at a node
cap, but an unresolved candidate is retained rather than rejected;
in the recorded complete run there are no such candidates.

A second family generator starts directly from the published matrices,
reverses the generator order, uses low-pivot full row reduction instead
of the first implementation's high-pivot elimination, and constructs
the quotient in free-parameter coordinates. An independent enumerator
reproduces the complete family and the retained-code counts. Code
cosets are compared after canonical reduction modulo their base code,
so the comparison is independent of the chosen particular solutions
and quotient bases. The resulting sets of all 490,816 retained cosets
agree exactly. Source, logs, input checksums, and positive certificates
are included in the accompanying code archive; see Appendix~\ref{b9:appendix}.

\begin{corollary}
No spectral subset of $\F_2^9$ has cardinality 48.
\end{corollary}
\begin{proof}
The projective-subcode theorem and its extension constraints place one side of every hypothetical pair
in the complete extension parametrization. Its code must pass the two
algebraic filters. The exhaustive certificate check places a
five-dimensional subspace in its difference set, contradicting
Lemma~\ref{b9:finite:five-obstruction}.
\end{proof}

\subsection{The remaining power-of-two sizes}
The proof of Lemma~\ref{b8:canonical} applies in dimension nine without
change. The complete order-$32$ Hadamard list~\cite{KTR,McKay} has $2,040$ classes of
dephased rank over $\F_2$ at most nine. Its rank counts are
\[
\begin{array}{c|rrrrr}
r&5&6&7&8&9\\\hline
\text{classes}&1&2&13&150&1,874.
\end{array}
\]
Both orientations give $4,080$ canonical sets. Each has an explicitly
checked subgroup complement. Together with the order-$16$ calculation,
the small orders, and faithful lifting, this proves that every spectral
set of cardinality at most $32$ in $\F_2^9$ has a subgroup complement.
The endpoints $128,256,512$ follow from Theorem~\ref{b8:quarter}.
Lemmas~\ref{b9:smithremaining} and~\ref{b9:96} and the order-$48$
exclusion therefore prove the first two assertions of Theorem~\ref{b9:main}.
As in dimension eight, the finite assertion is about complements of every
canonical model, with both matrix orientations included; it is not an
assumption that the classification itself lists embedded point sets.

\section{\texorpdfstring{Every tile in $\F_2^9$ is spectral}{Every tile in F2 to the ninth power is spectral}}\label{b9:tiles}
Only tiles of size $32$ require a new argument beyond the small-partner
lemma and the perfect-code construction. We first extend the latter
construction by one dimension.

\subsection{Sixteen-point tiles}\label{b9:tile:sixteen}
\begin{proposition}\label{b9:tile:tile16}
Every sixteen-point tile of affine dimension at most nine is a linear tile.
\end{proposition}
\begin{proof}
Normalize the tile $T$ to contain zero and restrict a tiling to its span
$V=\F_2^r$, $r\le9$. Let $H:\F_2^{15}\to V$ have columns $T\setminus\{0\}$.
If $T\oplus A=V$, the inverse image $C=H^{-1}(A)$ is a perfect length-15
code: the radius-one ball $S=\{0,e_1,\ldots,e_{15}\}$ tiles by $C$.
Let $M=\{m:C+m=C\}$. It contains $\ker H$, of dimension $15-r\ge6$.
After translation $0\in C$, so $M\subseteq C$ and $M$ contains no nonzero
word of weight one or two. Consequently $S/M$ has 16 distinct points.

Of the 5,983 codes classified in~\cite{OP}, exactly 374 have kernel dimension at least
six. The counts in dimensions $6,7,8,9,11$ are respectively
$305,48,17,3,1$. Each of the 374 configurations $S/M$ has an explicit
subgroup complement, checked by enumerating all sums. Pull it back along
$V=\F_2^{15}/\ker H\to\F_2^{15}/M$; the map is injective on $T$, so this is
a subgroup complement of $T$. Coordinate permutations preserve $S$ and
code translations preserve $M$, establishing coverage up to the published
equivalence. These 374 certificates, source hashes, the complete extraction,
and an independent Fourier computation of all code kernels are in the accompanying code archive.
\end{proof}

\subsection{The 32-point case: structural reduction}
Let $E\oplus T=G=\F_2^9$, $|E|=32$ and $|T|=16$, with both sets normalized
to contain zero. If $E$ has affine span of dimension at most eight,
restrict the tiling to that span. Its partner has size at most eight,
so Lemma~\ref{b8:universal} applies; a spectrum extends to $G$.
Assume henceforth that $E$ spans $G$.
By \cite{OV}, $T$ has affine span $V$ of dimension $r\le8$.

If $r=4$, $T=V$ and a singleton is a support certificate in $V^*$.
If $r=5$, a nonzero Fourier coefficient gives a two-point certificate.
If $r=6$, Lemma~\ref{b8:supportplane} gives a four-point certificate.
These are the sizes $2^r/16$ required in Lemma~\ref{b8:certificate}; their
full inverse images in $G^*$ all have size 32. It remains to consider
$r=7$ and $r=8$.

\subsection{Partners of affine rank seven}
By Proposition~\ref{b9:tile:tile16}, $T\subseteq\F_2^7$ has a three-dimensional
subgroup complement. In suitable coordinates it is a graph
\[
 T_f=\{(x,f(x)):x\in\F_2^4\},\qquad f:\F_2^4\to\F_2^3.
\]
Adding an affine function to $f$ is an affine change of ambient coordinates.
The space of Boolean functions on $\F_2^4$ modulo affine functions has
dimension $16-5=11$. Full affine rank seven is equivalent to independence
of the three component functions in this quotient. Changing their basis
does not change the equivalence class of the graph. Thus all relevant
graphs are covered by three-dimensional subspaces of an eleven-dimensional
space, acted on by $\operatorname{AGL}(4,2)$.

There are exactly
\[
 \genfrac{[}{]}{0pt}{}{11}{3}_2=50,955,971
\]
such subspaces. Exhaustive orbit enumeration gives 494 orbits. For every
representative, an eight-point set $B$ satisfying Lemma~\ref{b8:certificate}
was found and directly verified. Two complete orbit enumerations agree
on all 494 orbit sizes: the first uses algebraic normal form and lowest
pivots; the second uses truth tables normalized modulo affine functions,
highest pivots, and different generators. Both cover all 50,955,971
subspaces, and each orbit of the second calculation contains exactly one
representative from the first. This proves the required common-spectrum
property for every rank-seven partner $T$, and lifting gives a spectrum
of size 32 for $E$.

\subsection{Partners of affine rank eight}
Now $V=\F_2^8$ and
\[
 E=A\ \cup\ (v+B),\qquad v\notin V,
 \qquad A\oplus T=B\oplus T=V,\qquad |A|=|B|=16.
\]
Individual spectrality of $A$ and $B$ does not by itself supply a common
spectrum. The following calculations address that issue explicitly.

\subsection{Coverage by nonlinear perfect-code kernels}
Suppose at least one fiber, say $A$, is not an affine four-flat.
Let $P=\{p:A+p=A\}$. Then $\dim P\le2$: if its dimension were three,
$A$ would be two cosets of an eight-point subgroup and hence an affine
four-flat; dimension four also gives an affine four-flat.
With $H:\F_2^{15}\to V$ as above and $C=H^{-1}(A)$, one has
\[
 \ker(C)=H^{-1}(P),\qquad \dim\ker(C)=7+\dim P\in\{7,8,9\}.
\]
In the complete classification, 48, 17, and 3 codes respectively have
these kernel dimensions. For each representative kernel $M$, enumerate
\emph{every} seven-dimensional subspace $K\subseteq M$ and form $S/K$.
This covers every possible $T$, since the actual $\ker H$ is such a $K$.
The equivalence issue is handled exactly as in Proposition~\ref{b9:tile:tile16}:
coordinate permutations preserve $S$, translations preserve kernels, and
all $K$ inside the representative kernel are included.

The numbers of configurations are
\[
 48+17\genfrac{[}{]}{0pt}{}{8}{7}_2+3\genfrac{[}{]}{0pt}{}{9}{7}_2
 =48+17\cdot255+3\cdot43,435=134,688.
\]
The exhaustive computation gives:
\begin{center}
\begin{tabular}{lrrr}
\toprule
Kernel dimension & Configurations & Support certificate & Exceptional\\
\midrule
7 & 48 & 46 & 2\\
8 or 9 & 134,640 & 134,352 & 288\\
\bottomrule
\end{tabular}
\end{center}
Each positive certificate is a 16-point set $B\subseteq V^*$ whose
differences all have nonzero $W_T$. These lift to spectra of size 32
in $G$ by Lemma~\ref{b8:certificate}.

All 290 exceptional configurations are affinely equivalent to the single set
\[
 T_* = \{0,1,2,4,7,8,13,16,32,64,123,126,128,144,160,192\},
\]
where integers encode vectors over $\F_2$ with the least significant bit first.
The two exact-kernel exceptions are dataset records 686 and 4753; record
4753 is $T_*$. The linear map sending the coordinate vectors to
$1,2,8,16,4,128,32,192$ maps the former to the latter. Explicit invertible
affine maps are supplied and directly checked for all 288 other exceptions.
An absent support certificate is only a failure of this sufficient condition,
not a claim that a fiber pair is nonspectral.

\subsection{Every pair of exceptional fibers}
Exact-cover enumeration finds precisely 12,096 normalized sets $A$ with
$0\in A$ and $A\oplus T_*=\F_2^8$. An independent Python enumeration,
using integer masks and reversed branching order, recovers exactly the
same sets. Translation of a complement changes neither its difference
set nor its spectra, so normalization loses no relevant cases.
These complements have exactly 1,788 different sets
$(A+A)\setminus\{0\}$.

For every unordered pair of profiles, including equal profiles, the
calculation supplies a four-dimensional subgroup disjoint from both
nonzero difference sets. There are
\[
 \frac{1,788\cdot1,789}{2}=1,599,366
\]
pairs. A collection of just 75 subgroup certificates covers all pairs;
their dimensions and disjointness are checked directly, separately from
the subgroup search. Hence every pair of fibers $A,B$ has a common
subgroup complement $L\subseteq V$. Then $E\oplus L=G$, and
$L^\perp$ is a spectrum of $E$.

\subsection{Both fibers are affine flats}
If $A$ and $B$ are affine four-flats, their direction subspaces $U,W$
have a common four-dimensional complement in $V$, regardless of $T$.
For completeness, put $C=U\cap W$, $\dim C=c$, and choose
\[
 U=C\oplus U_0,\quad W=C\oplus W_0,\quad
 V=C\oplus U_0\oplus W_0\oplus Z,\quad\dim Z=c.
\]
Pair bases $u_i,w_i$ of $U_0,W_0$ and $c_j,z_j$ of $C,Z$.
The space spanned by the $u_i+w_i$ and $c_j+z_j$ has dimension four
and intersects both $U$ and $W$ trivially. It is a common tiling
complement of the affine flats, hence also a complement of $E$ in $G$.
This completes the 32-point case.

\begin{proof}[Completion of the proof of Theorem~\ref{b9:main}]
A tile in $\F_2^9$ has power-of-two size. Sizes at most eight are linear
tiles by~\cite{CLVZ}; size $16$ is covered by the preceding perfect-code
argument; and the preceding sections settle size $32$. A tile of size
at least $64$ has a partner of size at most eight and is spectral by
Lemma~\ref{b8:universal}.
\end{proof}

\section{\texorpdfstring{Further results for spectral sets of size $64$}{Further results for spectral sets of size 64}}\label{b9:size64}
Theorem~\ref{b9:main} leaves one cardinality. We give several sufficient
conditions for a $64$-point spectral set to tile, and then explain why a
direct eighth-density extension of the support-plane lemma fails.
First, if either member of a spectral pair has affine span at most eight,
restriction to that span and injective restriction of the opposite
characters reduce to the quarter-density or half-density argument.
Both members then have subgroup complements. Hence a remaining pair
may be assumed to span nine dimensions on both sides.

\begin{lemma}\label{b9:periodicset}
If a $64$-point spectral set $E\subset\F_2^9$ has a nonzero period, then
it has a three-dimensional subgroup tiling complement.
\end{lemma}
\begin{proof}
Quotient by a period line. Orthogonality splits a spectrum into its two
restriction fibers, each of size $32$, and either fiber supplies a spectrum
of the $32$-point quotient in $\F_2^8$, exactly as in the proof of
Lemma~\ref{b9:cubic:period}. That quotient has a three-dimensional
subgroup complement by the dimension-eight theorem. Lift this complement
to a subspace transverse to the period line. It complements $E$.
\end{proof}

\subsection{A common complement in two hyperplanes}
\begin{theorem}\label{b9:64:common}
Let $n\ge3$. Suppose $A,B\subset H=\F_2^n$ each have size $2^{n-3}$,
have a common spectrum $L$, and each has proper affine span.
Then they have a common three-dimensional subgroup tiling complement.
\end{theorem}
\begin{proof}
For any real function $h$ on a vector space over $\F_2$, write $W_h(t)=\sum_w h(w)(-1)^{t\cdot w}$ for its unnormalized Walsh transform. Translate the sets separately into linear hyperplanes $U,V$. Translations
do not change their spectra or difference sets. Put $M=2^{n-3}$.
If $U=V$, restriction of $L$ to $U^*$ is injective and has density $1/4$.
Lemma~\ref{b8:supportplane} gives a support plane $P\subset U$ avoiding both
nonzero difference sets. Adjoin any vector outside $U$.

Now suppose $U\ne V$. Write $U=p^\perp$, $V=q^\perp$ and
$W=U\cap V$, so $|W|=2M$. Each fiber of $L\to W^*$ has zero, one or
two points: injectivity of restriction to $U^*$ and $V^*$ prohibits
differences $p$ and $q$, and a double fiber must differ by $p+q$.
Let $m(w)$ be its size and $s(w)=\mathbf1_{m(w)=1}$. Thus $\sum m=M$.
Choose $a\in U\setminus W$, $b\in V\setminus W$, and define
\[
 f(w)=\sum_{\lambda|_W=w}(-1)^{a\cdot\lambda},\qquad
 g(w)=\sum_{\lambda|_W=w}(-1)^{b\cdot\lambda}.
\]
Double fibers cancel, giving $f^2=g^2=s$. Moreover
\[
 W_m(t)=W_L(t),\quad
 W_f(t)=W_L(a+t),\quad
 W_g(t)=W_L(b+t)\qquad(t\in W).
\]
If $|s|>0$, Parseval yields
\[
 \sum_{\ell\ne0}W_m(\ell)W_s(\ell)
 =|W|\sum_w m(w)s(w)-M|s|=M|s|>0.
\]
Choose a nonzero $\ell$ with both factors nonzero. Since $f^2=g^2=s$,
\[
 \sum_tW_f(t)W_f(t+\ell)
 =\sum_tW_g(t)W_g(t+\ell)
 =|W|W_s(\ell)\ne0.
\]
Choose $t,t'$ giving nonzero products and set
$u=a+t$, $v=b+t'$, $K=\langle\ell,u,v\rangle$.
These vectors are independent. The planes $K\cap U$ and $K\cap V$
belong to $\supp W_L$. The remaining two points
$u+v,u+v+\ell$ lie outside $U\cup V$, hence outside both difference sets.
Thus $K$ is a common complement.

If $s=0$, then $m=2\mathbf1_D$ with $D$ of density $1/4$ in $W^*$.
Choose a support plane $P\subset W$ by Lemma~\ref{b8:supportplane} and
$c\notin U\cup V$. Then $K=P+\langle c\rangle$ works: its points in
$W$ satisfy the support condition and all others lie outside both
hyperplanes. If $n=3$, the case $s=0$ cannot occur, since $M=1$.
\end{proof}

\subsection{Periodic spectra: the finite part is small}
\begin{theorem}\label{b9:64:periodic}
If $(E,\Lambda)$ is a 64-point spectral pair in $\F_2^9$ and
$\Lambda+p=\Lambda$ for some $p\ne0$, then $E$ tiles $\F_2^9$.
\end{theorem}
\begin{proof}
Choose dual coordinates with $p=e_9$. Then $\Lambda=L\times\F_2$,
$|L|=32$. Write $E=(A\times\{0\})\cup(B\times\{1\})$.
Orthogonality at the last coordinate gives $|A|=|B|=32$.
Orthogonality for $(l,0),(l',0)$ and $(l,0),(l',1)$ gives respectively
the sum and difference of the character inner products on $A$ and $B$;
thus $L$ is a common spectrum for them in $\F_2^8$.

If both fibers have proper affine span, Theorem~\ref{b9:64:common} applies.
Otherwise suppose $A$ spans $\F_2^8$ after translation. Translate $L$ to
contain zero. The phase matrix $(a\cdot l)$ over $\F_2$ has rank
$r=\dim\langle L\rangle$: evaluation on the spanning set $A$ is injective
on linear functionals. Hence $L$ itself, up to affine equivalence and
embedding, is one of the canonical order-32 models of rank at most eight.
There are 166 extracted classes and 332 oriented models, with rank counts
\[
\begin{array}{c|rrrr}
r&5&6&7&8\\\hline
\text{classes}&1&2&13&150\\
\text{oriented models}&2&4&26&300.
\end{array}
\]
For every model $L_0\subset\F_2^r$, the supplied finite certificate is a set
$K_0$ of size $2^{r-5}$ satisfying
\[
 K_0+K_0\subset\supp W_{L_0}.
\]
All 332 certificates pass direct integer checks. There are 304 subgroup
certificates and 28 general certificates. Embed $L_0$ as $L_0\times\{0\}$
and use $K_0\times\F_2^{8-r}$; this is an eight-point common complement.
Transport it by the inverse-dual coordinate transformation. In either
case a common complement $K$ of $A,B$ has been found, and $K\times\{0\}$
complements $E$.
\end{proof}
The full-span distinction is essential. When both fibers have smaller
span, the actual spectrum could be a nonlinear graph lift of a canonical
model. Theorem~\ref{b9:64:common} handles that case without assuming that such
lifts are affinely equivalent to the canonical list.

\subsection{Odd unions of parallel planes can be projected away}
\begin{theorem}\label{b9:64:section}
Let $(E,\Lambda)$ be a spectral pair of cardinality $N$ in $\F_2^d$, with
$16\mid N$ and $\Lambda$ affinely spanning the dual space. Suppose an affine
hyperplane section $B$ of $E$ is invariant under a two-dimensional
translation subspace $V$ and $|B|\equiv4\pmod8$.
Then both members project injectively to a spectral pair in dimension
$d-1$.
\end{theorem}
\begin{proof}
Translate a point of $E\setminus B$ to zero and translate $\Lambda$ to
contain zero. Write $B=\{x\in E:u\cdot x=1\}$ and
$V=\langle b,c\rangle$. Its invariance implies $u\cdot b=u\cdot c=0$.
Choose coordinate units $e_i\in\Lambda$. Their evaluation words have
weight $N/2$ and distinct pairs intersect in $N/4$ positions.
Their span over $\F_2$ is therefore doubly even and self-orthogonal.
Inclusion--exclusion modulo eight gives
\[
 q(s)=\wt((s\cdot x)_{x\in E})/4\pmod2
     =\sum_{i<j<k}T_{ijk}s_i s_j s_k,\qquad
 T_{ijk}=\sum_{x\in E}x_i x_j x_k\pmod2.
\]
The linear and quadratic terms vanish because $16\mid N$.
The homogeneous quadratic
\[
 Q_u(s)=\sum_i u_i\{q(s+e_i)+q(s)\}
\]
vanishes on $\Lambda$: the words corresponding to $\lambda$ and
$\lambda+e_i$ have weight zero or $N/2$. Its polar form is
\[
 Q_u(s+t)+Q_u(s)+Q_u(t)=\sum_{x\in B}(s\cdot x)(t\cdot x).
\]
Each coset of $V$ contributes
$(s\cdot b)(t\cdot c)+(s\cdot c)(t\cdot b)$. There are an odd number of
cosets, so for some $\ell\in\F_2^d$,
\[
 Q_u(s)=(s\cdot b)(s\cdot c)+s\cdot\ell.
\]
Euler's identity gives $Q_u(u)=q(u)=1$ (each cubic monomial is counted
three times), hence $u\cdot\ell=1$. On the spectrum this yields
\begin{equation}\label{b9:64:sectionquadratic}
 \lambda\cdot\ell=(\lambda\cdot b)(\lambda\cdot c).
\end{equation}
Replace $B$ by $B+\ell$, or equivalently apply
$\pi(x)=x+(u\cdot x)\ell$. For $\delta=\lambda+\mu$ the change in its
character inner product is
\[
 -2\mathbf1_{\delta\cdot\ell=1}W_B(\delta).
\]
Invariance under $V$ makes this Fourier coefficient zero outside $V^\perp$.
Inside $V^\perp$, the $b,c$ evaluations of $\lambda,\mu$ coincide, so
\eqref{b9:64:sectionquadratic} gives $\delta\cdot\ell=0$. Every inner product is preserved.
The resulting square matrix is Hadamard; its point columns and frequency
rows cannot repeat. Thus $\pi$ is injective on $E$, and restriction of
$\Lambda$ to $u^\perp$ is injective as well. They give the asserted
spectral pair in dimension $d-1$.
\end{proof}

\begin{corollary}\label{b9:64:four}
For a 64-point spectral pair in $\F_2^9$, if either member has such a section,
both members tile by three-dimensional subgroups. In particular a
four-point affine hyperplane section on either side suffices.
\end{corollary}
\begin{proof}
The projected pair has quarter density in $\F_2^8$. Apply
Lemma~\ref{b8:supportplane} to each member and lift the plane complements through
the one-dimensional kernels. Cases with smaller affine spans were settled at the start of this section. For a four-point section, self-orthogonality implies
$\sum_{x\in B}x=0$; four distinct points over $\F_2$ with sum zero form an
affine plane, so Theorem~\ref{b9:64:section} applies.
\end{proof}
This uses actual translation periods of $B$. A rank-two cubic contraction
alone does not provide the Fourier-support condition required in the proof.

\subsection{An affine-flat section projection criterion}
\begin{proposition}\label{b9:flat:flatprojection}
Let $(E,\Lambda)$ be a spectral pair in $\mathbb F_2^d$, let $u\ne0$, and
suppose one hyperplane section
\[
 B=\{x\in E:u\cdot x=\epsilon\}=a+V
\]
is an affine flat. Define
\[
 W=\operatorname{span}\{\lambda+\mu:\lambda,\mu\in\Lambda,
          \ (\lambda+\mu)|_V=0\}.
\]
If $u\notin W$, then both members of the pair admit injective linear
projections to spectral sets in dimension $d-1$.

\end{proposition}
\begin{proof}
Choose $\ell\in W^\perp$ with $u\cdot\ell=1$. Replace the points
of $B$ by their translates by $\ell$. For $\delta=\lambda+\mu$, the change
in the character inner product is
\[
 -2\mathbf1_{\delta\cdot\ell=1}
       \sum_{x\in B}(-1)^{\delta\cdot x}.
\]
The affine-flat sum vanishes unless $\delta\in V^\perp$; in that case
$\delta\in W$, so $\delta\cdot\ell=0$. Thus every inner product is
preserved. The square matrix remains Hadamard, so neither its point
columns nor its frequency rows can repeat. The changed set lies entirely
in the opposite affine hyperplane, and the frequencies restrict
injectively to its direction hyperplane.
The point map is the affine projection
\[
 x\longmapsto x+(u\cdot x+1+\epsilon)\ell,
\]
whose linear part has kernel $\langle\ell\rangle$; the frequency map is
the quotient by $\langle u\rangle$. Translation of the target hyperplane
removes the affine constant. Therefore both original sets are graphs over
spectral sets in one lower dimension.\end{proof}

For $|E|=64$ and $d=9$, the quarter-density lemma in dimension eight then
provides three-dimensional subgroup complements of both original sets.
The criterion is a short exact linear-algebra check for each affine-flat
hyperplane section. It does not assert that every spectral pair has such
a section, or that $u\notin W$ is automatic.

\subsection{Cubic rank three and its label sets}\label{b9:64labels}
Let $(E,\Lambda)$ be a full-span spectral pair of size $64$, translated
to contain zero. Its affine evaluation codes are doubly even, by the
same weight and intersection argument used above. In coordinates drawn
from a frequency basis in $\Lambda$, the weight function
\[
 q_E(s)=\operatorname{wt}_E(s)/4\pmod2
\]
is the pure cubic whose coefficients are $T_E(e_i,e_j,e_k)$, and its
first Boolean derivatives vanish on $\Lambda$. The proof is identical
to that for order $48$: the coordinate weights are now $32$ and the
pairwise intersections are $16$, so the linear and quadratic coefficients
again vanish modulo two.

Suppose the contraction rank of $T_E$ is three. The derivation in
Section~\ref{b9:rank3:gradientsection} uses only these cubic and derivative
identities. It therefore gives
\[
 \Lambda\subset S\cong K_m\times\F_2^{9-m},\qquad 3\le m\le7,
\]
where
\[
 K_m=\begin{cases}
 \{0,e_1,\ldots,e_m\},&3\le m\le5,\\
 \{0,e_1,\ldots,e_6,\mathbf1\},&m=6,\\
 \{0,e_1,\ldots,e_7,\mathbf1,\mathbf1+e_1,\ldots,\mathbf1+e_7\},&m=7.
 \end{cases}
\]
The sizes of $S$ are respectively $256,160,96,64,64$. For $m\le5$,
all $m+1$ label fibers of $\Lambda$ are nonempty, because the labels are
an affine basis and $\Lambda$ spans. Each fiber size is divisible by four:
an affine function isolates that label, and the affine evaluation code is
doubly even. After applying the four-point projection corollary wherever
possible, every nonempty affine hyperplane section has at least eight
points. In particular all these fiber sizes are at least eight.

\paragraph{The $m=6$ and $m=7$ branches tile directly.}
In these cases $|S|=|\Lambda|=64$, so $\Lambda=S$. Its free factor makes
it periodic, but no separate periodicity theorem is needed. For $m=6$,
\[
 W_{K_6}(u)=1+6-2\operatorname{wt}(u)+(-1)^{\operatorname{wt}(u)}
\]
is nonzero at precisely weights $0,1,2,5,6$. All differences of the
eight-point set $K_6$ have these weights. Thus $T=K_6\times\{0\}$ has
$T+T$ inside the nonzero Fourier support of $\Lambda$.
For $m=7$,
\[
 W_{K_7}(u)=
 (1+(-1)^{\operatorname{wt}(u)})(8-2\operatorname{wt}(u))
\]
is nonzero at precisely weights $0,2,6$. Set
\[
 T_7=\{0,e_1+e_2,\ldots,e_1+e_7,\mathbf1+e_1\}.
\]
Its seven nonzero elements, and all their nonzero pairwise differences,
have weight two or six. Again $T=T_7\times\{0\}$ has eight elements
and $T+T$ lies in the nonzero Fourier support of $\Lambda$.
Spectrality therefore gives $(E+E)\cap(T+T)=\{0\}$, and
$|E||T|=512$ proves that $E$ tiles. These complements need not be subgroups.

\subsection{A finite weight-only reduction inside the rank-three label model}
Suppose the label reduction puts one member of a
$64$-point spectral pair in the form
\[
 \Lambda=\bigcup_{i=0}^m\{a_i\}\times A_i
 \subset\mathbb F_2^m\times\mathbb F_2^{9-m},\qquad
 (a_0,a_1,\ldots,a_m)=(0,e_1,\ldots,e_m).
\]
Write $n_i=|A_i|$. Suppose every fiber is nonempty, $4\mid n_i$,
and both affine evaluation codes have minimum nonzero weight at least
$8$. This implies $n_i\ge8$, since every affine-basis label can
be isolated by an affine functional. Also $n_i\le2^{9-m}$ and
$\sum_i n_i=64$.

For a frequency $(u,0)$ we have
\[
 W_\Lambda(u,0)=n_0+\sum_{i=1}^m n_i(-1)^{u_i}.
\]
If a three-dimensional subspace $K\subset\mathbb F_2^m$ has this
quantity nonzero for every $u\in K$, then spectrality implies
$(E+E)\cap(K\times\{0\})=\{0\}$. Thus $E$ tiles by this subgroup.
This conclusion is independent of the detailed fibers $A_i$.

Permuting affine-basis labels is an affine automorphism, and translating
$\Lambda$ only multiplies Fourier values by signs. We can therefore sort
$(n_0,\ldots,n_m)$ without losing any case. Enumerating the integer
partitions under the three displayed constraints gives:
\[
\begin{array}{c|rrr}
 m&\text{patterns}&\text{positive subgroup certificates}&\text{remaining}\\\hline
 3&15&7&8\\
 4&10&9&1\\
 5&3&3&0
\end{array}
\]
The short standard-library script \texttt{check\_rank3\_fiber\_weights.py}
enumerates every pattern and every candidate subspace in reduced echelon
form. Its JSON output records the positive basis and the eight nonzero
Fourier values on the complement. Its negative outcomes mean only that no complement of
this particular label-space form exists, not that the set does not tile.

For $m=5$, the three sorted patterns and positive bases, in binary-mask
notation, are
\[
\begin{array}{c|c}
 (8,8,8,8,16,16)&(1,2,4)\\
 (8,8,8,12,12,16)&(1,2,4)\\
 (8,8,12,12,12,12)&(1,26,4).
\end{array}
\]
Thus the $m=5$ branch is completely settled under the stated label-model
assumptions. For $m=4$ only $(8,8,16,16,16)$ remains. Its Fourier-zero
label directions form precisely the three nonzero points of a plane.
For $m=3$ the remaining sorted patterns are
\[
\begin{gathered}
(8,8,16,32),\ (8,8,24,24),\ (8,12,12,32),\ (8,12,20,24),\\
(8,16,16,24),\ (12,12,20,20),\ (12,16,16,20),\ (16,16,16,16).
\end{gathered}
\]
This check does not establish the rank-three label reduction itself, nor
cover cubic tensors of other contraction ranks.

\subsection{Closing the remaining five-fiber pattern}
\begin{lemma}[Three halfsets]\label{b9:halfsets:threehalfsets}
Let $H$ be a finite-dimensional vector space over $\F_2$, and let $C_1,C_2,C_3\subset H$
each have size $|H|/2$. Define
\[
 D_{ij}=\{z\in H:C_i+z=H\setminus C_j\},\qquad D_0=D_{11}\cap D_{22}\cap D_{33}.
\]
In a two-dimensional space over $\F_2$ choose independent vectors $a_{12},a_{13}$
and put $a_{23}=a_{12}+a_{13}$. Then
\[
 U=(\{0\}\times D_0)\ \cup\ \bigcup_{1\le i<j\le3}(\{a_{ij}\}\times D_{ij})
\]
is contained in a linear-functional level set $\ell=1$. In particular,
$0\notin U$.

\end{lemma}\begin{proof} Write $R_i=\{z:C_i+z=C_i\}$. A nonempty $D_{ij}$ is a coset
of $R_i=R_j$. Form a graph on the three indices, with edge $ij$ when
$D_{ij}\ne\varnothing$.
If this graph is connected, write $C_i$ as $C+t_i$, complemented if
$\varepsilon_i=1$, for one fixed halfset $C$ and $\varepsilon_i\in\mathbb F_2$.
All periods equal $R=\operatorname{Per}(C)$. The set
$D=\{z:C+z=H\setminus C\}$ is either empty or a coset $d+R$ with $d\notin R$.
If $D=d+R$, choose a functional $h$ on $H$ vanishing on $R$ with $h(d)=1$.
For $z\in D_{ij}$,
\[
 h(z)=h(t_i+t_j)+1+\varepsilon_i+\varepsilon_j.
\]
Choose a functional $v$ on the two-dimensional label space with
$v(a_{ij})=h(t_i+t_j)+\varepsilon_i+\varepsilon_j$.
These three prescriptions are consistent because both their arguments
and values sum to zero. Then $v(a)+h(z)=1$ on $U$; on $D_0\subset D$ this
uses $h=1$. If $D$ is empty, $D_0$ is empty and an edge requires
$\varepsilon_i+\varepsilon_j=1$. Take $h=0$ and
$v(a_{ij})=\varepsilon_i+\varepsilon_j$, again a consistent prescription.

If the graph has exactly one edge, relabel it $12$ and put $R=R_1=R_2$.
When $D_0$ is nonempty, $D_{11}=d+R$ with $d\notin R$; choose $h$ vanishing
on $R$ and equal to one on $d$. When $D_0$ is empty, take $h=0$.
In either case $h$ is constant, say $k$, on $D_{12}$, and choosing
$v(a_{12})=1+k$ proves the assertion. Finally, if there are no edges,
$D_0$ is empty or an affine subspace excluding zero, being an intersection
of such cosets. A functional $h$ equal to one on $D_0$ exists, and
$v=0$ suffices. This covers every graph on three vertices.\end{proof}

\paragraph{Application to the rank-three label model.}
Suppose $(E,\Lambda)$ is a spectral pair of size $64$ in
$\mathbb F_2^4\times\mathbb F_2^5$ and
\[
 \Lambda\subset\{0,e_1,e_2,e_3,e_4\}\times\mathbb F_2^5
\]
has fiber sizes $(8,8,16,16,16)$, in that order. Any ordering of the two
small and three large fibers can be put in this order by an affine
permutation of the five affine-basis labels. Translating the spectrum
and making the corresponding invertible linear coordinate changes
preserve spectrality and tiling of $E$.

For $u\in\mathbb F_2^4$,
\[
 W_\Lambda(u,0)
 =8+8(-1)^{u_1}+16\sum_{i=2}^4(-1)^{u_i}.
\]
Its zeros are exactly the three nonzero points of
$P=\langle e_2+e_3,e_2+e_4\rangle$. Consequently, at each fixed
$y\in\mathbb F_2^5$, the first-coordinate differences of points of $E$
lie in $P$. Set
\[
 L=P^\perp=\langle e_1,e_2+e_3+e_4\rangle
 \quad\text{inside }\mathbb F_2^4.
\]
The restriction of the standard dot product to $P$ is nondegenerate,
so $L\cap P=\{0\}$. The four translates $E+l$, $l\in L$, are disjoint,
and $B=E+L$ has size $256$.

Let $C_i\subset\mathbb F_2^5$ be the free-coordinate fiber of $\Lambda$
at $e_{i+1}$, for $i=1,2,3$. Each $C_i$ has size $16$. Fourier support
of $1_B$ lies in $L^\perp\times\mathbb F_2^5=P\times\mathbb F_2^5$.
At label zero, nonzero Fourier support is contained in
$D_0=D_{11}\cap D_{22}\cap D_{33}$: spectrality forces vanishing on each
$(C_i+C_i)\setminus\{0\}$, whose complement away from zero is $D_{ii}$.
At label $a_{ij}=e_{i+1}+e_{j+1}$, it is contained in
\[
 \mathbb F_2^5\setminus(C_i+C_j)=D_{ij},
\]
where the equality follows because the two sets have half the ambient
size. The three nonzero labels $a_{12},a_{13},a_{23}$ sum to zero.
Thus every nonzero Fourier frequency of $1_B$ lies in the set $U$ of
the three-halfsets lemma. Choose a translation vector $t$ whose pairing
with these frequencies is one. Fourier inversion gives
\[
 1_B(x)+1_B(x+t)=1\qquad(x\in\mathbb F_2^9),
\]
since the constant Fourier coefficient on the left is
$2|B|=512$, and all other coefficients vanish.
Hence $E$ tiles with the eight-point subgroup $L+\langle t\rangle$.
Here $t\notin L$, since $B$ is $L$-periodic and $B+t$ is its complement.

The preceding calculations and the three-halfsets lemma settle all
label models with $m\ge4$. It remains to examine the four-fiber model,
corresponding to $m=3$. We will settle every unequal pattern, including
all seven unequal patterns left by the weight-only calculation.

\subsection{The exact setting and result}
For $A\subset\F_2^n$, use the unnormalized Fourier transform
$W_A(u)=\sum_{x\in A}(-1)^{u\cdot x}$.
For real functions use normalized coefficients
$\widehat f(u)=2^{-n}\sum_x f(x)(-1)^{u\cdot x}$.
An equally sized pair $(E,\Lambda)$ is spectral when its character matrix
is Hadamard. In particular
\[
 (E+E)\setminus\{0\}\subseteq\{u:W_\Lambda(u)=0\}.
\]
An orthogonal character family on a set need not be a full spectrum.
We use its stated cardinality explicitly below.

The three-label reduction gives the following model:
\begin{equation}\label{b9:seven:model}
 G=\F_2^3\times\F_2^6,\quad |E|=|\Lambda|=64,\qquad
 \Lambda=\bigcup_{i=0}^3\{a_i\}\times A_i,
 \quad(a_0,a_1,a_2,a_3)=(0,e_1,e_2,e_3),
\end{equation}
with the normalized weight identity
\begin{equation}\label{b9:seven:weight}
 q_E(u,z):=\frac{\operatorname{wt}_E(u,z)}4\pmod2=u_1u_2u_3.
\end{equation}
Here $\operatorname{wt}_E(s)=|\{x\in E:s\cdot x=1\}|$ and all these
weights are divisible by four. In the rank-three construction the three
distinct nonzero labels are independent. Grouping the original coordinate
parities by those labels makes its pure cubic exactly the product in the displayed cubic identity;
zero-labeled coordinates do not occur. Thus the displayed weight identity has no hidden lower terms.

\begin{theorem}\label{b9:seven:main}
For every spectral pair satisfying the displayed four-fiber model and weight identity, if the four
fiber sizes are not all sixteen, then $E$ tiles $G$ by a three-dimensional
subgroup.
\end{theorem}
This proves tiling of $E$, whose cubic tensor supplied the model. It does
not by itself assert simultaneous tiling of the spectrum. The seven unequal patterns left by the weight-only test are all covered:
\[
\begin{array}{c|c}
\text{sorted fiber sizes}&\text{size of the quotient orthogonal family}\\\hline
(8,8,16,32)&48\\
(8,8,24,24)&48\\
(8,12,12,32)&44\\
(8,12,20,24)&44\\
(8,16,16,24)&40\\
(12,12,20,20)&40\\
(12,16,16,20)&36
\end{array}
\]
The proof applies to unequal fibers generally, not just this enumerated list.

\subsection{A small coding fact with positive certificates}
\begin{lemma}\label{b9:seven:cap}
Let $O$ be a set of $m\in\{4,8,12,16\}$ distinct nonzero vectors over $\F_2$.
Suppose it has no triangle $a,b,a+b$ and
\[
 \sum_{s\in O}s_i s_j=0\pmod2\quad\text{for all }i,j.
\]
Then some linear functional takes the value one on all of $O$.
\end{lemma}
\begin{proof}
Let $C$ be the evaluation code with columns $O$. The moment condition
says $C\subseteq C^\perp$. Distinct, nonzero, triangle-free columns say
that $C^\perp$ has minimum distance at least four. Since $C$ is even and
$m$ is even, $C+\langle\mathbf1\rangle$ is self-orthogonal. Extend it to
a binary self-dual code $D$. Then $D\subseteq C^\perp$, so $D$ has minimum
distance at least four.

At length four no such self-dual code exists. At lengths eight, twelve,
and sixteen the classical complete classification gives respectively
one, one, and three models of minimum distance at least four. For each,
every odd coset has a representative of weight one or three. The supplied
positive certificates verify all these cosets: respectively $8$, $32$,
and $128$ per length-sixteen model.

If $\mathbf1\notin C$, duality supplies an odd $v\in C^\perp$.
Its coset $v+D$ remains in $C^\perp$, but has a representative of weight
one or three, contradicting its minimum distance. Thus $\mathbf1\in C$,
which is precisely the desired functional.
\end{proof}
The external input is classification completeness, not an unverified
negative search. The finite checker verifies generator ranks,
self-duality, minimum distances, and every supplied odd-coset certificate.
The classification is due to Pless~\cite{Pless}; the generator matrices are available in~\cite{HMdata}.

\subsection{Quarter-integral Boolean Fourier spectra}
\begin{theorem}\label{b9:seven:boolean}
Let $f:\F_2^r\to\{-1,1\}$ be balanced and satisfy
$\widehat f(s)\in\tfrac14\mathbb Z$ for every $s$.
If its Fourier support has no triangle $a,b,a+b$, then $f$ has an
antiperiod: $f(x+t)=-f(x)$ for some $t\ne0$ and all $x$.
\end{theorem}
\begin{proof}
Write $\widehat f(s)=k_s/4$. Parseval gives $\sum k_s^2=16$.
If all $k_s$ are even, the possibilities are a single coefficient of
magnitude four, or four coefficients of magnitude two. The former is
a character. In the latter case the convolution identity
$\widehat f*\widehat f=\mathbf1_{\{0\}}$ forces the four support points
to sum to zero, so they form an affine plane. Balance excludes zero from
that plane. In either case a functional is odd on the entire support,
giving an antiperiod.

Otherwise let $O=\{s:k_s\text{ odd}\}$. The identity
\[
 \sum_s k_s(-1)^{s\cdot x}=4f(x)\equiv4\pmod8
\]
shows that $\sum_{s:s\cdot x=1}k_s$ is divisible by four for all $x$.
Using coordinate units and their pairwise sums gives
\[
 \sum_{s\in O}s_i=\sum_{s\in O}s_i s_j=0\pmod2.
\]
Also $|O|\equiv\sum k_s^2\equiv0\pmod4$, hence
$|O|\in\{4,8,12,16\}$. Lemma~\ref{b9:seven:cap} supplies a functional $t$ odd
on $O$, and rules out $|O|=4$. There are therefore at most two nonzero
even coefficients, each of magnitude two.

Put $g=(f+f(\cdot+t))/2$. Its values are $-1,0,1$ and its surviving
Fourier coefficients are precisely the $t$-even coefficients. There cannot
be exactly one, since it would make $g$ half-integral everywhere. If none
survive, $t$ already works. Otherwise there are two, at $a,b$, and
\[
 g=(\varepsilon_a\chi_a+\varepsilon_b\chi_b)/2.
\]
Parseval now forces eight odd coefficients, all of magnitude $1/4$.
Their span $W$ has dimension at most four by self-orthogonality and at
least four since its odd hyperplane contains eight distinct points.
Thus $O$ is the entire odd affine hyperplane of $W$.

For $h=(f-f(\cdot+t))/2$ we have $gh=0$ and $h^2=1-g^2$.
It follows that $h$ is supported on the physical halfspace
$\chi_{a+b}=-\varepsilon_a\varepsilon_b$, so its Fourier support $O$
is invariant under translation by $a+b$. Hence $a+b\in W$.
If $a\in W$, it is $t$-even, so $a\in O+O$, producing a forbidden
triangle. Thus $a\notin W$. Choose $v$ vanishing on $W$ with $v(a)=1$;
then also $v(b)=1$. The functional $t+v$ is odd on the whole support.
\end{proof}

\begin{corollary}[Half-density partial-spectrum lemma]\label{b9:seven:half}
Let $B\subset\F_2^7$ have size $64$, and let every linear evaluation word
on $B$ have weight divisible by eight. If more than $32$ distinct
characters are mutually orthogonal on $B$, then $B$ tiles by two translates.
\end{corollary}
\begin{proof}
For $f=1-2\mathbf1_B$, balance holds and its normalized nonzero Fourier
coefficients are quarter-integral, since $W_B$ is divisible
by sixteen. Let $S$ be the orthogonal character family. Its difference
set avoids the support of $\widehat f$. If that support contained
$a,b,a+b$, the translates $S,S+a,S+b,S+a+b$ would be pairwise disjoint,
forcing $4|S|\le128$. Thus it is triangle-free. Theorem~\ref{b9:seven:boolean}
gives $B+t=\F_2^7\setminus B$.
\end{proof}
The Boolean theorem is independent of ambient dimension. The numerical
divisibility in this corollary is specialized to the needed dimension seven.

\subsection{The quotient closes every unequal pattern}
\begin{proof}[Proof of Theorem~\ref{b9:seven:main}]
Choose the two largest fibers, indexed by $i,j$. Their combined size
is greater than $32$, unless all four fibers have size sixteen.
Put $a=a_i+a_j$ and let
\[
 L=\{(u,0):u\cdot a=0\}\subset G.
\]
This is a two-dimensional subgroup in the point-coordinate label space.
For any $u\in L$, the two large fibers have the same sign in
$W_\Lambda(u,0)$. Their mass exceeds that of the other two
fibers, so this coefficient is nonzero. Consequently
$(E+E)\cap L=\{0\}$ and the quotient map $G\to G/L$ is injective on $E$.
Its image $\bar E$ has size $64$ and half density in $G/L\cong\F_2^7$.

The quotient characters are
$L^\perp=\langle a\rangle\times\F_2^6$. The difference $a$ of two
affine-basis labels has Hamming weight one or two, so the displayed cubic identity vanishes on
all of $L^\perp$. Hence every quotient evaluation word on $\bar E$
has weight divisible by eight.

Translate the two large frequency fibers by $(a_i,0)$. They become
\[
 S=(\{0\}\times A_i)\cup(\{a\}\times A_j)\subset L^\perp.
\]
Their pairwise character inner products are unchanged by frequency
translation, and descend unchanged to the injective image $\bar E$.
Thus $S$ is an orthogonal quotient family with more than $32$ members.
Corollary~\ref{b9:seven:half} gives an antiperiod $\bar t$ of $\bar E$.
Lifting it to $t\in G$, the eight-point subgroup
$L+\langle t\rangle$ complements $E$.
\end{proof}

\subsection{The one remaining pattern: an exact frontier}
Suppose all four fibers in the four-fiber model have size sixteen. Then
\[
 W_\Lambda(u,0)=0
 \quad\Longleftrightarrow\quad
 u\in P\setminus\{0\},\qquad P=\{0,3,5,6\}.
\]
The binary masks denote $e_1+e_2$, $e_1+e_3$, $e_2+e_3$.
At each fixed free coordinate, the point labels lie in a coset of $P$.
The line $\langle111\rangle$ avoids $P$, so projection along it is
injective on $E$. Let $Q\subset\F_2^8$ be the 64-point image.
Its quotient characters have even label parity, where the cubic weight identity vanishes.
Therefore $Q$ is triply even.

The three spectrum fibers at $e_1,e_2,e_3$, translated by $e_1$, descend
to an ordinary orthogonal character family of size $48$ on $Q$.
This is an exact reduction, not a full spectrum of $Q$.
A plane tiling complement for $Q$ would lift to a three-dimensional
complement for $E$. Completing the structured 48-character family to
a 64-character spectrum would also suffice, by the quarter-density
spectral tiling argument of Theorem~\ref{b8:quarter}. No such completion
has been proved here.

One useful necessary Fourier condition is particularly small. Put
$c=W_Q/16$. Then
\[
 c(0)=4,\qquad c(u)\in\mathbb Z,\qquad
 \sum_{u\ne0}c(u)^2=48,\qquad c*c=16c.
\]
Moreover its support cannot contain a three-dimensional subspace:
the eight translates of the 48-character family would otherwise be
pairwise disjoint in a group of size $256$.
These conditions have not yet yielded a plane complement.

The four-fiber conclusion concerns $E$, the member whose cubic moment
has contraction rank three. Applying it to the transposed pair proves
an assertion about the other member. In particular, it does not justify
assuming that every remaining pair has a rank-three tensor. The explicit
examples supplied with the code include spectral pairs of cubic ranks
$(5,5)$ and pairs with both cubic tensors zero; they have tiling complements
and satisfy the period and odd-plane-section filters. The unresolved
$64$-point problem includes these other tensor types as well as the
uniform four-fiber case.

\section{The eighth-density support assertion fails}\label{b9:eighth}
Lemma~\ref{b8:supportplane} is the central high-density step in dimension
eight. A possible replacement in dimension nine would assert that the
Fourier support of a set of density one eighth contains a three-dimensional
subspace. The following example shows that even adding spectrality does
not make this assertion true.

\begin{proposition}\label{b9:counterexample}
There is a spectral set $B\subset\F_2^9$ of cardinality $64$ whose Fourier
support contains no three-dimensional subspace. Both $B$ and the spectrum
in the construction nevertheless have linear subspace tiling complements.
\end{proposition}
\begin{proof}
We identify integers from $0$ to $255$ with vectors of $\F_2^8$, the least
significant bit being the first coordinate, and write addition as binary
exclusive-or. Put
\[
\begin{split}
 L=\{&0,128,120,248,68,196,60,188,34,162,90,218,153,25,225,97,\\
     &238,110,150,22,83,211,43,171,207,79,183,55,141,13,245,117\},\\
 A=\{&0,1,2,3,4,5,6,7,8,16,32,56,64,88,104,112,\\
     &128,152,168,176,200,208,224,248,247,239,223,199,191,167,151,143\}.
\end{split}
\]
The two sets form a spectral pair in $\F_2^8$, as direct character sums
verify. Set $B=L\times\F_2$ and $E=A\times\F_2$. The tensor product
of their character matrix with the order-two Walsh matrix shows that
$(E,B)$ is a spectral pair in $\F_2^9$.

A direct Walsh transform gives $\operatorname{supp}W_B=S\times\{0\}$,
where
\[
\begin{split}
 S=\{&0,25,26,28,31,41,43,45,47,50,51,52,53,73,74,\\
     &77,78,84,85,86,87,98,99,102,103,121,123,124,126\}.
\end{split}
\]
There are $56$ two-dimensional subspaces contained in $S$, and for each
such plane $P$ and each $z\in S\setminus P$, at least one point of $z+P$
lies outside $S$. This finite test is exhaustive: a three-dimensional
subspace would be $P\cup(z+P)$ for any one of its planes and any point
outside that plane. The complete integer checker is printed in
Appendix~\ref{b9:appendix}.

For clarity, this is not a counterexample to Fuglede's conjecture.
The eight-point set
\[
 T=\{0,25,43,52,77,87,98,126\}
\]
satisfies $T+T=S$, and therefore complements $E$ by
Lemma~\ref{b8:certificate}. It is not a subgroup. Moreover direct
addition verifies the subgroup complements
\[
 K_E=\langle25,42,76\rangle,\qquad
 K_B=\langle1,2,4\rangle,
\]
embedded in the first eight coordinates. In each case the $64\cdot8$
sums are all distinct. Thus the failed assertion concerns a subspace
inside the Fourier support, rather than the existence of a subspace
tiling complement by other means.
\end{proof}


\section{The odd-prime case: reductions and direction lemmas}\label{p11:prelim}
\providecommand{\D}{\mathcal D}
\providecommand{\U}{\mathcal U}
\providecommand{\PG}{\mathrm{PG}}
We now prove Fuglede's conjecture in $\F_{11}^3$. The geometric starting
point is the argument of Fallon, Mayeli, and Villano \cite{FMV} for
$\F_7^3$. For a set $B\subset\F_p^3$, let $\D(B)\subset\PG(2,p)$
be the directions of its nonzero differences, and let $\U_B$ be the
normal directions for which all $p$ parallel planes have the same
intersection size with $B$. When convenient we write $D(B)$ for $\D(B)$.
Spectrality relates these two sets of directions by
\[
 \D(A)\subseteq\U_E,\qquad \D(E)\subseteq\U_A.
\]
Our proof retains this relation, the line and plane intersection bounds,
and the reduction from a nearly full pencil of uniform normals to parallel
layers. Additional moment identities and finite planar classifications
will be needed to finish the argument.

\begin{proposition}[The cardinalities to be excluded]\label{p11:cardinalities}
To prove Fuglede's conjecture in $\F_{11}^3$, it suffices to exclude
spectral sets of cardinalities $22,33,44,55,66,77$.
\end{proposition}
\begin{proof}
The reductions of Aten et al.\ \cite{Aten} show that tiling implies
spectrality in $\F_p^3$, that spectral sets of cardinalities $1,p,p^2,p^3$
tile, and that every other spectral cardinality is $mp$, $2\le m\le p-1$.
We recall direct proofs of the elementary parts below. Malikiosis
\cite[Theorem 1.4]{Malikiosis} excludes
\[
 p^2-(p-1)\sqrt p<|E|<p^2.
\]
For $p=11$, this excludes $88,99,110$, since
$121-10\sqrt{11}<88$. The six asserted cardinalities remain.
\end{proof}

\begin{lemma}[Elementary spectral reductions]\label{p11:prelims-basic}
Let $(E,A)$ be a spectral pair in $\F_p^3$, of size $mp$,
$2\le m\le p-1$. Each affine plane contains at most $p$ points of
either member, and each affine line contains at most $\min(m,p-m)$.
Neither uniform-normal set, and hence neither difference-direction set,
contains a complete projective line.
\end{lemma}
\begin{proof}
Write $\zeta=e^{2\pi i/p}$. A nonzero Fourier zero of $1_E$ gives an
integer polynomial
\[
 \sum_{t=0}^{p-1}|\{x\in E:x\cdot v=t\}|X^t
\]
vanishing at $\zeta$. It is a multiple of
$1+X+\cdots+X^{p-1}$, so its coefficients are equal. The square
character matrix of a spectral pair is unitary after scaling; its
transpose also has orthogonal rows. This proves both the symmetry and
the displayed inclusions of direction sets.

If an affine line has $a$ points of $E$, the $p+1$ planes through it
have total occupancy
\begin{equation}\label{p11:pencil}
 |E|+pa.
\end{equation}
A complete pencil of uniform normals would give $(p+1)m=mp+pa$,
which is impossible. A plane with $p+1$ points determines every
direction of its direction plane, by pigeonholing into parallel lines.
It would give a complete uniform pencil for the partner. This proves
the plane bound.

For a line of direction $v$, the spectrum has two points whose difference
belongs to $v^\perp$, because $|A|>p$ and there are $p$ cosets of
$v^\perp$. The line is therefore contained in a uniform plane, and
$a\le m$. In the pencil through that line one plane has occupancy $m$
and all others have occupancy at most $p$. Thus
$mp+pa\le m+p^2$, and integrality gives $a\le p-m$.
\end{proof}

For completeness, the cardinality reductions used above also have short
proofs. A spectral set of size greater than $p^2$ determines every
direction, so the spectrum has all nonzero Fourier coefficients zero.
Fourier inversion makes the spectrum the whole space. A spectral set of
size $p$ is a transversal of parallel planes, since any nonzero Fourier
zero gives equidistribution; it tiles by the corresponding plane.
A spectral set of size $p^2$ must miss a direction, since otherwise
its spectrum again has all nonzero coefficients zero. It is a transversal
of the lines in that direction, and tiles by the line.

Conversely, a tile in $\F_p^3$ has size $1,p,p^2$ or $p^3$. If its
size is $p$, its tiling complement has a nonzero Fourier coefficient at
some nonzero frequency, and the tiling identity forces a Fourier zero
of the tile there. It is a plane transversal and the corresponding
frequency line is a spectrum. If the tile has size $p^2$, choose a
nonzero difference $w$ in its size-$p$ complement. For every nonzero
$v\in w^\perp$, the complement has two points with the same dot
product with $v$, so it is not equidistributed there. Its Fourier
coefficient is nonzero, forcing the tile's coefficient to vanish.
Thus $w^\perp$ is a spectrum. The other two cardinalities are immediate.

\begin{lemma}[The parallel-plane criterion]\label{p11:parallel-planes}
Let $|E|=mp$, $2\le m\le p-1$, with every affine plane containing
at most $p$ points of $E$. If at least $p-1$ uniform normals lie in a
two-dimensional subspace $P$, then projection along $P^\perp$ is
injective on $E$, and $E$ is contained in $m$ parallel planes whose
normal direction belongs to $P$.
\end{lemma}
\begin{proof}
The pencil through a line parallel to $P^\perp$ gives
$mp+pa\le(p-1)m+2p$, whence $a\le1$. The projection therefore has
Boolean indicator $f$. Its Fourier support is contained in at most
two lines. After a linear change of coordinates, Fourier inversion
writes $f(x,y)=a(x)+b(y)$. If $a$ is nonconstant, the difference of
two rows is a fixed nonzero number, necessarily $1$ or $-1$. Both rows
must then be constant, forcing $b$ constant. Thus $f$ depends on one
coordinate and is the indicator of $m$ parallel lines. Their inverse
images are the required planes. Their normal lies in $P$, which is
the dual of the projection quotient.
\end{proof}
This is the parallel-plane argument of \cite[Lemma 3]{FMV}, with the
plane bound and the location of the resulting normal made explicit.

\begin{lemma}[The planar projection identity]\label{p11:radon}
Project a set of size $11m$ onto $\F_{11}^2$, retaining its integer
multiplicity function $f$. Suppose $K$ line directions are uniform
with line sum $m$, and set $L=12-K$. Choose a uniform direction to be
vertical, and let $s_1,\ldots,s_L$ be the remaining slopes. If $g_i(t)$
is the mass on $y-s_i x=t$, then
\begin{equation}\label{p11:radon-identity}
 \sum_{i=1}^L g_i(y-s_i x)=m(L-1)+11f(x,y).
\end{equation}
The reduced polynomial of every $g_i$ has degree at most $L-2$.
For its coefficients $c_{i,j}$ of positive degree, put
$w_i=(\prod_{k\ne i}(s_i-s_k))^{-1}$. Then
\begin{equation}\label{p11:radon-null}
 c_{i,j}=w_iR_j(s_i),\qquad \deg R_j\le L-j-2.
\end{equation}
\end{lemma}
\begin{proof}
The twelve lines through a point have total mass $11m+11f(x,y)$.
Remove the $K$ constant families to obtain the identity. Modulo eleven,
represent each profile by its polynomial of degree at most ten. The
identity is formal, since its degree in each variable is at most ten.
For every $j\ge1$ its homogeneous coefficients give
$\sum_i c_{i,j}s_i^h=0$ for $0\le h\le j$; the binomial coefficients
are invertible. The Vandermonde matrix has full rank for $j\ge L-1$.
For lower $j$, its kernel is precisely the displayed family, by the
Lagrange identities $\sum_iw_is_i^h=0$ for $0\le h\le L-2$ and
the dimension of the kernel.
\end{proof}

\subsection{Incidence and finite planar direction bounds}
\begin{lemma}[Incidence obstruction to a low-degree curve]\label{p11:incidence}
Let $S\subset\mathrm{PG}(2,p)$ have $N$ points, with
$a\le |S\cap L|\le b$ on every projective line. Let $a\le d\le b$.
If a nonzero homogeneous polynomial of degree $d$ vanishes on $S$, then
\[
 N^2+\bigl[p-(a+d)(p+1)\bigr]N+ad(p^2+p+1)
 \le d(b-a)(b-d).
\]
\end{lemma}
\begin{proof}
A line not dividing the polynomial contains at most $d$ roots. Thus
$(n_L-a)(n_L-d)\le0$ on every nonfactor line, whereas on each of at
most $d$ factor lines it is at most $(b-a)(b-d)$. Sum and use
$\sum_Ln_L=(p+1)N$ and $\sum_Ln_L^2=N^2+pN$.
\end{proof}
We shall apply this inequality repeatedly.
If the left side's minimum over integers exceeds the right side, no
such polynomial exists. Lower degrees are covered by multiplying by a
power of a nonzero linear form.

\begin{proposition}[Weighted incidence divisibility]\label{p11:weighted-incidence}
Let $(E,A)$ be a spectral pair in $\F_p^3$ of size $mp$, with $p$ odd
and $1\le m<p$. For a projective normal $v$, let
\[
 q_E(v)=\sum_t\binom{|E\cap\{x:v\cdot x=t\}|}{2},
 \qquad k_E(v)=\frac{q_E(v)-p\binom m2}{p},
\]
and similarly for $A$. Suppose the $k$'s are integers, and every nonzero
one is at least $a>0$. Put $V_E=\{v:k_E(v)>0\}$ and define $V_A$
similarly. Suppose every projective line meets each $V$ in at least
$r$ points. Define
\[
 C=\frac{m(p-m)}2,\qquad T=\frac{m(p^2-m)}2.
\]
If $C<(r+1)a$, then
\[
 C|V_E|=C|V_A|=rT.
\]
In particular, nonintegrality of $rT/C$ excludes the spectral pair.
\end{proposition}
\begin{proof}
Convexity of $\binom n2$ shows that $k_E\ge0$, with equality exactly
at the uniform normals. Let $t_E(d)$ count unordered pairs in $E$ with
difference direction $d$. Let $B$ be the projective point-line incidence
matrix defined by $B_{v,d}=1$ when $v\cdot d=0$. The polarity gives
\[
 B^2=pI+J,\qquad B\boldsymbol1=(p+1)\boldsymbol1,
 \qquad Bt_E=q_E.
\]
Also $\sum_dt_E(d)=\binom{mp}{2}$. Applying $B$ and simplifying gives
\[
 t_E=Bk_E-C\boldsymbol1,\qquad \sum_vk_E(v)=T.
\]
If $d\in V_A$, then $d\notin D(E)$ by spectrality, so
\[
 \sum_{v\perp d}k_E(v)=C.
\]
There are at least $r$ nonzero summands and, since each is at least $a$
and $C<(r+1)a$, at most $r$. Hence there are exactly $r$. Exchanging
the two sets shows that each point of $V_E$ is incident with exactly
$r$ points of $V_A$. Summing the displayed equation over $V_A$ gives
$C|V_A|=rT$; symmetry gives the other equality.
\end{proof}

\begin{lemma}\label{p11:six-point}
Every noncollinear six-point subset of $\F_{11}^2$ determines at least six directions.
\end{lemma}
\begin{proof}[Finite certificate]
Normalize a noncollinear triangle to $T=\{(0,0),(1,0),(0,1)\}$. Any direction set of size at most five extends to
\[
 S=\{0,10,\infty,a,b\},\qquad 1\le a<b\le9.
\]
For each of the $36$ masks, enumerate every point outside $T$ whose
three directions to $T$ belong to $S$. Three additional points extend
$T$ to a six-point set only if all their pairwise directions also lie
in $S$. The complete enumeration in
\path{code/prime11/audit/checks/size66_77/independent66_checks.py}
checks $1,566$ candidate triples and finds no compatible triple.
The full program is printed in Appendix~\ref{p11:code}. It constructs
all candidate lists directly and uses exact arithmetic, so neither
an orbit classification nor an externally supplied list is needed.
\end{proof}

\begin{lemma}[Canonical completion of a nearly full graph]\label{p11:completion}
Let $p$ be an odd prime, and let $E\subset\F_p^2$ have $p-1$
points and a missing direction.
Then $e_0=-\sum_{e\in E}e$ is not in $E$, and
\[
 D(E\cup\{e_0\})=D(E).
\]
\end{lemma}
\begin{proof}
For any linear form $\ell$ injective on $E$, its values omit exactly one
field element. That element is $-\sum_{e\in E}\ell(e)=\ell(e_0)$,
because the sum of all field elements is zero for odd $p$.
Thus $e_0\notin E$ and every previously missing direction remains missing.
Adding a point cannot remove a determined direction, proving equality.
\end{proof}

\begin{lemma}[Ten-point direction bound over $\F_{11}$]\label{p11:ten-point}
Every ten-point set with at most three collinear points determines
at least ten directions.
\end{lemma}
\begin{proof}[Finite classification with three geometric witnesses]
A counterexample has at least three missing directions. Choose two as
coordinate directions and apply the completion lemma. The resulting
eleven-point set is a permutation graph, has at most four points on a
line, and still has at most nine directions. Normalize two points to
$(0,0),(1,1)$, leaving $9!$ permutations to examine.

The complete checker \path{code/prime11/audit/checks/size55/ten_point_completion.cpp} finds exactly
1320 such normalized permutations. They have exactly three four-point
lines each and form three affine orbits, with representatives
\[
\begin{array}{c|rrrrrrrrrrr}
x&0&1&2&3&4&5&6&7&8&9&10\\\hline
f_1(x)&0&1&2&3&5&7&4&9&6&10&8\\
f_2(x)&0&1&2&4&3&8&10&9&7&6&5\\
f_3(x)&0&1&2&4&6&5&9&3&10&8&7
\end{array}
\]
The three four-point lines $y=mx+b$ are respectively
\[
\begin{array}{c|ccc}
 & (m,b)&(m,b)&(m,b)\\\hline
f_1&(1,0)&(5,7)&(9,6)\\
f_2&(2,9)&(6,0)&(10,4)\\
f_3&(1,0)&(2,9)&(7,0)
\end{array}
\]
In each row the three lines have no common point. Deleting one point
therefore cannot reduce all line intersections to at most three.
This contradicts the assumed original set.

For coverage, the checker examines every permutation with the normalized
first two values. Its affine canonicalization tests every invertible
linear coordinate change whose two projections are injective, every
horizontal translation, and the uniquely determined vertical translation
and scaling restoring the first two values. These exhaust the affine
equivalences between normalized permutation graphs. Independently, it
checks that the intersection of all four-point lines is empty for every
one of the 1320 graphs, so the exclusion does not depend on the orbit
canonicalization.
\end{proof}


\section{\texorpdfstring{A computer-assisted exclusion of cardinality $22$}{A computer-assisted exclusion of cardinality 22}}
\label{p11:sec:size22}

The finite checks in this section are implemented by the standard-library
Python program \path{code/prime11/original/output/fuglede11_next/size22/verify22.py}; its complete output gives the counts recorded below. All searches use exact arithmetic in
$\F_{11}$, not numerical optimization. The geometric and algebraic
arguments below explain their coverage. Appendix~\ref{p11:code} records the programs, their inputs, and the
finite conclusions used here.

Write $D(B)$ for the projective directions of differences of $B$, and write
$U_B$ for the normal directions in which the eleven parallel-plane
occupancies of $B$ all equal two. We use the published facts of
Fallon--Mayeli--Villano: for a spectral pair $(E,A)$ of size $22$, every
line contains at most two points, every plane at most eleven, and
$D(A)\subseteq U_E$, $D(E)\subseteq U_A$. We also use their Lemma 3.
Neither $U_E$ nor $U_A$ contains a projective line: otherwise the twelve
planes through any affine line in the perpendicular direction would
all have occupancy two, giving $22+11a=24$ for an integer line occupancy
$a$, which is impossible. Consequently neither direction set contains
a projective line either.

\begin{lemma}\label{p11:lem:22two}
Neither member of a spectral pair of size $22$ can lie in two parallel
affine planes.
\end{lemma}
\begin{proof}
Suppose $E$ does. Its two slices each have eleven points and no three
collinear. Each slice determines at least eleven directions: its $55$
pairs contribute at most five pairs in each direction. Since $D(E)$
does not contain a projective line, the union of their in-plane direction
sets has exactly eleven directions. The published parallel-plane
criterion therefore puts $A$ in two parallel planes as well.

Fix one of the eleven tangent directions $d$ determined by the slices.
In this direction each slice consists of five pairs and one singleton.
Thus $E$ consists of ten pairs parallel to $d$ and two singletons.
For a normal $v\in U_E\cap d^\perp$, every pair is wholly in one of the
two-point planes $v\cdot x=t$. The two singletons must therefore lie
in the same plane. If their difference is $s$, then $v\cdot s=0$.
Since the singletons lie in different layers, $s$ is not parallel to
$d$, and there is at most one projective direction satisfying both
conditions on $v$. The eleven projective lines $d^\perp$ form a pencil
through the layer normal $n$. The twelfth pencil line contains at most
eleven points of $U_E$, because $U_E$ contains no full line. Hence
$|D(A)|\le|U_E|\le22$.

On the other hand, write the two eleven-point slices of $A$ as
$P\times\{0\}$ and $Q\times\{1\}$ after an affine change of coordinates.
Their in-plane differences determine eleven directions. The cross-layer
directions are in bijection with $Q-P$. We have $|Q-P|>11$: otherwise all
the eleven-point sets $Q-p$, $p\in P$, coincide. Hence $Q$ is invariant
under a nonzero difference of $P$, and its size eleven forces it to be an
affine line, contradicting the cap condition. Consequently
$|D(A)|\ge11+12=23$, a contradiction.
\end{proof}

\subsection{Projection profiles}
Apply Lemma~\ref{p11:radon} with $m=2$. Thus the bad profiles satisfy
\begin{equation}\label{p11:eq:22radon}
 \sum_{i=1}^L g_i(y-s_i x)=2(L-1)+11f(x,y),
\end{equation}
and their degree-$j$ coefficients satisfy
\begin{equation}\label{p11:eq:22null}
 c_{i,j}=w_iR(s_i),\qquad \deg R\le L-j-2.
\end{equation}
Here $f$ takes values in $\{0,1,2\}$ and has total mass $22$.

\begin{lemma}\label{p11:lem:22nine}
If $U_E$ contains nine directions in a two-dimensional linear subspace,
then $E$ lies in two parallel planes.
\end{lemma}
\begin{proof}
Use the preceding projection with $K=9$ and $L=3$. Each $g_i$ is affine
modulo eleven. A nonconstant affine function takes every residue once,
so its nonnegative integer values sum to at least $0+1+\cdots+10=55$,
whereas every $g_i$ has sum $22$. Thus each $g_i$ is constant modulo
eleven. Since $0\le g_i\le11$, it is either constantly two or takes
values zero and eleven, with exactly two elevens. The constant term of
\eqref{p11:eq:22radon} is four modulo eleven, so exactly two of the three
functions are constantly two and the remaining one is of the latter
type. Equation \eqref{p11:eq:22radon} now says that $f$ is the indicator of
two parallel lines. Its inverse image contains $E$ in two parallel planes.
\end{proof}

\subsection{The finite lemmas and their coverage}

\begin{lemma}[Verified finite planar statements]\label{p11:lem:22finitecaps}
In $\F_{11}^2$ the following hold.
\begin{enumerate}
\item Every seven-point cap determines at least nine directions.
\item Every six-point cap determining at most six directions determines
exactly six, projectively equivalent to
$G_6=\{0,1,4,7,10,\infty\}$.
\item Every five-point cap determining at most five directions determines
exactly five, projectively equivalent to
$G_5=\{0,2,3,10,\infty\}$.
\item A cap whose parallel chords always have equal difference vectors
up to sign has at most six points.
\end{enumerate}
\end{lemma}

Here and below a cap means a set with no three collinear. For the first
three assertions, normalize three points to $(0,0),(1,0),(0,1)$.
Their slopes are $0,\infty,-1$. The program enumerates every allowed
direction set containing these slopes and then every compatible cap
extension of the triangle. There are respectively $126$, $84$, and $36$
allowed masks. The seven-point search exhausts $6756$ states without a
seven-point cap. The six-point search finds ten supporting masks, all
in the $110$-element projective orbit of $G_6$; the five-point search
finds six masks, all in the $132$-element orbit of $G_5$. A separate
exhaustive check of the $36$ normalized five-direction masks finds no
six-point cap, so the six-point direction sets have exactly six elements.
For five points, the elementary pair count already gives five directions.

For the fourth assertion, normalize a triangle as above and retain, for
each direction, the unique permitted difference vector up to sign.
Reject any insertion creating a collinear triple or a different chord
length in an already used direction. The search exhausts $890$ states;
its maximum size is six and no seven-point configuration survives.
Both conditions and the triangle normalization are affine-invariant.

\begin{lemma}[Verified line-count polynomial statements]\label{p11:lem:22finiteprofiles}
There is no nonconstant polynomial of degree at most three whose
eleven values, interpreted as integers, lie in $\{0,\ldots,6\}$ and sum
to $22$. The following projection configurations are also impossible:
\begin{enumerate}
\item plane occupancies at most six and six uniform directions $G_6$;
\item plane occupancies at most five and any six uniform directions;
\item plane occupancies at most five and five uniform directions $G_5$;
\item plane occupancies at most four and any five uniform directions.
\end{enumerate}
In the first two cases assume that no projective line has more than six
uniform directions; in the last two assume the upper bound five.
\end{lemma}

We give the exact search parametrizations, so the asserted exclusions
do not rely on an unrestricted solver. There are $385$ nonconstant
quartics with values in $\{0,\ldots,6\}$ summing to $22$. Every such
quartic has eleven translates, and exactly one translate has cubic
coefficient zero. There are $35$ centered quartics. Of these, $15$ have
maximum value at most five and $10$ have maximum at most four.
There are $74$ centered quintics with values at most five and sum $22$;
centering a quintic means that its degree-four coefficient is zero.
Of these, $24$ have maximum value at most four. All these lists are
obtained by enumerating polynomial coefficients, not by interpolation
from a selected sample.

For six uniform directions, $L=6$ and the bad profiles have degree at
most four. None is uniform under the hypothesis on the total number of
uniform directions. The low-degree exclusion therefore makes every
bad profile a quartic. By \eqref{p11:eq:22null}, their leading coefficients
are $\lambda w_i$ for one nonzero $\lambda$. Their cubic coefficients
are $w_i(a+bs_i)$. Translating the projected coordinates replaces the
argument of each profile by $t+H+Js_i$, and hence can make all cubic
coefficients zero simultaneously. The program tests every selection of
the centered quartics with the specified leading coefficients, and
checks all remaining quadratic, linear and constant coefficients in
\eqref{p11:eq:22radon}. The constant required here is ten modulo eleven.
For $G_6$ this tests $10{,}000$ tuples, with no identity. For all $84$
normalized six-direction masks with the occupancy bound five, it tests
$8880$ tuples, again with no identity.

For five uniform directions, $L=7$ and the profiles have degree at most
five. There are two exhaustive branches. If a quintic occurs, every
profile is quintic and the leading coefficients are $\lambda w_i$.
Translation makes all degree-four coefficients zero. The degree-three
coefficients then have the form $w_i(a+bs_i+cs_i^2)$. Enumerating
$\lambda,a,b,c$, and the centered profiles with those coefficients,
leaves only the degree-two, degree-one and constant equations to test.

Otherwise the profiles are quartics. Their leading coefficients are
$w_i(a+bs_i)$. They must all be nonzero, since a profile of degree at
most three would be constantly two and give an extra uniform direction.
The cubic coefficients are $w_iR(s_i)$ with $\deg R\le2$. Translation
changes $R(s)$ by a multiple of $(a+bs)(H+Js)$. If $a\ne0$, every
translation class has a representative $R(s)=cs^2$; if $a=0$, every
class has a representative $R(s)=c$. Thus these two forms cover all
quartic configurations. Again the remaining coefficient equations are
checked, now with required constant $2(7-1)=1$ modulo eleven.

For $G_5$ and occupancy bound five, the quintic branch tests $3602$
profile tuples and the quartic branch $9900$, with no identities.
For all $36$ normalized five-direction masks and occupancy bound four,
the corresponding counts are $552$ and $19{,}800$, again with no
identities. The hypotheses ensure that omitted profiles of degree at
most three would be uniform, so the two branches are exhaustive.

\subsection{Successive reductions and the final contradiction}

\begin{theorem}\label{p11:thm:22}
There is no spectral pair
of cardinality $22$ in $\F_{11}^3$.
\end{theorem}
\begin{proof}
Both $E$ and $A$ are caps. If either has seven points in an affine plane,
Lemma \ref{p11:lem:22finitecaps}(1) gives nine directions of uniformity for
the other in a linear plane. Lemmas \ref{p11:lem:22nine} and \ref{p11:lem:22two}
give a contradiction. Thus every affine plane contains at most six
points of either set.

If $U_E$ has at least seven directions in a linear plane, choose seven
of them in \eqref{p11:eq:22radon}. The remaining five profiles have degree
at most three and values at most six. The profile check forces them all
to be constantly two, which contradicts the identity (its constant
terms would be $10$ and $8$ modulo eleven). Consequently every projective
line meets $U_E$ in at most six points, and likewise for $U_A$.

If an affine plane contains six points of $A$, their direction set lies
in $U_E$ and hence has at most six directions. The six-cap classification
reduces it to $G_6$. The first profile exclusion is a contradiction.
By symmetry, the plane bound improves to five. The second profile
exclusion now improves the bound on every projective-line intersection
with either $U$ to five.

If a plane contains five points of $A$, those points determine exactly
five directions and the five-cap classification reduces them to $G_5$.
The third profile exclusion is a contradiction. Thus the plane bound
improves to four. The fourth profile exclusion improves every
projective-line intersection with either $U$ to at most four.

Since $D(E)\subseteq U_A$, every projective line meets $D(E)$ in at most
four points. Taking the twelve lines through one point of $D(E)$ gives
\begin{equation}\label{p11:eq:22directions}
 |D(E)|\le1+12\cdot3=37.
\end{equation}
Every four coplanar points of $E$ form a parallelogram: they are a cap
and determine at most four directions. Among their six pairs, two
distinct directions must occur twice, forcing both pairs of opposite
sides to be parallel.

Fix a determined direction $d$. Since $E$ is a cap, each secant parallel
to $d$ contains exactly two points. Any two such secants lie in an
affine plane with four points of $E$, hence form a parallelogram.
Their difference vectors therefore agree up to sign. Choose one
nonzero vector $v$ in this direction and orient every secant consistently,
writing their endpoints as $b_i,b_i+v$, $1\le i\le r_d$.

Project the $b_i$ to the quotient by $\langle v\rangle$, obtaining a
planar set $C$ of size $r_d$. It is a cap: three collinear projected
points would put six points of $E$ in an affine plane. We claim that
parallel chords of $C$ have equal vectors up to sign. Choose lifts of
two parallel chord vectors as $w$ and $w'=\lambda w+\mu v$, with
$\lambda\ne0$. The first pair of secants determines the four directions
\[
 [v],\quad[w],\quad[w+v],\quad[w-v].
\]
These exhaust the directions of $E$ in their linear plane. The second
pair of secants therefore forces
\[
 \left\{\frac\mu\lambda,\frac{\mu+1}\lambda,
                  \frac{\mu-1}\lambda\right\}=\{0,1,-1\}.
\]
Taking means gives $\mu=0$, and taking sums of squares gives
$\lambda^2=1$. Thus the projected chord vectors agree up to sign,
as claimed. Lemma \ref{p11:lem:22finitecaps}(4) yields $r_d\le6$.

There are $\binom{22}{2}=231$ unordered pairs in $E$, each belonging to
one direction. Hence
\[
 231=\sum_{d\in D(E)}r_d\le6|D(E)|\le6\cdot37=222,
\]
the final contradiction.
\end{proof}


\section{\texorpdfstring{The exclusion of cardinality $33$}{The exclusion of cardinality 33}}\label{p11:33}
The argument for cardinality $33$ has two parts. We first show that a
spectral set must be a graph over one particular planar set. We then use
vanishing moments to prove that this graph is centrally symmetric, after
an affine change of coordinates. A determinant obstruction finishes the
proof. The finite calculation in the second part is the largest calculation
needed for $\F_{11}^3$.

\subsection{Plane intersections and uniform directions}
Throughout this section $(E,A)$ denotes a hypothetical spectral pair in
$\F_{11}^3$ with $|E|=|A|=33$. The published bounds in
Proposition~\ref{p11:prelims-basic} give at most three points on each
line and at most eleven points on each plane, for either member.

\begin{lemma}\label{p11:33-small-planar}
A planar set over $\F_{11}$ with no four collinear points satisfies the
following assertions.
\begin{enumerate}
\item If it has eight points, it determines at least eight directions.
\item If its missing directions contain a harmonic four-set, it has
at most eight points.
\end{enumerate}
\end{lemma}
\begin{proof}
Both assertions are finite planar checks. We give the normalizations so
that their coverage is explicit. For the first assertion, a noncollinear
triple can be sent to $(0,0),(1,0),(0,1)$. If the set determined at most
seven directions, its directions would be contained in one of the $126$
seven-element subsets of $\mathrm{PG}(1,11)$ containing
$0,-1,\infty$. For each of these subsets, enumerate the points compatible
with the triangle, and extend the triangle in increasing order, requiring
each new difference to have an allowed direction and rejecting any four
collinear points. The search has $25,725$ nodes; the largest set found has
seven points.

For the second assertion, send the four missing directions to
$\{0,1,2,\infty\}$ and translate a point to the origin. Enumerate
sets whose differences have slopes in $\{3,4,\ldots,10\}$, again
rejecting four collinear points. This search has $266,001$ nodes and
maximum size eight. These are precisely the tests in
\path{verify_small_profiles.py} and \path{classify_eight_cliques.py}
in the accompanying \path{size33} directory. Each extension considers
all remaining compatible points, so no equivalence test is needed for
completeness.
\end{proof}

\begin{lemma}\label{p11:33-layers}
Neither $E$ nor $A$ is contained in three parallel affine planes.
\end{lemma}
\begin{proof}
Suppose $E$ is so contained. Each layer has eleven points, and hence
determines at least ten directions by Lemma~\ref{p11:ten-point}.
The parallel-plane criterion, Lemma~\ref{p11:parallel-planes}, places
$A$ in three parallel planes and also makes its projection along the
annihilator of the original direction plane injective. Apply the same
criterion to an eleven-point layer of $A$. The two conclusions give dual
coordinates in which
\[
 E=\{(x,f_s(x),s):x\in\F_{11},\ s\in S\},\qquad
 A=\{(u,t,g_t(u)):u\in\F_{11},\ t\in T\},
 \qquad |S|=|T|=3.
\]
For clarity, the first layer normal may be taken as the third coordinate;
the normal furnished for the partner belongs to its orthogonal plane
and may be taken as the second coordinate. The two injectivity
conclusions then make both displayed parametrizations bijective.

Let $\zeta=e^{2\pi i/11}$ and let
$F=(\zeta^{xu})_{x,u\in\F_{11}}$, so that $FF^*=11I$.
The character matrix has blocks
\[
 H_{st}=D_{t f_s}FD_{s g_t},
\]
where the diagonal factors have entries that are powers of $\zeta$.
For distinct $s,s'$, block-row orthogonality gives
\[
 \sum_{t\in T}U_t=0,\qquad U_t=11^{-1}H_{st}H_{s't}^*.
\]
The three $11\times11$ matrices $U_t$ are unitary over
$K=\mathbb Q(\zeta)$. Multiplication by the adjoint of one of them
gives $I+X+Y=0$, where $X,Y$ are unitary over $K$.
It follows that $X+X^*=-I$ and $X^2+X+I=0$.
But $t^2+t+1$ is irreducible over $K$, since
$[\mathbb Q(\zeta_{33}):\mathbb Q(\zeta_{11})]=2$.
A vector space admitting such an operator has even dimension over $K$,
contrary to its dimension eleven.
\end{proof}

\begin{lemma}\label{p11:33-plane8}
Every affine plane contains at most eight points of either member.
Moreover every projective line meets either $\U_E$ or $\U_A$ in at most
eight points.
\end{lemma}
\begin{proof}
A plane with ten points determines at least ten directions by
Lemma~\ref{p11:ten-point}. The parallel-plane criterion would put the
partner in three parallel planes, contradicting
Lemma~\ref{p11:33-layers}. Thus the plane bound improves to nine.

No projective line can lie in $\U_E$. Indeed, the twelve planes in the
corresponding pencil, through a line containing $a$ points of $E$,
would give $33+11a=12\cdot3$, which is impossible. If a projective
line contained nine points of $\U_E$, the restricted moment polynomials
of degrees one through eight would vanish identically. The eleven
plane counts in every direction on this line would consequently be
represented modulo eleven by an affine polynomial. A nonconstant affine
polynomial has least residues summing to $55>33$. A constant polynomial,
with all actual counts between zero and nine, gives the constant profile
three. Thus the entire projective line would be uniform, again a
contradiction. Hence the uniform-direction bound is eight.

If a projective line contains eight uniform directions, its plane
profiles are quadratic modulo eleven. Among the $110$ centered
quadratics $at^2+c$, $a\ne0$, the profiles whose least residues have
sum $33$ and maximum at most nine are exactly
\begin{equation}\label{p11:33-quadratic}
 [a(t-h)^2+5]_{11},\qquad a\in\{2,6,7,8,10\},
\end{equation}
where brackets denote the least nonnegative residue. Their common
multiset is
\[
 \{0,0,1,1,2,2,4,4,5,7,7\}.
\]
This follows either by substituting the five nonzero quadratic residues,
or from the univariate check in \path{verify_small_profiles.py}.
The four nonuniform normals, represented by forms $l_i$, therefore
satisfy $\sum_{i=1}^4 a_i l_i^2=0$, with all $a_i$ nonsquares.
Normalize the forms to $x,y,x+y,x+\lambda y$. Their coefficients,
up to a common nonzero scalar $d$, are
\[
 ((\lambda-1)d,\ \lambda(1-\lambda)d,\ -\lambda d,\ d).
\]
The nonsquare conditions imply that $\lambda$ is a nonsquare and
$\lambda-1$ is a square. Thus $\lambda\in\{2,6,10\}$, the harmonic
orbit.

If $A$ had nine points in a plane, those points would determine at least
eight directions by Lemma~\ref{p11:33-small-planar}, and at most eight
by the bound just proved. The four missing directions would be harmonic,
contradicting the second part of that lemma. Interchanging $E,A$ proves
the assertions.
\end{proof}

\begin{lemma}\label{p11:33-seven-gap}
A projective line cannot meet $\U_E$ or $\U_A$ in exactly seven points.
\end{lemma}
\begin{proof}
Project along the annihilator of such a line and write $f$ for the
multiplicity function on $\F_{11}^2$. Its mass is $33$, and every
line has $f$-mass at most eight. Make the vertical direction uniform,
and write the five nonuniform profiles as $g_i(y-s_i x)$, with the
$s_i$ distinct. The Radon identity is
\begin{equation}\label{p11:33-five-ridges}
 \sum_{i=1}^5g_i(y-s_i x)=12+11f(x,y).
\end{equation}
The Vandermonde equations of Lemma~\ref{p11:radon} show that each $g_i$
has degree at most three modulo eleven. None is constant, since that
would give another uniform direction, and none is nonconstant affine,
since such a profile takes the value ten.

Put $P(X)=\prod_i(X-s_i)$ and $w_i=P'(s_i)^{-1}$. If one profile is
cubic, all five leading coefficients are $\lambda w_i$, with
$\lambda\ne0$, and their quadratic coefficients are $w_i(a+bs_i)$.
A translation makes these quadratic coefficients zero. The only monic
centered cubics whose residues lie in $\{0,\ldots,8\}$ and sum to
$33$ are
\[
 t^3+t+3,\qquad t^3+9t+1.
\]
This is a check of the $121$ pairs of remaining coefficients.
Since cubing permutes $\F_{11}^{\times}$, rescaling the argument
reduces any centered cubic to the monic case without changing its
constant coefficient. The five constant coefficients therefore sum to
$5+2k$, $0\leq k\leq5$, never to $12\equiv1\pmod{11}$.
This contradicts \eqref{p11:33-five-ridges}.

Otherwise all five profiles are quadratic of the form
\eqref{p11:33-quadratic}. Their quadratic and linear coefficients have
forms
\[
 q_i=w_iQ(s_i),\quad \ell_i=w_iR(s_i),\qquad
 Q(X)=a+bX,\quad \deg R\leq2.
\]
Every $q_i$ is a nonsquare. The constant coefficient of the $i$th profile
is $5+\ell_i^2/(4q_i)$, so the constant equation requires
\begin{equation}\label{p11:33-square-contradiction}
 \sum_i\frac{w_iR(s_i)^2}{Q(s_i)}=3.
\end{equation}
The product $\prod_iw_i$ is a square. If $b=0$, the product of the
five $q_i=aw_i$ shows that $a$ is a nonsquare. The Vandermonde
identities then make the left side $r_2^2/a$, where
$r_2=[X^2]R$. If $b\ne0$, put $r=-a/b$. Since
$\prod_i(s_i-r)=-P(r)$ and $-1$ is a nonsquare, the product of the
five $q_i$ shows that $bP(r)$ is a square. Polynomial division and
$\sum_iw_i/(s_i-r)=-1/P(r)$ make the left side
$-R(r)^2/(bP(r))$. In either case it is zero or a nonsquare, whereas
three is a nonzero square. This proves the lemma. The small cubic and
quadratic checks are reproduced by \path{verify_seven_lemma.py}.
\end{proof}

\subsection{Forcing a planar projection}
\begin{lemma}\label{p11:33-six-triples}
Let $B\subset\F_{11}^3$ have at most eight points in every affine
plane. If every projective line contains at most six directions of
$\D(B)$, then $B$ does not contain six parallel three-point lines.
\end{lemma}
\begin{proof}
Suppose six such lines have direction $d$. Two of their parameter sets
$T_i,T_j\subset\F_{11}$ give $|T_j-T_i|$ cross directions, in
addition to $d$, on one projective line. Hence $|T_j-T_i|\leq5$.
The three-element case of Cauchy--Davenport~\cite{Davenport}
and its equality statement~\cite{Vosper}
imply that both triples are arithmetic progressions with a common step.
Here the equality statement can be checked directly on all
$\binom{11}{3}^2=27,225$ ordered pairs; precisely $605$ have difference
set of size five, and all have the asserted form. Thus the six triples
may be written $b_i+\{-v,0,v\}$ with one nonzero $v$.

The six projected centers in $\F_{11}^3/\langle v\rangle$ form a
cap, since three collinear centers would put nine points in a plane.
If two parallel projected chords lift to $w$ and
$w'=\lambda w+\mu v$, their two sets of five cross directions,
together with $[v]$, force
\[
 \{(\mu+k)/\lambda:k=-2,-1,0,1,2\}
   =\{-2,-1,0,1,2\}.
\]
Taking the sum and then the sum of squares gives $\mu=0$ and
$\lambda^2=1$. Thus parallel chords of the planar cap, and their
lifts, are equal up to the same sign.

There are four affine classes of planar six-caps with this property:
\[
\begin{array}{llllll}
(0,0)&(0,1)&(1,0)&(1,1)&(2,4)&(2,5),\\
(0,0)&(0,1)&(1,0)&(1,1)&(2,4)&(3,5),\\
(0,0)&(0,1)&(1,0)&(1,1)&(2,4)&(10,8),\\
(0,0)&(0,1)&(1,0)&(1,1)&(2,5)&(2,6).
\end{array}
\]
For completeness, normalize a noncollinear triple, extend it in
increasing order while imposing the cap and equal-chord conditions, and
identify sets using all noncollinear affine frames. The resulting
$890$ search states give $300$ normalized sets and the four displayed
classes; \path{six_equal_chords.py} implements this enumeration.
Write the lifts of a representative as $(h_i,y_i,z_i)$ with
$v=(1,0,0)$. An affine shear makes the first three heights zero.
The equal-chord equations leave exactly eleven height vectors in each
class. Direct enumeration of their difference directions shows a
projective-line intersection of size at least $10,11,9,10$, respectively.
Thus every one of the $44$ possible lifts contradicts the bound six.
The linear equations and the witness lines are checked by
\path{six_triple_skeletons.py}; the three-set difference calculation
is \path{verify_three_set_difference.py}.
\end{proof}

\begin{proposition}\label{p11:33-forced-base}
After an affine change of coordinates, each member is a graph over
\begin{equation}\label{p11:33-base}
 B_{33}=\{(x,y)\in\F_{11}^2:S(x,y)=20\},
\end{equation}
where
\[
 S(x,y)=[5+2x^2]_{11}+[5+7y^2]_{11}
       +[5+7(x+y)^2]_{11}+[5+2(x+2y)^2]_{11}.
\]
This set has $33$ points, contains zero, and satisfies $B_{33}=-B_{33}$.
\end{proposition}
\begin{proof}
Suppose $\U_A$ has no projective-line intersection of size eight.
Lemmas~\ref{p11:33-plane8} and \ref{p11:33-seven-gap} bound all such
intersections by six, and therefore do the same for $D=\D(E)$. Put
$N=|D|$ and $n_L=|D\cap L|$. If every $n_L\geq3$, then
\[
 \sum_L(n_L-3)(n_L-6)\leq0.
\]
The projective incidence identities
$\sum_Ln_L=12N$ and $\sum_Ln_L^2=N^2+11N$ instead give
\[
 \sum_L(n_L-3)(n_L-6)=N^2-97N+2394>0,
\]
since the discriminant is $-167$. Thus some $n_L\leq2$.
Each of the eleven affine planes with direction plane $L$ has at most
three points of $E$: a noncollinear triple would give three directions,
and four collinear points violate the line bound. All eleven planes
therefore contain exactly three collinear points. At least six of these
triples are parallel, contrary to Lemma~\ref{p11:33-six-triples}.
Thus $\U_A$ has an eight-point line, and similarly so does $\U_E$.

Consider the projection multiplicity $f$ associated to such a line in
$\U_E$. Its eight uniform profiles have squared mass $99$ each, and the
four profiles in \eqref{p11:33-quadratic} have squared mass $165$ each.
The squared Radon identity gives
\[
 8\cdot99+4\cdot165=33^2+11\sum_x f(x)^2,
 \qquad \sum_xf(x)^2=33=\sum_xf(x).
\]
Since $f$ is nonnegative and integer-valued, it takes only values zero
and one; the projection is injective.

Normalize the four nonuniform forms to $x,y,x+y,x+2y$. Their
quadratic coefficients are proportional to $(2,7,7,2)$, by a square
scalar, so a common dilation makes them exactly these coefficients.
Their linear coefficient relation says
$\sum_i a_i h_i l_i=0$, where $h_i$ are their centers. This system has
rank two. The two-dimensional family $h_i=l_i(x_0,y_0)$ lies in its
kernel, by $\sum_i a_i l_i^2=0$, and therefore is the whole kernel.
A translation makes all centers zero. The unsquared Radon identity
then reads $S=9+11f$, giving \eqref{p11:33-base}. Substitution at the
$121$ points shows that $S$ takes the values nine and twenty and that
twenty occurs $33$ times. Its expression is even, and $S(0,0)=20$.
The same argument applies to $A$.
\end{proof}

\subsection{Fifth moments and heights}
\begin{lemma}\label{p11:33-rigidity}
Let $B\subset\F_{11}^3$ have $33$ points, line intersections of size
at most three, plane intersections of size at most eight, and at most
eight directions of $\D(B)$ on every projective line. No nonzero
homogeneous polynomial of degree at most five vanishes on $\D(B)$.
\end{lemma}
\begin{proof}
Multiplication by a power of a nonzero linear form reduces to degree
five. Suppose $F$ is a polynomial of that degree vanishing on $\D(B)$,
and put $N=|\D(B)|$, $n_L=|L\cap \D(B)|$.
If every $n_L\geq3$, a projective line not dividing $F$ has at most
five zeros. There are at most five distinct linear factors, and each
corresponding line has at most eight directions. Hence
\[
 \sum_L(n_L-3)(n_L-5)\leq5(8-3)(8-5)=75.
\]
But the incidence identities make this sum
\[
 N^2-85N+1995\geq755/4>75.
\]

Otherwise a projective line meets $\D(B)$ in at most two points. As in
the preceding proof, the corresponding eleven affine planes each
contain three collinear points, with at least six parallel triples.
Their projected positions form a planar six-cap and determine at least
six directions by Lemma~\ref{p11:six-point}. Each pair of triples
contributes at least five cross directions by Cauchy--Davenport, in
addition to their common direction. Thus each of the six quotient
directions gives a different projective line with at least six zeros
of $F$. All six lines must divide $F$, impossible in degree five.
\end{proof}

\begin{corollary}\label{p11:33-moments}
Both members of a size-$33$ spectral pair have all mixed moments of
total degree at most five equal to zero in $\F_{11}$.
\end{corollary}
\begin{proof}
For $1\leq j\leq5$, uniformity gives
\[
 M_{j,E}(v)=\sum_{e\in E}(v\cdot e)^j
          =3\sum_{t\in\F_{11}}t^j=0\qquad(v\in \D(A)).
\]
Lemma~\ref{p11:33-rigidity} makes $M_{j,E}$ identically zero.
All its multinomial coefficients are nonzero modulo eleven, so its
coefficients are exactly the required vanishing mixed moments, up to
nonzero factors. The degree-zero moment is $33=0$ in $\F_{11}$.
Interchange the two members.
\end{proof}

\begin{lemma}[The finite height calculation]\label{p11:33-height}
If $h:B_{33}\to\F_{11}$ satisfies
\begin{equation}\label{p11:33-height-moments}
 \sum_{(x,y)\in B_{33}}x^ay^bh(x,y)^k=0
 \qquad(a,b,k\geq0,\ a+b+k\leq5),
\end{equation}
then subtraction of an affine function makes $h$ odd.
\end{lemma}
\begin{proof}
Subtract $h(0,0)$, and write $h=e+o$ with $e$ even and $o$ odd.
Subtract a linear function to make $o(1,0)=o(0,4)=0$.
These operations preserve \eqref{p11:33-height-moments} by the binomial
formula. In particular $e(0,0)=o(0,0)=0$. Representatives of the sixteen
nonzero antipodal pairs of $B_{33}$ are
\[
\begin{split}
 R=\{&(0,4),(0,5),(1,0),(1,3),(1,4),(1,5),(1,6),(1,7),\\
     &(1,10),(2,4),(2,5),(3,4),(4,0),(4,3),(4,4),(4,7)\}.
\end{split}
\]
The equations with $k=1$ split into
\[
 \sum_R x^ay^be=0\quad(a+b=0,2,4),\qquad
 \sum_R x^ay^bo=0\quad(a+b=1,3).
\]
The even matrix has rank nine, leaving seven variables. The odd matrix
has rank six, and its two normalization equations increase its rank to
eight, leaving eight variables. These ranks and nullspaces are computed
by Gaussian elimination modulo eleven.

The ten equations with $k=2$ are four equations
\begin{equation}\label{p11:33-even-quadratics}
 \sum_Rw(e^2+o^2)=0,\qquad w\in\{1,x^2,xy,y^2\},
\end{equation}
and six bilinear equations
\begin{equation}\label{p11:33-mixed-quadratics}
 \sum_Rweo=0,\qquad w\in\{x,y,x^3,x^2y,xy^2,y^3\}.
\end{equation}
Suppose $e\ne0$. Scaling the entire height makes its first nonzero
even coefficient one, giving exactly $(11^7-1)/10=1,948,717$ vectors.
For each of them, solve \eqref{p11:33-mixed-quadratics} in the eight
odd variables. The rank census is
\[
\begin{array}{c|r|c}
\text{rank}&\text{even vectors}&\text{odd nullity}\\\hline
4&187&4\\
5&3,052&3\\
6&1,945,478&2.
\end{array}
\]
Enumerating every vector in these kernels gives
\[
 187\cdot11^4+3,052\cdot11^3+1,945,478\cdot11^2
 =242,202,917
\]
pairs. Exactly $21,611$ satisfy \eqref{p11:33-even-quadratics}.
The remaining equations are
\[
\begin{array}{ll}
 \sum_Rw(e^3+3eo^2)=0,&w\in\{1,x^2,xy,y^2\},\\
 \sum_Rw(3e^2o+o^3)=0,&w\in\{x,y\},\\
 \sum_R(e^4+6e^2o^2+o^4)=0,&\\
 \sum_Rw(e^3o+eo^3)=0,&w\in\{x,y\},\\
 \sum_R(e^5+10e^3o^2+5eo^4)=0.&
\end{array}
\]
Exactly four pairs satisfy the cubic equations, and all four satisfy
the quartic equations. All have $o=0$. In the even basis generated by
\path{prepare_split.py}, their coefficients and final moments are
\[
\begin{array}{c|c}
\text{even coefficients}&\sum_{B_{33}}h^5\\\hline
(1,3,4,4,6,6,4)&8\\
(1,6,2,8,2,2,8)&3\\
(1,7,10,8,5,1,2)&3\\
(1,10,2,7,9,1,6)&3.
\end{array}
\]
Thus none satisfies the fifth-moment equation.

The source generators and exhaustive checkers are supplied in the
\path{fuglede11_round3/algebra} and \path{quadrics} directories.
The second implementation enumerates every odd kernel vector and
evaluates moments directly on $(x,y,e+o)$ and $(-x,-y,e-o)$; it gives
the same quadratic count and no solutions. This avoids the first
implementation's shortcut for a two-dimensional kernel. The normalization
is exhaustive because every nonzero even vector has a unique projective
representative, and the odd part is scaled simultaneously. Every odd
kernel vector is tested. The equations displayed above exhaust all
$k\geq2$ moments of total degree at most five. Hence $e\ne0$ is
impossible, and the normalized height is odd.
\end{proof}

\subsection{A determinant obstruction}
\begin{lemma}\label{p11:33-determinant}
Let $p$ be an odd prime. If a centrally symmetric spectral set in
$\F_p^d$ has odd cardinality $N$, then
\[
 \sqrt{(-1)^{(N-1)/2}N}\in\mathbb Q(\zeta_p).
\]
In particular a centrally symmetric spectral set cannot have cardinality
$33$ in any $\F_{11}^d$.
\end{lemma}
\begin{proof}
Translate the center to zero. Then $E=-E$, and zero belongs to $E$
because $N$ is odd. For a spectrum $A$, put
$H=(\zeta_p^{e\cdot a})_{e\in E,a\in A}$ and
$\Delta=\det H$. Complex conjugation permutes the rows by negation,
a permutation of sign $\varepsilon=(-1)^{(N-1)/2}$. Thus
\[
 \overline\Delta=\varepsilon\Delta,\qquad
 \Delta\overline\Delta=N^N,\qquad
 \Delta^2=\varepsilon N^N.
\]
Consequently $\Delta/N^{(N-1)/2}$ is the asserted square root in
$\mathbb Q(\zeta_p)$. When $p=11,N=33$, this would put
$\sqrt{33}$ in the real cyclotomic subfield, whose degree over
$\mathbb Q$ is five. That field has no quadratic subfield.
\end{proof}

\begin{corollary}\label{p11:33-square-m}
If $1\leq m<p$ is odd, a centrally symmetric spectral set of cardinality
$mp$ in $\F_p^d$ can exist only if $m$ is a square.
\end{corollary}
\begin{proof}
The square root furnished by Lemma~\ref{p11:33-determinant} is irrational,
since $p\nmid m$. The unique quadratic subfield of
$\mathbb Q(\zeta_p)$ is
\[
 \mathbb Q(\sqrt{(-1)^{(p-1)/2}p}).
\]
Equality of these fields forces
\[
 (-1)^{(mp-1)/2-(p-1)/2}m
\]
to be a rational square.
Its sign is positive, and therefore $m$ is an integer square.
\end{proof}

\begin{proposition}\label{p11:33-excluded}
There is no spectral subset of $\F_{11}^3$ of cardinality $33$.
\end{proposition}
\begin{proof}
Proposition~\ref{p11:33-forced-base} writes such a set as a graph over
$B_{33}$. Corollary~\ref{p11:33-moments} supplies its moments through
degree five, and Lemma~\ref{p11:33-height} makes its height odd after
an affine change. Since $B_{33}=-B_{33}$, the graph is then centrally
symmetric. Affine changes preserve spectrality, contradicting
Lemma~\ref{p11:33-determinant}.
\end{proof}

The determinant argument is useful here because it makes further
classification of the odd heights unnecessary. The calculation needs
only to exclude nonzero even parts. Thus no enumeration of the
$21,435,888$ projective odd coefficient vectors, or of their plane
intersections, is required in this proof.


\section{\texorpdfstring{Cardinality $44$: moments and a plane criterion}{Cardinality 44: moments and a plane criterion}}
For a spectral pair $(E,A)$ of size $44$, the results of
Fallon--Mayeli--Villano give plane intersections at most $11$, line intersections
at most $4$, and no full projective line in either direction set $\D(E),\D(A)$.
Every direction in $\D(A)$ is a normal direction in which $E$ is equidistributed,
with four points in each of eleven planes; the same holds with $E,A$ reversed.
Here $\D(B)$ denotes the projectivized nonzero differences of $B$.

\begin{lemma}[Finite planar inputs]\label{p11:planar44}
Over $\F_{11}$:
\begin{enumerate}
\item A noncollinear eight-point planar set determines at least seven directions.
\item A seven-point planar set with line cap four and containing four collinear
points determines at least seven directions.
\item An eleven-point planar set with line cap four determines at least nine directions.
\end{enumerate}
\end{lemma}
The first two assertions are certified by \path{code/prime11/original/output/fuglede11_next/verify44_planar.py}, which
uses \path{generic_rigidity_checks.py} in the same directory. Normalize a noncollinear triple to
$(0,0),(1,0),(0,1)$. A set using at most six directions extends to one of the
$\binom93=84$ six-direction masks containing $0,-1,\infty$. Enumerate all
compatible cliques containing that triple. A noncollinear set with six
directions cannot contain six collinear points, so line cap five loses no cases
in the first assertion. In the second enumeration all $180$ normalized
seven-cliques with line cap four actually have line cap three.

For the third assertion, a set with at least two missing directions is affinely
equivalent to the graph of a permutation. Normalize $f(0)=0,f(1)=1$ and exhaust
the $9!$ permutations. Every line-cap-four graph determines at least nine
directions. With fewer than two missing directions the assertion is immediate.
The permutation enumeration is recorded in the accompanying layer certificates.

\subsection{Degree-six rigidity and moments}
\begin{theorem}\label{p11:rigidity44}
Let $B\subset\F_{11}^3$ have size $44$, plane cap $11$, line cap $4$, and suppose
$\D(B)$ contains no complete projective line. A homogeneous polynomial of degree
at most six vanishing on $\D(B)$ is identically zero.
\end{theorem}
\begin{proof}
Multiplying a lower-degree counterexample by a power of a nonzero linear form
reduces to degree six. Suppose $F$ is such a nonzero sextic.

First suppose $B$ has no four collinear points. Every parallel family of planes
has a slice containing at least four points. These four are noncollinear and
determine at least four directions. Thus $D=\D(B)$ meets every projective line
in between four and eleven points. Put $N=|D|$ and $n_L=|D\cap L|$.
If $k\le6$ is the number of distinct rational line factors of $F$, a nonfactor
line has $n_L\le6$, whence
\[
 \sum_L(n_L-4)(n_L-6)\le35k\le210.
\]
The projective-plane identities $\sum n_L=12N$ and
$\sum n_L^2=N^2+11N$ give instead
\[
 \sum_L(n_L-4)(n_L-6)=N^2-109N+3192\ge887/4>210.
\]

It remains to consider a four-point affine line $\ell$, of direction $d$.
Let $k_d$ count rational line factors of $F$ whose associated linear planes
contain $d$. Every nonfactor affine plane through $\ell$ has at most six
points of $B$, by Lemma~\ref{p11:planar44}(2). The twelve-plane pencil therefore gives
\[
 44+11\cdot4=88\le11k_d+6(12-k_d)=72+5k_d,
\]
so $k_d\ge4$.

For $x\in B\setminus\ell$, the four differences $x-y$, $y\in B\cap\ell$,
give four distinct projective points on a projective line $L_x$ through $d$,
all distinct from $d$. Unless $L_x$ is itself a factor, the restriction
$F|_{L_x}$ has a zero of multiplicity at least $k_d$ at $d$, and the four
additional zeros just described. This exceeds degree six. Thus $L_x$ is a
factor through $d$.

Consequently $B$ is covered by $k_d$ planes through $\ell$, and
\[
 44\le4+7k_d.
\]
Since $k_d\le6$, this forces $k_d=6$: all six factors are distinct and pass
through $d$. Project $B$ along $d$. The image has at least eleven points,
is noncollinear (otherwise $B$ lies in one plane), and determines at most six
directions. This contradicts Lemma~\ref{p11:planar44}(1).
\end{proof}

\begin{proposition}\label{p11:moments44}
If $(E,A)$ is spectral of size $44$, then for each multi-index $\alpha$ with
$0\le|\alpha|\le6$,
\[
 \sum_{e\in E}e^\alpha=\sum_{a\in A}a^\alpha=0\quad\hbox{in }\F_{11}.
\]
\end{proposition}
\begin{proof}
For $1\le j\le6$, the homogeneous moment polynomial
$M_j(v)=\sum_{e\in E}(v\cdot e)^j$ vanishes on $\D(A)$ by equidistribution
and $\sum_{t\in\F_{11}}t^j=0$. Apply Theorem~\ref{p11:rigidity44}.
All relevant multinomial coefficients are nonzero modulo $11$, so its
coefficients give every mixed moment. The degree-zero assertion is $44=0$
in the field. Apply symmetry.
\end{proof}

\subsection{An improved parallel-plane criterion}
\begin{lemma}\label{p11:radon44}
Let $f:\F_{11}^2\to\mathbb Z_{\ge0}$ have total mass $44$, and suppose every
affine line has $f$-mass at most $11$. If its nonzero Fourier support is
contained in at most three one-dimensional subspaces, then $f$ is the
indicator of four parallel lines.
\end{lemma}
\begin{proof}
Enlarge the support directions to three distinct directions if necessary.
Normalize their defining linear forms to $x,y,x+y$, and let $n_1,n_2,n_3$
denote the corresponding line-mass profiles. Fourier inversion, or averaging
over the three line families, gives
\[
 11f(x,y)=n_1(x)+n_2(y)+n_3(x+y)-8.
\]
Reduction modulo $11$ shows that each $n_i$ is affine as a function
$\F_{11}\to\F_{11}$: subtract the identity at $(x,y)$ and $(x+h,y)$ and
vary $y$ to see that each first difference of $n_3$ is constant.

A nonconstant affine profile runs through all residues and has integer mass
at least $0+1+\cdots+10=55$, contradicting its total $44$. Thus each residue
profile is constant. Since $0\le n_i\le11$, its actual profile is either
eleven fours or four elevens and seven zeros. The displayed identity modulo
$11$ forces exactly two uniform profiles and one nonuniform profile.
Substitution now gives $f=1$ on the four lines of the latter family and zero
elsewhere.
\end{proof}

\begin{proposition}\label{p11:fullplane44}
If one member of a size-$44$ spectral pair has eleven points in a plane, the
other member lies in four parallel planes, each containing eleven points.
\end{proposition}
\begin{proof}
The eleven-point slice determines at least nine directions by
Lemma~\ref{p11:planar44}(3). Project the partner along the perpendicular direction.
Its integer multiplicity function has mass $44$, line masses at most $11$,
and Fourier support on at most three of the twelve normal directions.
Lemma~\ref{p11:radon44} applies. Pulling back the four lines gives the assertion.
\end{proof}

\subsection{The planar slice classification}

We use the following established inputs for a hypothetical spectral pair
$(E,A)$ of size $44$ in $\F_{11}^3$: every affine plane contains at
most eleven points of either member, every line at most four, the direction
set contains no complete projective line, all moments of total degree at
most six vanish, and no nonzero homogeneous polynomial of degree at most
six vanishes on the Fourier-zero directions of $E$.

Suppose that $E$ is contained in four parallel planes. Write their
coordinates as $z=s$, $s\in S$, where $|S|=4$. Each slice has eleven points.
A common missing direction in the layer plane puts all slices in graph
form
\[
 B_s=\{(x,f_s(x)):x\in\F_{11}\}.
\]
Lagrange interpolation in $z$, of degree three, shows that every $B_s$
has all moments of total degree at most three equal to zero.

\begin{lemma}[The moment-flat graph classification]\label{p11:flat-graphs}
Let $B=\{(x,f(x)):x\in\F_{11}\}$ have all moments through degree
three equal to zero and at most four points on a line. After vertical
translation, shear and scaling, there are $254$ possible polynomials,
all with at most three points on a line. Under horizontal affine changes
and these vertical operations, they form six graph orbits, of sizes
$1,11,11,11,110,110$.

Put $Q_B(u,v)=\sum_B(ux+vy)^4$ and $R_B(u,v)=\sum_B(ux+vy)^5$.
The six fourth-moment types are respectively zero, $(3,1)$, $(3,1)$,
$(2,2)$, $(2,1,1)$, and $(1,1,1,1)$, where a tuple records the
multiplicities of the rational linear factors. The zero type consists
of parabolas, the next two of affine cubic graphs, and the fourth of
affine quartic graphs.

For each of the last three types, fix a representative $B_0$ and write
$Q_0=Q_{B_0}$. Every admissible graph with $Q_B=kQ_0$, $k\ne0$, is
of the form $B=\rho B_0+v$, with $k=\rho^4$. Moreover
$R_{B_0}\notin Q_0\langle u,v\rangle$.
\end{lemma}
\begin{proof}[Exact finite classification]
For a graph $B=\{(x,f(x))\}$ with vanishing moments through degree three,
apply a vertical translation and a shear to arrange $f(0)=f(1)=0$.
Writing $f$ as its unique polynomial of degree at most ten, the equations
$\sum_xx^if(x)=0$, $i=0,1,2$, give respectively $c_{10}=c_9=c_8=0$.
Thus
\[
 f(x)=\sum_{j=2}^7c_j(x^j-x).
\]
The zero polynomial is a line and is inadmissible. Normalize the first
nonzero coefficient $c_2,\ldots,c_7$ to one. This leaves exactly
$(11^6-1)/10=177156$ polynomials. The remaining equations are
\[
 \sum_xf(x)^2=\sum_xxf(x)^2=\sum_xf(x)^3=0.
\]
Exactly $375$ normalized graphs satisfy them. Their line bounds and
numbers of determined directions are
\[
\begin{array}{c|rrrrr}
\text{maximum line intersection}&2&3&3&5&6\\
\text{directions}&11&10&11&8&7\\
\text{count}&1&242&11&110&11.
\end{array}
\]
Consequently the line bound four leaves $254$ graphs, all with line bound
at most three.

Put $Q_B(u,v)=\sum_{(x,y)\in B}(ux+vy)^4$ and define $R_B$ similarly
in degree five. The same checker partitions the $254$ graphs under
horizontal affine changes, vertical scaling, shear and translation.
There are six orbits, of sizes $1,11,11,11,110,110$. The size-one orbit
is the parabola; two size-eleven orbits consist of affine images of the
cubic curve $(x,x^3)$, with the two possible roles of the missing direction;
the third consists of affine images of $(x,x^4)$. The rational-root
multiplicity patterns of $Q_B$ are respectively
\[
 0,\qquad(3,1),\ (3,1),\qquad(2,2),\qquad(2,1,1),\qquad(1,1,1,1).
\]
All these identifications, including the orbit coverage, are checked
exactly in the supplied program.

We need one further finite statement for the quartic orbit and the two
size-$110$ orbits. Fix any of their representatives $B_0$, and put
$Q_0=Q_{B_0}$. Among all $254$ normalized graphs, allowing arbitrary
vertical scaling, shear and translation, those satisfying
$Q_B=kQ_0$ are precisely
\[
 B=\rho B_0+v_0,\qquad \rho\in\F_{11}^{\times},\quad
 k=\rho^4.
\]
Moreover $R_{B_0}$ is not a multiple of $Q_0$ by a linear form.
The checker compares the two finite sets directly: after removing the
vertical translation, there are $22$ graphs for each nonzero square
$k$, and none for a nonsquare. It also checks the nonmembership of
$R_{B_0}$ in the two-dimensional space $Q_0\langle u,v\rangle$.

These coefficient enumerations, orbit calculations, and equalities of
finite lists are implemented by \path{verify_layer_obstructions.py} in
\path{code/prime11/original/output/fuglede11_next/}. The stated
translations, shears and nonzero rescaling exhaust the normalizations.
All polynomial coefficients and all remaining vertical scalings and
shears are enumerated. The program checks the moments, line intersections,
orbit membership, and both inclusions in the homothety assertion by exact
field arithmetic.
\end{proof}
\subsection{Excluding four parallel layers}
\begin{lemma}[Common fourth moments]\label{p11:44-quartics}
Under the four-layer assumption above, the slice fourth moments satisfy
$Q_{B_s}=\lambda_sQ$ for one binary quartic $Q$ and the weights below.
\end{lemma}
\begin{proof}
Define
\[
 \lambda_s=\frac1{\prod_{r\in S\setminus\{s\}}(s-r)}.
\]
The global moments give $\sum_s s^jQ_{B_s}=0$ for $j=0,1,2$.
The nullspace of this three-by-four Vandermonde matrix is generated by
$(\lambda_s)_{s\in S}$. Therefore $Q_{B_s}=\lambda_s Q$ for one binary
quartic $Q$. Since every $\lambda_s$ is nonzero, either all slice quartics
vanish, or they all have the same root multiplicity pattern.

\end{proof}

\begin{lemma}\label{p11:44-noncubic}
The common fourth moment cannot have root pattern $(2,2)$, $(2,1,1)$,
or $(1,1,1,1)$.
\end{lemma}
\begin{proof}
Suppose the common pattern is $(2,2)$, $(2,1,1)$ or $(1,1,1,1)$.
Using the finite classification, make a single affine change in the layer
coordinates so that one slice is a representative $B_0$. The classification
then gives
\[
 B_s=\rho_s B_0+v_s,\qquad
 \rho_s^4=\lambda_s/\lambda_{s_0}.
\]
All ratios of the $\lambda_s$ are consequently squares.
Because moments through degree three vanish, translation changes the
degree-five moment only by a linear form times the fourth moment.
In the quotient of binary quintics by $Q_0\langle u,v\rangle$, this gives
\[
 [R_{B_s}]=\rho_s^5[R_{B_0}].
\]
The class on the right is nonzero. The degree-five and degree-six global
moments therefore imply
\[
 \sum_s\rho_s^5=0,\qquad\sum_s s\rho_s^5=0.
\]
Each $\rho_s^5$ is $+1$ or $-1$. Four such signs sum to zero in
$\F_{11}$ only when two are positive and two negative. Their
corresponding pairs of layer coordinates have equal sums. Thus
$S=\{h\pm a,h\pm b\}$ is centrally symmetric. Its derivative weights
at opposite points satisfy
\[
 \lambda_{h+a}=-\lambda_{h-a}.
\]
This contradicts all weight ratios being squares, since $-1$ is not a
square in $\F_{11}$.

\end{proof}

\begin{lemma}\label{p11:44-cubic}
The common fourth moment cannot have root pattern $(3,1)$.
\end{lemma}
\begin{proof}
An affine image of $\{(x,x^3)\}$ has fourth moment equal to a nonzero
multiple of $L_1L_2^3$, with $L_1,L_2$ independent. The common fourth
moment fixes the simple and triple factors separately. A single linear
change of layer coordinates therefore puts every cubic slice in the form
\[
 B_s=\{(x,a_s(x-b_s)^3+c_s):x\in\F_{11}\},\qquad a_s\ne0.
\]
Indeed the linear parts of the affine parametrizations are diagonal
in the common pair of factor coordinates. No separate layer-dependent
linear change is used here.
Now
\[
 Q_{B_s}=-4a_s^3uv^3,
 \qquad\sum_{(x,y)\in B_s}x^4y^2=-a_s^2.
\]
The second identity follows from the leading coefficient of a degree-ten
polynomial; all lower power sums, including the constant sum, vanish in
$\F_{11}$. Thus $a_s^3=K\lambda_s$ and $\sum_sa_s^2=0$.
The inverse of cubing on $\F_{11}^{\times}$ is the seventh power,
so this requires $\sum_s\lambda_s^4=0$.
The checker examines all $\binom{11}{4}=330$ layer sets and finds none.
Equivalently one may normalize $S=\{0,1,a,b\}$ and check the $36$ pairs
$2\le a<b\le10$.

\end{proof}

\begin{lemma}\label{p11:44-parabolic}
The common fourth moment cannot be zero.
\end{lemma}
\begin{proof}
If $Q=0$, the classification says that all slices are parabolas. Their
common missing direction gives
$f_s(x)=a_sx^2+b_sx+c_s$, $a_s\ne0$. Write $\zeta=e^{2\pi i/11}$.
For a frequency $(u,t,w)$ with $t\ne0$, completing the square gives
\[
 \widehat{1_E}(u,t,w)
 =G\chi(t)\sum_{s\in S}\chi(a_s)
 \zeta^{\phi_s},\qquad
 \phi_s=ws+tc_s-\frac{(u+tb_s)^2}{4ta_s},
\]
up to the harmless Fourier-sign convention, where $G\ne0$ is the
quadratic Gauss sum and $\chi$ is the quadratic character.
A sum of four signed eleventh roots can vanish only by exact cancellation
of equal roots: the only rational relation between the eleven roots is
their total sum, and four terms cannot support a nonzero multiple of it.
Hence there must be two positive signs and two negative signs. Fix a
positive index $i$ and denote the two negative indices by $j,k$.
Every Fourier-zero frequency with $t\ne0$ satisfies
\[
 t(\phi_i-\phi_j)=0\quad\hbox{or}\quad t(\phi_i-\phi_k)=0.
\]
Each left side is a homogeneous quadratic polynomial, nonzero because
its $tw$ coefficient is respectively $s_i-s_j$ or $s_i-s_k$.
Thus the nonzero degree-five polynomial
\[
 t\,[t(\phi_i-\phi_j)]\,[t(\phi_i-\phi_k)]
\]
vanishes on every Fourier-zero direction, also including $t=0$.
If the signs are unbalanced, the linear polynomial $t$ already does so.
Both conclusions contradict the established degree-six direction
rigidity. All cases have been excluded, proving that a size-$44$
spectral set cannot lie in four parallel planes.

\end{proof}

\begin{corollary}\label{p11:44-cap10}
Every affine plane meets either member of a size-$44$ spectral pair in at
most ten points.
\end{corollary}
\begin{proof}
By the full-plane criterion established above, an eleven-point affine
plane intersection of either member forces its partner into four parallel
planes. The exclusion just proved therefore improves the plane bound to
\[
 \max_P|E\cap P|\le10,\qquad\max_P|A\cap P|\le10.
\]
\end{proof}

\begin{lemma}\label{p11:44-pencil8}
Every projective line contains at most eight uniform normals.
\end{lemma}
\begin{proof}
If a pencil contained nine, use Lemma~\ref{p11:radon} with three remaining
profiles. They have degree at most one modulo eleven. A nonconstant affine
profile has least-residue mass $55$, whereas each profile here has mass
$44$. They are constant, and the plane bound ten makes the residues their
actual values. Each profile is constantly four. The Radon identity would
then give $12\equiv8\pmod{11}$, a contradiction. Selecting nine normals
also covers a pencil with more than nine.
\end{proof}
\subsection{Seventh moments and the energy contradiction}
The following argument assumes the preceding cardinality-44 results:
every affine line has at most four points, every affine plane has at most
ten points, and every projective line contains at most eight Fourier-zero
directions. As usual, $\D(E)\subseteq\U(A)$ for a spectral pair $(E,A)$.

\begin{lemma}[Degree-seven rigidity]\label{p11:rigidity44seven}
Let $B\subset\F_{11}^3$ have 44 points, line cap four and plane cap
ten. Suppose every projective line meets $\D(B)$ in at most eight points.
No nonzero homogeneous polynomial of degree at most seven vanishes on
$\D(B)$.
\end{lemma}
\begin{proof}
Multiplying by a suitable power of a linear polynomial reduces the claim
to degree seven. Let $F$ be such a polynomial, $N=|\D(B)|$, and
$n_L=|L\cap\D(B)|$. First suppose $n_L\ge4$ for all 133 projective lines.
A line not dividing $F$ has $n_L\le7$. There are at most seven distinct
line factors, and each has $n_L\le8$. Consequently
\[
 \sum_L(n_L-4)(n_L-7)\le 7\cdot4=28.
\]
But the projective-plane incidence identities give
\[
 \sum_L(n_L-4)(n_L-7)
 =N^2-121N+3724\ge\frac{255}{4}>28,
\]
a contradiction.

It remains to consider a projective line $L$ with $n_L\le3$. The eleven
parallel affine planes with tangent directions $L$ each contain at most
four points of $B$: any four noncollinear points over an odd field determine
at least four directions, whereas five collinear points violate the line
cap. For the elementary four-point fact, a cap determining only three
directions would have two parallel pairs in each direction; its three
perfect matchings cannot all be parallel in odd characteristic. Since $|B|=44$, each plane contains exactly four collinear points.
Among these eleven four-point lines, at least four have a common direction
$d\in L$. No three of these parallel lines lie in an affine plane, since
that would give twelve points in a plane. Their positions in the quotient
by $d$ therefore include a four-point cap, determining at least four
directions.

For any two of these parallel four-point lines, Cauchy--Davenport gives
at least seven distinct cross-difference directions in their common
projective line, in addition to $d$. Such a projective line contains at
least eight zeros of $F$, so it divides $F$. It follows that at least four
distinct line factors of $F$ pass through $d$.

Fix one of the four-point affine lines $\ell$. For any $b\in B\setminus
\ell$, the four differences from $b$ to $B\cap\ell$ determine four
distinct points other than $d$ on the projective line $L_b$ through $d$.
If $L_b$ were not a factor of $F$, the restriction of $F$ to it would
have multiplicity at least four at $d$ and four further roots, impossible
in degree seven. Thus every point of $B$ lies in one of the $k$ factor
planes through $\ell$, where $4\le k\le7$. Plane cap ten yields
\[
 44\le4+(10-4)k,
\]
so $k=7$. All seven factors are therefore concurrent at $d$.

Project $B$ along $d$ to a planar set $C$. The line cap gives $|C|\ge11$,
and the factorization shows that $C$ determines at most seven directions.
Any twelve planar points over $\F_{11}$ determine every direction,
by pigeonhole projection, so $|C|=11$ and every fiber of $B\to C$ has
exactly four points. Plane cap ten now implies that $C$ is a cap: three
collinear quotient points would lift to twelve coplanar points. A planar
eleven-point cap determines at least eleven directions, since it has
55 unordered pairs and at most five parallel pairs per direction.
This is the final contradiction.
\end{proof}

\begin{corollary}\label{p11:moment44seven}
For a spectral pair of size 44 in $\F_{11}^3$, every mixed moment
of total degree at most seven vanishes in $\F_{11}$.
\end{corollary}
\begin{proof}
Equidistribution gives
$\sum_{e\in E}(v\cdot e)^7=4\sum_{t\in\F_{11}}t^7=0$
for each $v\in\D(A)$. Apply Lemma~\ref{p11:rigidity44seven}; multinomial
coefficients in total degree seven are nonzero modulo eleven. Combine
with the previously established moments through degree six, and exchange
$E$ and $A$.
\end{proof}

\begin{lemma}[A quadratic profile energy gap]\label{p11:44-energy}
Let $n:\F_{11}\to\{0,1,\ldots,10\}$ have total mass $44$,
and suppose its reduction modulo eleven is a polynomial of degree at
most two. Then either $n\equiv4$, or
\[
 q(n):=\sum_t\binom{n(t)}2\ge99.
\]
In every case $q(n)$ is divisible by eleven.
\end{lemma}
\begin{proof}
The sums of $n$ and $n^2$ vanish modulo eleven, since their polynomial
degrees are at most two and four. This gives $q(n)\equiv0\pmod{11}$.
A nonconstant affine polynomial has the residues $0,\ldots,10$ once
each and therefore mass $55$, so it cannot occur. A constant profile
of mass $44$ is identically four.

A nonconstant quadratic has one value $c$ attained once and five other
distinct values $b_1,\ldots,b_5$ attained twice. All these are ordinary
integers in $\{0,\ldots,10\}$. Write
\[
 d=c-4,\qquad a_i=b_i-4.
\]
The mass identity gives $d+2\sum_i a_i=0$, so $d$ is even,
$-4\le d\le6$, and the $a_i$ are five distinct integers different
from $d$. Also
\[
 2(q(n)-66)=d^2+2\sum_i a_i^2=:V.
\]
Suppose, for contradiction, that $V\le44$. Five distinct integers have
square sum at least ten, so $|d|\ge6$ would give $V\ge56$.
If $|d|=4$, then $\sum_i a_i^2\le14$. The only five-element integer
set with square sum below fifteen is $\{-2,-1,0,1,2\}$, whose sum
zero contradicts $\sum_i a_i=-d/2$.

If $d=0$, all five $a_i$ are nonzero and their square sum is at most
twenty-two. A term of magnitude at least four would force square sum
at least $16+1+1+4+4=26$. Thus the $a_i$ are five of
$\{-3,-2,-1,1,2,3\}$, which cannot sum to zero.

If $|d|=2$, reflect to $d=2$. The five distinct integers avoid two,
sum to minus one, and have square sum at most twenty. A term of
magnitude at least four would force square sum at least
$16+0+1+1+4=22$. Thus they are five of
$\{-3,-2,-1,0,1,3\}$. Their sum forces the omitted element to be
minus one, but the remaining square sum is twenty-three. This is the
last contradiction, proving $V>44$.

Consequently $q(n)>88$. Its divisibility by eleven now yields $q(n)\ge99$.
\end{proof}

\begin{theorem}\label{p11:exclude44}
There is no spectral pair of cardinality $44$ in $\F_{11}^3$.
\end{theorem}
\begin{proof}
The moments through degree seven imply that every normal occupancy
profile is quadratic modulo eleven. The plane bound ten makes its
least residues its actual values. Lemma~\ref{p11:44-energy} therefore gives
\[
 k_E(v)=\frac{q_E(v)-66}{11}\in\{0\}\cup\{3,4,\ldots\},
 \qquad k_E(v)=0\Longleftrightarrow v\in\U_E,
\]
and the same assertions for $A$. Every projective line meets the support
of $k_E$, and that of $k_A$, in at least four points, by
Lemma~\ref{p11:44-pencil8}. Apply Proposition~\ref{p11:weighted-incidence}
with $p=11,m=4,a=3,r=4$. It has $C=14$, $T=234$, and
$C<5\cdot3$. It would follow that $14|V_A|=4\cdot234=936$,
which is impossible.
\end{proof}


\section{\texorpdfstring{Excluding cardinality $55$}{Excluding cardinality 55}}
Let $(E,A)$ be a spectral pair in $\F_{11}^3$ of size $55$.
Write $D(B)$ for the projective directions determined by $B$ and $U_B$ for
the normal directions in which $B$ is equidistributed (five points per plane).
The published bounds are $|B\cap H|\le11$ for planes and $|B\cap\ell|\le5$
for lines, for either member $B$. Spectrality gives $D(A)\subset U_E$ and
$D(E)\subset U_A$.

\begin{lemma}[Finite planar and profile bounds]\label{p11:55-finite}
The following statements hold over $\F_{11}^2$:
\begin{enumerate}
\item Every noncollinear six-point set determines at least six directions.
\item Every eight-point set with line bound five determines at least seven directions.
\item Every nine-point set with line bound four determines at least eight directions.
\item Every eight-point set with line bound three determines at least eight directions.
\item Every eight-point set with line bound two determines at least ten directions.
\item A function $g:\F_{11}\to\{0,\ldots,7\}$ that is represented
modulo eleven by a polynomial of degree at most three and satisfies
$\sum_tg(t)=55$ is identically five.
\end{enumerate}
\end{lemma}
\begin{proof}
Items 2--5 are verified by \path{verify55_planar.js}; item 6 by
\path{verify55_profiles.js}. The planar searches normalize a noncollinear
triple to $(0,0),(1,0),(0,1)$ and exhaust the compatible cliques for every
allowed direction set. Their respective numbers of visited states are
$5770,39816,25725,20187$. The profile search checks all $11^4=14641$ cubics. Item~1 is
Lemma~\ref{p11:six-point}.
\end{proof}

An elementary auxiliary fact is that five noncollinear planar points determine
at least four directions. If four are collinear, use the fifth point. Otherwise
the line bound is three. Three directions cannot account for all ten pairs:
some direction would account for at least four pairs, requiring occupancy
$3+2$; its cross-directions number at least four by Cauchy--Davenport, in
addition to the parallel direction.

No $U_B$ contains a projective line. Indeed, the twelve affine planes through
any fixed affine line with $a$ points of $B$ have total occupancy $55+11a$.
If all twelve normals were in $U_B$, this would equal $60$, an impossibility.
In particular $|U_B|\le121$: the complement meets every projective line and
a projective blocking set has at least twelve points (count lines through an
external point).

\subsection{Moment and line bounds}
\begin{lemma}\label{p11:55-sixmom}
Every mixed moment of either member through degree six vanishes.
\end{lemma}
\begin{proof}
In fact, no nonzero homogeneous polynomial of degree at most six vanishes on
$D(B)$, for either $B=E$ or $A$.

It suffices to consider degree six, multiplying a lower-degree polynomial by
a power of a nonzero linear form if necessary. Suppose first that $B$ has a
five-point line $\ell$, of direction $d$. If the restriction of the polynomial
to a projective line $P$ is nonzero, an affine $B$-plane parallel to $P$ has
at most seven points: eight points would determine at least seven directions
by finite lemma 2. There are at most six rational line factors in all. Thus
the pencil through $\ell$ has total occupancy at most
\[
 12\cdot7+6(11-7)=108,
\]
whereas its actual occupancy is $55+11\cdot5=110$.

If $B$ has no five-point line, every projective line meets $D(B)$ in at least
four points: some parallel affine plane contains five noncollinear points.
Also $D(B)$ contains no full projective line, since it is contained in the
partner's $U$ set. Put $N=|D(B)|$, $n_L=|D(B)\cap L|$, and let $k\le6$ count
the rational line factors. On nonfactor lines, $4\le n_L\le6$; on factor
lines, $n_L\le11$. Therefore
\[
 N^2-109N+3192
 =\sum_L(n_L-4)(n_L-6)\le35k\le210.
\]
The quadratic has minimum $887/4>210$, a contradiction.

The moment polynomial $M_j^E(v)=\sum_{e\in E}(v\cdot e)^j$ vanishes on
$D(A)$ for $1\le j\le9$, by spectral equidistribution. The rigidity just
proved therefore gives $M_j^E=M_j^A=0$ for $1\le j\le6$. Since $j<11$,
all mixed moments of total degree at most six also vanish.

\end{proof}
\begin{lemma}\label{p11:55-unique}
A size-$55$ spectral set has at most one orientation in which it lies in
five parallel planes.
\end{lemma}
\begin{proof}
Suppose $B$ lies in five parallel planes in each of two distinct normal
directions. Projecting onto these two coordinates gives weights $f(s,t)$
on $S\times T$, where $|S|=|T|=5$. Every row and column has integer sum eleven,
and $0\le f(s,t)\le5$ by the line bound.

Define barycentric weights
\[
 w_S(s)=\left(\prod_{s'\in S\setminus\{s\}}(s-s')\right)^{-1},
 \qquad
 w_T(t)=\left(\prod_{t'\in T\setminus\{t\}}(t-t')\right)^{-1}.
\]
The mixed moments through degree six imply
\[
 f(s,t)=w_S(s)w_T(t)(as+bt+c)\pmod {11}.
\]
Indeed, the $22$ monomials $s^it^j$ with $0\le i,j\le4$ and $i+j\le6$
are linearly independent on the grid. Their orthogonal complement has
dimension three, and the displayed three-dimensional family lies in it by
the barycentric interpolation identities.

There are six affine orbits of five-subsets of $\F_{11}$, represented by
\[
01234,\quad01235,\quad01236,\quad01237,\quad01245,\quad01247.
\]
The exact checker \path{check_grid55.py} examines these $36$ ordered pairs
and all $11^3$ coefficient triples. None of the resulting residue matrices
has entries in $\{0,\ldots,5\}$ and all row and column sums eleven.
Consequently a spectral set of size $55$ has at most one such normal direction.

\end{proof}
\begin{lemma}\label{p11:55-line4}
Every affine line contains at most four points of either member.
\end{lemma}
\begin{proof}
Suppose $E$ has a five-point line of direction $d$. If
$t=|U_E\cap d^\perp|$, its pencil count gives
\[
110\le5t+11(12-t),\qquad t\le3.
\]
Every $A$-slice parallel to $d^\perp$ determines at most three
directions. Five noncollinear points would determine at least four,
and six collinear points violate the line bound five. Thus each slice
has at most five points; the total mass $55$ forces all eleven slices
to contain exactly five points, all collinear. Among the resulting eleven five-point lines, at least four share
a direction $u$. Any three of these parallel lines are noncoplanar, since a
plane contains at most eleven points of $A$.

Take three such lines $\ell_1,\ell_2,\ell_3$. Each pair spans an affine plane
and determines at least ten directions: nine cross-directions by
Cauchy--Davenport, plus $u$. The parallel-plane criterion of
Fallon--Mayeli--Villano forces $E$ into five parallel planes with normal
$n_{12}$ in the linear plane spanned by $\ell_1,\ell_2$, and likewise a normal
$n_{13}$ for the other pair. This location of the normal follows from the
projection proof of that criterion. The two linear planes intersect exactly
in $u$. Neither normal equals $u$, since $u\in D(A)\subset U_E$ and a
five-plane normal is not an equidistribution direction. Thus the normals
are distinct, contradicting the preceding uniqueness statement.
By symmetry, both sets have line bound four.

\end{proof}
\begin{lemma}\label{p11:55-sevenmom}
Every mixed moment of either member through degree seven vanishes.
\end{lemma}
\begin{proof}
Now both $D(E)$ and $D(A)$ meet every projective line in at least four points.
Suppose $F=M_7^A$ were nonzero; it vanishes on $D(E)$.

No five rational line factors of $F$ are concurrent. Otherwise a nonfactor
line through their common point would contain at most $1+(7-5)=3$ distinct
zeros of $F$, contrary to the four-point lower bound for $D(E)$.

If $E$ has a four-point line $\ell$, its pencil has total occupancy
$55+11\cdot4=99$. At least four of its planes contain exactly five points,
because $U_E$ meets every projective line in at least four points. Any plane
containing at least nine points of $E$ determines at least eight directions
by finite lemma 3, so its direction plane must be a rational line factor of
$F$ through the direction of $\ell$. There are at most four such large planes.
Consequently the pencil occupancy is at most
\[
 4\cdot5+4\cdot11+4\cdot8=96<99,
\]
a contradiction. This bound remains valid if there are more than four
five-point planes or fewer than four large planes.

If $E$ has no four-point line, finite lemma 4 shows that on every nonfactor
projective line of $F$, every parallel $E$-slice has at most seven points.
The moments through degree six say that the eleven plane counts, modulo
eleven, are represented by a cubic: the orthogonal complement of
$1,t,\ldots,t^6$ in $\F_{11}^{11}$ is the evaluation space of cubics.
Finite lemma 6 forces all these counts to be five. Thus at least $133-7=126$
normal directions lie in $U_E$, contradicting $|U_E|\le121$.

Therefore $M_7^A=0$, and symmetry gives $M_7^E=0$. A hypothetical size-55
pair has all mixed moments through degree seven equal to zero, line bound
four, and at most one five-plane normal for each member. In particular,
every parallel-plane count function is a quadratic modulo eleven.

\end{proof}

\subsection{Excluding five parallel layers}
\begin{lemma}\label{p11:55-no-layers}
Neither member of a size-$55$ spectral pair lies in five parallel planes.
Consequently every projective line contains at most nine uniform normals.
\end{lemma}
\begin{proof}
Suppose $E$ lies in planes $z=s$, $s\in S$, $|S|=5$. Each has eleven
points. A common missing direction makes all slices graphs $B_s$ in one
coordinate system. Degree-four Lagrange interpolation in $z$ isolates
each slice's moments through degree three. Lemma~\ref{p11:flat-graphs}
therefore applies, with exactly the same six graph types as in the
four-layer argument.

Put $Q_s(u,v)=\sum_{B_s}(ux+vy)^4$ and
$\lambda_s=(\prod_{r\in S\setminus\{s\}}(s-r))^{-1}$.
The global moments give $\sum_s s^jQ_s=0$ for $0\le j\le3$;
hence $Q_s=\lambda_sQ$ for one quartic $Q$.
For a noncubic, nonparabolic type the homothety conclusion of
Lemma~\ref{p11:flat-graphs} gives $B_s=\rho_sB_0+v_s$.
Modulo $Q$ times linear forms, the fifth moments are
$\rho_s^5[R_{B_0}]$, with nonzero fixed class. The fifth-moment
equation requires $\sum_s\rho_s^5=0$. This is a sum of five signs,
an odd integer between $-5$ and $5$, and cannot vanish modulo eleven.

For the cubic type the common simple and triple factors of $Q$ align
all slice axes, giving $f_s(x)=a_s(x-b_s)^3+c_s$ after one common
linear change. As in Lemma~\ref{p11:44-cubic},
\[
 Q_s=-4a_s^3uv^3,\qquad \sum_{B_s}x^4y^2=-a_s^2.
\]
The sixth and seventh moments therefore require
\[
 A(S):=\sum_s\lambda_s^4=0,\qquad
 B(S):=\sum_s s\lambda_s^4=0.
\]
The following table checks the six affine classes of five-subsets of
$\F_{11}$. All entries are in $\F_{11}$.
\[
\begin{array}{c|rrrrrr}
S&01234&01235&01236&01237&01245&01247\\\hline
\text{orbit size}&55&110&110&55&110&22\\
A(S)&10&2&3&2&9&0\\
B(S)&9&10&4&3&9&8
\end{array}
\]
The orbit sizes sum to $\binom{11}{5}=462$. Under $S\mapsto aS+b$,
the pair transforms to
$(a^{-16}A,a^{-16}(aB+bA))$, so simultaneous vanishing is affine
invariant. The table excludes the cubic type.

Finally, if $Q=0$, every slice is parabolic, with
$f_s(x)=a_sx^2+b_sx+c_s$, $a_s\ne0$. At any frequency $(u,t,w)$
with $t\ne0$, completing the square expresses the Fourier coefficient,
up to a nonzero Gauss factor, as five signed eleventh roots of unity.
Such a sum cannot vanish: its eleven integer coefficients would have
to be equal, and fewer than eleven signed terms cannot give a nonzero
common coefficient; five signs cannot cancel in pairs. All Fourier
zeros would therefore have $t=0$. Spectrality would place all
differences of $A$ in that plane, forcing its $55$ points into one
affine plane, contrary to the plane bound eleven.

This excludes every type. The final assertion now follows from
Lemma~\ref{p11:parallel-planes}.
\end{proof}

\subsection{A moment obstruction for parallel crosses}
\begin{lemma}[Five-point direction lemma]\label{p11:55-cross}
Every noncollinear five-point subset of $\F_{11}^2$ that determines at most four directions is an affine image of
\[ \{(0,0),(1,0),(-1,0),(0,1),(0,-1)\}. \]
\end{lemma}
\begin{proof}
If no three points are collinear, each direction accounts for at most two pairs, whereas the set has ten pairs. Four collinear points and a fifth point determine at least five directions. Thus precisely three points lie on some line. Normalize these to $A\times\{0\}$, with $|A|=3$, and one remaining point to $P=(0,1)$; write the other as $Q=(q,r)$, where $r\ne0$. The three directions from $P$ to the line exhaust the directions other than that line. The same holds for $Q$. Parametrizing each cross-direction by its horizontal displacement divided by its vertical displacement gives $(A-q)/r=A$, equivalently $q+rA=A$. If $r=1$, this is a translation preserving a three-point set, hence the translation is trivial and $Q=P$, impossible. Otherwise this affine permutation is conjugate to multiplication by $r$, so its order divides ten. Its action on a three-point set is faithful (an affine map fixing two points is the identity), so its order also divides six. Consequently its order is two and $r=-1$. It follows that $A$ is an arithmetic progression centered at $q/2$; $P,Q$ are opposite about that same center. This is the required affine cross.
\end{proof}

\begin{lemma}[Moment obstruction for eleven cross slices]\label{p11:cross55complex}
Let $E\subset\F_{11}^3$ have $55$ points and line intersection bound four. Suppose every polynomial of total degree at most seven has zero sum on $E$ in $\F_{11}$. Then every projective line meets the direction set $D(E)$ in at least five points.
\end{lemma}
\begin{proof}
Suppose a two-dimensional direction space contains at most four directions of $E$. Its eleven parallel affine planes each contain at most five points. To see this without an additional direction theorem, suppose a slice contains six points. Four collinear points plus an external point determine at least five directions. If no three are collinear, the fifteen pairs require at least five directions, since any one direction accounts for at most three pairs. In the remaining case choose a five-point subset containing three collinear points. The preceding lemma normalizes it to the standard cross. A sixth point compatible with its four directions must lie on an axis or one of the two diagonals through the origin. On an axis, its directions to the two points on the other axis force it to belong to the cross already. On the diagonal $(a,a)$, its directions to $(1,0)$ and $(-1,0)$ respectively force
\[ a\in\{1,1/2\},\qquad a\in\{-1,-1/2\}, \]
which are disjoint in $\F_{11}$. Reflection handles the other diagonal. Each plane therefore contains exactly five points, and the preceding lemma makes each slice a cross with the same four directions. Normalize the first cross so these directions have slopes $0,\infty,1,-1$.

For this fixed set of directions, the two opposite-arm direction pairs of a cross must be either $\{0,\infty\}$ or $\{1,-1\}$. Indeed the other four choices are equivalent under the symmetries of this direction set; choosing $\{0,1\}$ forces one diagonal vertical and the other of slope $1/2$, rather than $-1$. The corresponding crosses, up to translation, are
\[ \{0,\pm(r,0),\pm(0,r)\},\qquad
   \{0,(r,r),(r,-r),(-r,r),(-r,-r)\}. \]

Work in $K=\mathbb F_{121}=\F_{11}(i)$ with $i^2=-1$, identifying $(x,y)$ with $w=x+iy$ and writing $s$ for the slice coordinate. Each slice is
\[ c_s+\bigl(\{0\}\cup\{u\in K:u^4=R_s\}\bigr),\qquad R_s\in\F_{11}^{\times}. \]
For the two displayed types, $R_s=r^4$ and $R_s=-4r^4$, respectively. Its four nonzero roots have a common norm
\[ t_s=u\bar u=u^{12}=R_s^3\in\F_{11}^{\times},\qquad R_s=t_s^7. \]
The assumed moment identities extend by linearity from $\F_{11}$ to $K$.

We repeatedly use the following elementary fact. If a function $h:\F_{11}\to K$ has $\sum_s s^j h(s)=0$ for $0\le j\le k$, its unique polynomial representative of degree at most ten has degree at most $9-k$. This follows from $\sum_s s^m=0$ except at positive multiples of ten, where the sum is $-1$.

The first three pure $w$ moments of a slice are $5c_s$, $5c_s^2$, and $5c_s^3$. Their weighted global identities show successively that $c(s)$ has degree at most three, that $c(s)^2$ has degree at most four (forcing its cubic coefficient to vanish), and that $c(s)^3$ has degree at most five (forcing its quadratic coefficient to vanish). Thus $c(s)$ is affine. An affine shear removes every center and preserves all the hypotheses.

The centered slice moments of $w\bar w$, $w^2\bar w^2$, and $w^3\bar w^3$ are $4t_s$, $4t_s^2$, and $4t_s^3$. Their weighted identities yield
\[ \deg t\le4,\qquad \deg(t^2)\le6,\qquad \deg(t^3)\le8. \]
The last two inequalities successively force $\deg t\le3$ and $\deg t\le2$. The pure fourth moment is $4R_s=4t_s^7$; hence the reduced polynomial representing $t^7$ has degree at most six.

The function $t$ never vanishes. A nonconstant linear polynomial is therefore impossible. If $t$ is quadratic, translate $s$ to write $t=a s^2+b$, where $a,b\ne0$. In $t^7$, reduction modulo $s^{11}-s$ sends the possible terms $s^{14},s^{12}$ to $s^4,s^2$. Its coefficient of $s^{10}$ is unchanged and equals
\[ \binom{7}{5}a^5b^2=10a^5b^2\ne0, \]
contradicting the degree-six bound. Hence $t$, and therefore $R$, is constant. All centered slices are identical; $E$ contains five parallel full affine lines, contradicting its line bound four.
\end{proof}

\begin{theorem}\label{p11:exclude55}
There is no spectral pair of cardinality $55$ in $\F_{11}^3$.
\end{theorem}
\begin{proof}
The preceding lemmas give line bound four, moments through degree seven,
and at most nine uniform normals on every projective line. The cross-slice
lemma shows that both difference-direction sets meet every projective
line in at least five points.
First, a nonzero homogeneous polynomial $F$ of degree eight cannot vanish
on either direction set $D(B)$. Put $N=|D(B)|$ and
$n_L=|D(B)\cap L|$. If $L$ is not a rational line factor of $F$, then
$5\le n_L\le8$. If $L$ is such a factor, then $5\le n_L\le9$, since
$D(B)\subset U_{B'}$ for the partner $B'$. There are at most eight distinct
rational line factors. Projective incidence counting therefore gives
\[
 \begin{split}
 N^2-145N+5320
 &=\sum_{L\subset\operatorname{PG}(2,11)}(n_L-5)(n_L-8)\\
 &\le4\cdot8=32.
 \end{split}
\]
Indeed, there are $133$ projective lines, each point lies on twelve, and
each pair of points lies on one; hence
$\sum_Ln_L=12N$ and $\sum_Ln_L^2=N^2+11N$.
But the quadratic on the left has real minimum $255/4>32$.

The eighth moment polynomial of either member vanishes on the partner's
direction set, so the preceding rigidity gives $M_8^E=M_8^A=0$.
Together with the earlier moment vanishings, every parallel-plane count
function $n(t)$ is affine modulo eleven. For a nonconstant affine residue
function, the least nonnegative residues are exactly $0,1,\ldots,10$,
whose sum is already $55$; hence these are the actual integer counts.
For a constant residue function, the integer sum $55$ permits only the
uniform profile $(5,\ldots,5)$ or five entries eleven and six entries zero.
The latter is excluded by the no-five-layer reduction. Consequently every
affine plane contains at most ten points.

Suppose that $B$ contains a four-point line $\ell$. The twelve planes
through $\ell$ have total occupancy $55+11\cdot4=99$. At least five of
their normal directions lie in $U_B$, since the partner's direction set
is five-blocking and is contained in $U_B$. Those five planes have exactly
five points each; the other seven have at most ten. This bounds the total
by $5\cdot5+7\cdot10=95$, a contradiction. Thus both sets have line bound
three.

Lemma~\ref{p11:ten-point} states that every ten-point planar set
with line bound three determines at least ten directions. Any nonconstant affine plane-count profile would
produce such a ten-point plane; its direction plane would then contain at
least ten directions from $D(B)\subset U_{B'}$, contrary to the upper
bound nine. Therefore every profile is uniform. This is impossible:
the twelve uniform planes through any fixed affine line would have total
occupancy $60$, whereas the exact total is $55+11a$ for some integer $a$.

\end{proof}


\section{\texorpdfstring{Excluding cardinalities $66$ and $77$}{Excluding cardinalities 66 and 77}}
\begin{proposition}[The two central cardinalities]\label{p11:central-sizes}
Let $p\ge5$ be prime and let $E\subset\F_p^3$ have cardinality
$mp$, where $2\le m\le p-1$. Suppose every affine plane meets $E$ in
at most $p$ points, every affine line in at most $p-m$ points, and
\[
 \sum_{x\in E}x^\alpha=0\qquad(0\le|\alpha|\le p-3).
\]
Then $m=(p-1)/2$ or $m=(p+1)/2$.
\end{proposition}
\begin{proof}
Fix a nonzero normal $v$ and write $n_t=|\{x\in E:v\cdot x=t\}|$.
The vanishing moments give
\[
 \sum_{t\in\F_p} t^j n_t=0\qquad(0\le j\le p-3).
\]
The evaluation vectors of $1,t,\ldots,t^{p-3}$ have rank $p-2$;
their annihilator is spanned by the evaluations of $1,t$. Thus
$n_t=at+b$ modulo $p$.

If $a\ne0$, the residues of the $n_t$ are $0,1,\ldots,p-1$ once each.
Because $0\le n_t\le p$, the zero residue is represented by either
zero or $p$, and all other residues have a unique representative.
Hence
\[
 |E|=\frac{p(p-1)}2\quad\hbox{or}\quad\frac{p(p+1)}2.
\]
Suppose neither cardinality occurs. Every normal then has either the
uniform profile $(m,\ldots,m)$ or a profile consisting only of zeros
and $p$'s, with $m$ nonzero entries.

Choose an affine secant line $\ell$ and put $a=|E\cap\ell|$.
Then $2\le a\le p-m$. All $p+1$ planes through $\ell$ are nonempty,
so their occupancies are $m$ or $p$. If $t$ of them have occupancy $m$,
the pencil count gives
\[
 mp+pa=mt+p(p+1-t),\qquad
 (p-m)t=p(p+1-m-a).
\]
As $\gcd(p,p-m)=1$, we have $p\mid t$. The range $0\le t\le p+1$
allows only $t=0,p$. These give $a=p+1-m$ or $a=1$, respectively,
contradicting the secant bounds in both cases.
\end{proof}

\begin{lemma}[Common eighth-moment reduction]\label{p11:66-77-moments}
If $(E,A)$ is a spectral pair in $\F_{11}^3$ of cardinality
$66$ or $77$, all mixed moments of both members through degree eight
vanish.
\end{lemma}
\begin{proof}
Let $U$ be the uniform-normal directions of $E$. In every affine
parallel family of planes some slice of $A$ contains at least six
points. The published line bound is at most five, so this slice
contains six noncollinear points. The six-point direction lemma
therefore gives $|U\cap L|\ge6$ for every projective line $L$.
Also $|U\cap L|\le11$: twelve uniform directions in a pencil would
give $11m+11a=12m$, impossible for an integer $a$ and $m=6,7$.

Suppose a nonzero homogeneous polynomial of degree at most eight
vanishes on $U$. Multiplying by a power of a nonzero linear form, we
may assume its degree is eight. Put $N=|U|$, $n_L=|U\cap L|$, and
let $k\le8$ be its number of distinct rational line factors. On a
nonfactor line $6\le n_L\le8$; on a factor line $6\le n_L\le11$.
The projective incidence identities therefore imply
\[
 N^2-157N+6384
 =\sum_L(n_L-6)(n_L-8)\le15k\le120.
\]
The quadratic on the left has minimum $887/4>120$, a contradiction.
For each $1\le j\le8$, the polynomial
$M_j(v)=\sum_{x\in E}(v\cdot x)^j$ vanishes on $U$, because every
uniform plane has $m$ points and $\sum_t t^j=0$ in $\F_{11}$.
Thus every $M_j$ is identically zero. Its multinomial coefficients
are invertible modulo eleven, giving all mixed moments. Spectral
duality gives the same conclusion for $A$.
\end{proof}

\begin{corollary}\label{p11:exclude77}
There is no spectral subset of $\F_{11}^3$ of cardinality $77$.
\end{corollary}
\begin{proof}
Apply the eighth-moment reduction and the two-central-cardinalities
proposition with $p=11$ and $m=7$. The latter allows only $m=5,6$.
\end{proof}
\subsection{\texorpdfstring{The degree-nine zero set at cardinality $66$}{The degree-nine zero set at cardinality 66}}
Suppose for contradiction that $(E,A)$ is a spectral pair of size 66, and put $U=\U_E$.

\begin{lemma}\label{p11:lem:66bounds}
Every projective line meets $U$ in between six and eleven points. Every affine line meets $E$ or $A$ in at most three points.
\end{lemma}
\begin{proof}
For any two-dimensional linear subspace $P$, some translate of $P$ contains at least six points of $A$. The published five-point line bound makes these noncollinear; Lemma \ref{p11:six-point} supplies six directions in $\D(A)\cap P\subset U\cap P$. A full projective line in $U$ is impossible by \eqref{p11:pencil}. For a line containing $a$ points of $E$, at least six incident planes contain six points and all others at most eleven. Hence
\[
 66+11a\le 6\cdot6+6\cdot11=102,
\]
so $a\le3$. Apply symmetry.
\end{proof}

By Lemma~\ref{p11:66-77-moments}, every mixed moment through degree eight
vanishes. The proof of that lemma also shows that no nonzero homogeneous
polynomial of degree at most eight vanishes on $U$.

\begin{lemma}\label{p11:lem:66profiles}
Every normal direction has one of the following plane-occupancy profiles:
\[
 U:(6,\ldots,6),\qquad
 C:(11^6,0^5),\qquad
 R:(1,2,\ldots,11).
\]
If their numbers are $u,c,r$, then
\[
 u=64+2c,\qquad r=69-3c.
\]
The degree-nine polynomial $H=M_9$ is nonzero and its projective zero set is $Z=U\cup C$.
\end{lemma}
\begin{proof}
For a normal $v$, write $n_t=|E\cap\{x:v\cdot x=t\}|$. The nine equations for moments of orders zero through eight imply $n_t=a_vt+b_v$ in $\F_{11}$: the annihilator of the evaluations of $1,t,\ldots,t^8$ is spanned by $1,t$. If $a_v\ne0$, the residues run through $0,\ldots,10$, whose sum is 55; the zero residue must be represented by eleven to give total 66. If $a_v=0$, one obtains precisely $U$ or $C$. Also $M_9(v)=-a_v$.

Counting unordered pairs of points in all parallel-plane families gives
\[
 165u+330c+220r=12\binom{66}{2},\qquad u+c+r=133,
\]
and hence the stated formulas. If $H=0$, all plane occupancies are $0,6,11$. Choose a line containing $a=2$ or $3$ points. Its incident planes have positive occupancies. If $t$ of them have occupancy six, then
\[
 66+11a=6t+11(12-t),\qquad 5t=66-11a,
\]
which is impossible for either $a$.
\end{proof}

\begin{lemma}\label{p11:lem:66factors}
The polynomial $H$ is squarefree. Every rational line factor has exactly eleven $U$-points and one $C$-point. If $k$ is the number of these factors, the only possibilities are
\[
\begin{array}{c|c|c|c}
c&u&r&k\\\hline
7&78&48&9\\
8&80&45&7,9\\
9&82&42&5,6,7,9\\
10&84&39&9
\end{array}
\]
No five rational line factors are concurrent.
\end{lemma}
\begin{proof}
If $H$ had a repeated irreducible factor, its squarefree part would have degree at most eight and vanish on $U$, contradicting Lemma \ref{p11:66-77-moments}. On a factor line, all directions have type $U$ or $C$. Apply the pencil equation to a parallel line through any point of $E$: $1\le a\le3$ and $5t=66-11a$ force $a=1,t=11$.

On a nonfactor line both $U$ and $Z$ have between six and nine points. On a factor line their occupancies are eleven and twelve. The incidence identities therefore give
\[
 u^2-169u+7182\le10k,\qquad
 (u+c)^2-169(u+c)+7182\le18k.
\]
Substitute $u=64+2c$, $0\le c\le23$, $1\le k\le9$. The resulting integer possibilities are those displayed, with $k=8$ removed because squarefreeness would force a ninth line factor.

If $t\ge5$ factors passed through $P$, a nonfactor line through $P$ would have at most $1+(9-t)\le5$ zeros of $H$, contradicting its at least six $U$-points.
\end{proof}

\begin{theorem}\label{p11:thm:66}
There is no spectral subset of $\F_{11}^3$ of size 66.
\end{theorem}
\begin{proof}
A union of $k$ projective lines with maximum concurrency $t$ has at most
\[
 \min\left\{12k-\frac{k(k-1)}t,\quad
 11t+1+(k-t)(12-t)\right\}
\]
points. The first inequality counts pairwise intersections; the second begins with a maximal pencil. For $t\le4$ the respective upper bounds for $k=5,6,7,9$ are $53,61,70,85$.

A degree-$s$ curve with no rational line component has at most $1+12(s-1)$ rational points: count along the twelve lines through a rational point of the curve (the empty case is immediate). Consequently the bounds for the zero set of $H$ are $53+37=90$ when $k=5$, $61+25=86$ when $k=6$, $70+13=83$ when $k=7$, and $85$ when $k=9$.

Since $|Z|=64+3c$, all rows of Lemma \ref{p11:lem:66factors} are excluded except $(c,u,r,k)=(7,78,48,9)$, for which $|Z|=85$. Equality requires four factors through a point $P$: concurrency at most three gives at most 84 points. Their union has 45 points, and the remaining five lines must contribute eight new points each. Thus every intersection between two of the remaining five lies on the original four.

Every other line through $P$ consequently meets $Z$ in exactly six points: $P$ and five distinct intersections. All six must belong to $U$. Every $C$-point therefore lies on the original four factor lines. Each factor line has exactly one $C$-point, giving $c\le4$, contrary to $c=7$.
\end{proof}


\section{\texorpdfstring{Fuglede's conjecture in $\F_{11}^3$}{Fuglede's conjecture in F11 cubed}}\label{p11:conclusion}
\begin{theorem}\label{p11:theorem}
A nonempty subset of $\F_{11}^3$ is spectral if and only if it tiles by
translation.
\end{theorem}
\begin{proof}
Proposition~\ref{p11:cardinalities} reduces the spectral-to-tiling
implication to excluding cardinalities $22,33,44,55,66,77$. These are
excluded in the preceding sections. The converse was proved in
Section~\ref{p11:prelim}.
\end{proof}

\subsection{\texorpdfstring{Comparison with the proof over $\F_7$}{Comparison with the proof over F7}}
Fallon, Mayeli, and Villano \cite{FMV} reduce their proof over $\F_7$
to cardinalities $14,21,28$. Their argument and ours start with the same
Fourier-zero equidistribution, symmetry of spectral pairs, and
restrictions on intersections with lines and planes. In both proofs,
several uniform directions in one pencil force parallel layers. Their
most difficult case, cardinality $21$, is settled using planar direction
bounds, projected multiplicities and Cauchy--Davenport arguments.

One of their decisive planar assertions does not persist over $\F_{11}$.
Every five-point cap in $\F_7^2$ determines at least six directions
\cite[Lemma 5]{FMV}. In contrast,
\[
 C=\{(0,0),(0,1),(1,0),(1,4),(4,1)\}\subset\F_{11}^2
\]
is a cap with precisely five directions, of slopes $0,3,4,10,\infty$.
Each direction occurs twice among its ten pairs. In their normalization,
a five-direction cap would require $b^2-b-1=0$. Its discriminant $5$
is a nonsquare modulo seven but a square modulo eleven. Thus the
failure is arithmetic, even though both primes are $3$ modulo four.

The new arguments overcome this failure by combining geometric bounds
with polynomial information. The planar projection identity forces
low-degree line-count profiles. Incidence inequalities force mixed
moments to vanish, and these moments restrict plane profiles and slice
shapes. At cardinality $44$, weighted incidence converts a profile
energy gap into the divisibility contradiction $14\nmid936$.
At cardinality $55$, the cross-slice moment argument replaces a longer
classification. At cardinalities $66$ and $77$, the same eighth-moment
argument applies, after which the two-central-cardinalities proposition
immediately removes $77$.

The proof remains computer assisted. Its planar direction and graph
classifications are finite, and cardinality $33$ requires the larger
mixed-parity height calculation. The determinant obstruction then
removes all centrally symmetric lifts without a further height search.
These distinctions identify which parts of the method are already
general and which still depend on finite calculations over $\F_{11}$.


\section*{Acknowledgments}
\addcontentsline{toc}{section}{Acknowledgments}
The authors acknowledge the use of OpenAI's ChatGPT for assistance in
writing the computer programs used in this work and for collaboration
in drafting and revising this paper.
\clearpage
\phantomsection
\addcontentsline{toc}{section}{References}


\clearpage
\appendix
\section{The dimension-eight computations}\label{b8:appendix}

This appendix specifies the finite inputs and the code used in
Lemmas~\ref{b8:h16}, \ref{b8:h32}, and \ref{b8:tile16}.
Points of $\F_2^r$ are represented by integers in $[0,2^r)$;
vector addition is bitwise exclusive-or. Thus every arithmetic
operation in the certificate checks is exact. Python must be run
without \texttt{-O} or \texttt{-OO}, which would remove assertions.

\subsection{Data, coverage, and independent checks}

The order-$32$ input is McKay's \texttt{had32\_tr.txt.gz},
representing the classification of \cite{KTR}. Its SHA-256 digest is
\begin{quote}\footnotesize\ttfamily
e51a18c458e1bae995e335e5067b53078e7e8c6901278e755d33c1b281547520.
\end{quote}
The perfect-code input is \texttt{anc/perfect15} from the
source archive of \texttt{arXiv:0806.2513v3}. Its digest is
\begin{quote}\footnotesize\ttfamily
3bcb77324a7770c44f64b785d67d8da0f003113451575e06ca24852c7e5084f3.
\end{quote}
In the latter file the words are represented by sixteen-bit
hexadecimal integers with physical bit $3$ always zero. Deleting
this unused bit gives the fifteen coordinates used here.

The large classification files are retrieved by the reproduction
program and checked against these digests. The retained matrices,
codes, and positive certificates are supplied with the manuscript.
The extraction programs scan every input record. The Hadamard
filter dephases and computes the rank over $\F_2$; the perfect-code
filter computes the translation kernel and retains dimensions at
least seven. An independent Hadamard filter uses column spaces
and the opposite pivot order. An independent perfect-code check
recovers the kernel as
\[
 M=\bigl\langle\supp W_C\bigr\rangle^\perp.
\]
To see this identity, $C+u=C$ is equivalent, by Fourier inversion,
to $(-1)^{u\cdot v}W_C(v)=W_C(v)$ for every $v$. The second
calculation therefore does not reuse the translation test for
the first. Both extractions are compared exactly with the
retained records, rather than merely with their numbers. In the full
reproduction for this manuscript, both complete-source scans and
both independent audits passed, covering all $6{,}857{,}010$
Hadamard records and all $5983$ perfect-code records.

The completeness of the Hadamard and perfect-code classifications
is an external mathematical input. The programs check the
extraction from those classifications and the resulting tiling
certificates; they do not reconstruct either classification.
In particular, a check of the retained positive certificates alone
does not certify that all relevant records were retained.

From the source-bundle root, the commands for the selected-record
checks and for complete extraction with the independent audit are,
respectively,
\begin{lstlisting}[language=bash]
cd code/binary/support/d7_d8_codes
python3 run_checks.py --dimension 8
python3 run_checks.py --dimension 8 --full --audit
\end{lstlisting}
The first command uses Python 3.10 or later and its standard
library. The full extraction additionally requires a C++17 compiler
and zlib; the independent Fourier audit requires NumPy. The
expected output includes five order-$16$ projection checks,
$332$ order-$32$ subgroup certificates, and $69$ perfect-code
quotient certificates, all with zero failures.

\subsection{\texorpdfstring{The five order-$16$ matrices}{The five order-16 matrices}}

The following is the order-$16$ part of the supplied verifier.
Each hexadecimal word records one Hadamard row, with the most
significant bit first; a set bit represents $+1$. The five rows
of \texttt{H16\_DATA} are the five matrices of Hall's
classification in McKay's ordering. Besides orthogonality, the
code checks the five explicit projection certificates in
Lemma~\ref{b8:h16}. Together with the helper and invocation shown
here, this excerpt can be run independently.

\begin{lstlisting}[language=Python]
def require(condition, message):
    if not condition:
        raise AssertionError(message)
\end{lstlisting}

\begin{lstlisting}[language=Python]
H16_DATA = '''FFFF FF00 F0F0 F00F CCCC CC33 C3C3 C33C AAAA AA55 A5A5 A55A 9999 9966 9696 9669
FFFF FF00 F0F0 F00F CCCC CC33 C3C3 C33C AAAA AA55 A5A5 A55A 9996 9969 9699 9666
FFFF FF00 F0F0 F00F CCCC CC33 C3C3 C33C AAAA AA55 A599 A566 9996 9969 96A5 965A
FFFF FF00 F0F0 F00F CCCC CC33 C3C3 C33C AAAA A995 A659 A566 9A56 9969 96A5 959A
FFFF FF00 F0F0 F00F CCCC CC33 C3AA C355 AAC3 AA3C A599 A566 9996 9969 96A5 965A'''


def check_order16():
    columns = [(0,1,2,4,8), (0,1,2,8,12), (0,1,8,10,12),
               (0,1,8,10,12), (0,8,9,10,12)]
    for line, C in zip(H16_DATA.splitlines(), columns):
        H = [[1 if int(v, 16) >> (15-j) & 1 else -1
              for j in range(16)] for v in line.split()]
        require(len(H) == 16, 'Incomplete Hadamard data')
        for i in range(16):
            for j in range(16):
                require(sum(a*b for a,b in zip(H[i], H[j])) ==
                        (16 if i == j else 0), 'Not Hadamard')
        patterns = {tuple(row[c] * row[C[0]] for c in C[1:]) for row in H}
        require(len(patterns) == 16, 'Failed order-16 projection certificate')
    print('All five order-16 Hadamard projection certificates passed.')
\end{lstlisting}

\begin{lstlisting}[language=Python]
check_order16()
\end{lstlisting}

\subsection{\texorpdfstring{The order-$32$ subgroup certificates}{The order-32 subgroup certificates}}

We include the complete selected-record verifier. The function
\texttt{subspace\_bases} enumerates every reduced row-echelon
basis of the stated dimension: it chooses the pivot positions
and then every allowed nonpivot entry. Hence it omits no possible
subgroup. Finding one complement is enough, but every returned
complement is checked directly by forming all sums. Both matrix
orientations are processed, even when their canonical point sets
happen to be equivalent.

\begin{lstlisting}[language=Python]
#!/usr/bin/env python3
"""Exact subgroup certificates for order-32 log-Hadamard matrices of rank <=8.

Usage: python verify_fuglede8.py had32_rank_le8.txt fuglede8_certificates.json
The input is produced by fuglede8_scan.cpp from McKay's complete extended
classification. A failure to find a SUBGROUP complement is not a non-tiling
certificate. Both matrix orientations are checked. No solver is required.
"""
import itertools
import json
import sys
import time
from pathlib import Path


def span(basis):
    points = [0]
    for v in basis:
        points += [x ^ v for x in points]
    return points


def independent_basis(rows):
    pivots = {}
    basis = []
    for row in rows:
        v = row
        while v:
            j = v.bit_length()-1
            if j in pivots:
                v ^= pivots[j]
            else:
                pivots[j] = v
                basis.append(row)
                break
    return basis


def subspace_bases(n, k):
    for pivots in itertools.combinations(range(n), k):
        free = [(i,j) for i,p in enumerate(pivots)
                for j in range(p+1,n) if j not in pivots]
        for mask in range(1 << len(free)):
            basis = [1 << p for p in pivots]
            for a,(i,j) in enumerate(free):
                if mask >> a & 1:
                    basis[i] ^= 1 << j
            yield basis


def canonical_columns(rows):
    basis = independent_basis(rows)
    E = [sum(((b >> j)&1) << k for k,b in enumerate(basis))
         for j in range(32)]
    assert len(set(E)) == 32
    return E, len(basis)


def difference_mask(E):
    return sum(1 << x for x in {a ^ b for a in E for b in E} if x)


def main(source, target):
    start = time.perf_counter()
    spaces = {}
    for rank in range(5,9):
        spaces[rank] = [(basis, sum(1 << x for x in span(basis) if x))
                        for basis in subspace_bases(rank,rank-5)]
    assert len(spaces[8]) == 97155
    results = []
    histogram = {}
    for line in Path(source).read_text().splitlines():
        record, rank, *tokens = line.split()
        record, rank = int(record),int(rank)
        assert len(tokens) == 32
        H = [int(t,16) for t in tokens]
        assert all((H[i]^H[j]).bit_count()==16
                   for i in range(32) for j in range(i))
        L = [h ^ H[0] for h in H]
        L = [v ^ (0xffffffff if v&1 else 0) for v in L]
        assert len(independent_basis(L)) == rank
        histogram[rank] = histogram.get(rank,0)+1
        LT = [sum(((v >> j)&1) << i for i,v in enumerate(L))
              for j in range(32)]
        for orientation, matrix in [('columns',L),('rows',LT)]:
            E, r = canonical_columns(matrix)
            assert r == rank
            D = difference_mask(E)
            K = next((b for b,m in spaces[r] if not m&D),None)
            if K is not None:
                elements = span(K)
                assert len(elements) == 1 << (r-5)
                assert len({e ^ k for e in E for k in elements}) == 1 << r
            results.append(dict(record=record,rank=rank,orientation=orientation,
                                points=E,complement_basis=K))
    assert histogram == {5:1,6:2,7:13,8:150}, histogram
    failures = [r for r in results if r['complement_basis'] is None]
    output = dict(input=str(source),histogram=histogram,oriented_cases=len(results),
                  subgroup_failures=len(failures),results=results)
    Path(target).write_text(json.dumps(output,indent=2)+'\n')
    print('Classes by rank:',histogram,flush=True)
    print('Oriented cases:',len(results),'subgroup failures:',len(failures),flush=True)
    for r in failures:
        print('NO SUBGROUP CERTIFICATE:',r['record'],r['orientation'],flush=True)
    print('Seconds:',round(time.perf_counter()-start,3),flush=True)


if __name__ == '__main__':
    main(*sys.argv[1:])
\end{lstlisting}

\subsection{The perfect-code quotient certificates}

Here is the second complete selected-record verifier. For each
code it first verifies the perfect-code partition, computes the
entire translation kernel, and forms a quotient using an
annihilator basis. The sixteen distinct syndrome vectors are
exactly the image of the radius-one Hamming ball. A subgroup
complement of that image proves the needed assertion by
Lemma~\ref{b8:tile16}; the original tiling complement need not
itself be a subgroup.

\begin{lstlisting}[language=Python]
#!/usr/bin/env python3
"""Check canonical 16-point tiles coming from perfect-code kernels.

Usage: python verify_fuglede8_tiles16.py perfect15_kernel_ge7.txt tiles16_certificates.json

Coverage: normalize E to contain zero and restrict a tiling to its span V,
of dimension r<=8. Let H:F2^15->V have the 15 nonzero points of E as
columns. C=H^{-1}(T intersect V) is a perfect length-15 code, and ker(H),
of dimension 15-r>=7, lies in M=Ker(C). The image of the radius-one
ball modulo M has 16 distinct points: M contains no words of weight 1
or 2. The quotient V->F2^15/M is injective on E. A subgroup complement
for the ball image therefore pulls back to a subgroup complement for E.
Coordinate permutations and translations of C preserve this quotient
configuration up to linear equivalence. Thus the 69 classes with kernel
dimension >=7 in the complete 5983-code classification cover EVERY
16-point tile in dimension <=8. No dimension-seven Fuglede theorem is
needed in this reduction. The 48 dimension-seven kernels are the largest
quotient cases; all 21 larger-kernel cases are checked as well.
The external completeness theorem is Ostergard--Pottonen, arXiv:0806.2513.
"""
import json
import sys
import time
from pathlib import Path
from verify_fuglede8 import independent_basis, span, subspace_bases, difference_mask


def annihilator(basis,n):
    return independent_basis(v for v in range(1<<n)
                             if all((v&b).bit_count()%2==0 for b in basis))


def main(source,target):
    start=time.perf_counter()
    spaces={r:[(b,sum(1<<v for v in span(b) if v))
               for b in subspace_bases(r,r-4)] for r in range(4,9)}
    assert len(spaces[8])==200787
    output=[]
    histogram={}
    for line in Path(source).read_text().splitlines():
        record,k,*tokens=line.split()
        record,k=int(record),int(k)
        C={int(t,16) for t in tokens}
        assert len(tokens)==len(C)==2048 and 0 in C
        assert all(0<=x<32768 for x in C)
        balls={x^v for x in C for v in [0]+[1<<i for i in range(15)]}
        assert len(balls)==32768
        M={p for p in C if all(x^p in C for x in C)}
        assert len(M)==1<<k
        mbasis=independent_basis(M)
        assert len(mbasis)==k and set(span(mbasis))==M
        H=annihilator(mbasis,15)
        r=len(H)
        assert r==15-k
        def syndrome(x):
            return sum(((x&h).bit_count()%2)<<i for i,h in enumerate(H))
        E=[0]+[syndrome(1<<i) for i in range(15)]
        assert len(set(E))==16
        T=sorted({syndrome(c) for c in C})
        assert len(T)==1<<(r-4)
        assert len({e^t for e in E for t in T})==1<<r
        D=difference_mask(E)
        K=next((b for b,m in spaces[r] if not m&D),None)
        if K is not None:
            assert len({e^v for e in E for v in span(K)})==1<<r
        histogram[k]=histogram.get(k,0)+1
        output.append(dict(record=record,kernel_dimension=k,code_rank=len(independent_basis(C)),
                           kernel_basis=mbasis,quotient_map=H,points=E,
                           known_tiling_complement=T,subgroup_complement_basis=K))
        print('code',record,'kernel',k,'subgroup certificate',K,flush=True)
    assert histogram=={7:48,8:17,9:3,11:1},histogram
    failures=[o for o in output if o['subgroup_complement_basis'] is None]
    Path(target).write_text(json.dumps(dict(kernel_histogram=histogram,
        cases=len(output),subgroup_failures=len(failures),results=output),indent=2)+'\n')
    print('Cases',len(output),'failures',len(failures),'seconds',round(time.perf_counter()-start,3),flush=True)


if __name__=='__main__':
    main(*sys.argv[1:])
\end{lstlisting}


\section{The computations in dimension nine}\label{b9:appendix}
The dimension-nine computations use the same integer and binary-mask
conventions as the preceding appendix. The directory
\path{code/binary/} contains the complete proof packet, including its
input register, extraction programs, certificate generators, independent
verifiers, and finite certificate data. The accompanying logs distinguish
a check of a retained certificate from a complete regeneration of the
family it covers. These are different obligations.

\subsection{Canonical models and sixteen-point tiles}
The directory \path{support/d9_finite/} extends the dimension-eight
canonical calculation to dephased rank nine. There are $2,040$ Hadamard
classes and $4,080$ oriented canonical models. The perfect-code calculation
has $374$ kernels of dimension at least six. The combined file therefore
has exactly $4,454$ cases. For every case, the checker verifies the rank
over $\F_2$ of the recorded basis and enumerates all sums with the set being
tiled. These checks prove the complement assertions for the selected
models. Full extraction from the published databases, with the pinned
source hashes, is run separately by the following commands from that
directory:
\begin{verbatim}
python3 audit_fuglede9.py
python3 reproduce_fuglede9.py --full
\end{verbatim}
Both matrix orientations are retained. The code-kernel extraction checks
invariance of the complete code under the proposed kernel, rather than
inferring the kernel from a generator matrix for a nonlinear code.

\subsection{\texorpdfstring{The order-$48$ exclusion}{The order-48 exclusion}}
The directory \path{support/d9_order48/} contains the complete finite
extension argument. Its inputs are the published triply even code
matrices. The dimension-ten list contains $1,247$ codes containing the all-ones word; it concerns the
zero-cubic obstruction. The corresponding dimension-nine list has $1,960$ codes; it
is the list filtered to obtain the $26$ projective starting codes for
the extension calculation. These counts refer to different dimensions.

For a chosen starting code $D$, the generator enumerates the $10,795$
rank-two alternating bilinear forms on $D/\langle\mathbf1\rangle$.
For each form, Gaussian elimination over $\F_2$ solves the linear equations on $g$.
The solution kernel is $(D*D)^\perp$ and contains $D$, so enumeration
uses a basis of the quotient by $D$. It then applies exactly the two
proved filters: $\operatorname{wt}(g)\equiv0\pmod4$ and
$38\le\dim(C*C)\le43$. A separately implemented generator reverses
the elimination convention and compares extension cosets after canonical
reduction. This verifies coverage without requiring both programs to
choose the same representatives.

The surviving $490,816$ point sets each have a five-vector positive
certificate. The following complete C++17 checker constructs the set and
its difference set and verifies all $32$ subset sums. It also requires
exactly one certificate for each retained candidate; repeated certificates
or omitted candidates cause failure.
\begin{lstlisting}[language=C++,basicstyle=\ttfamily\scriptsize,breaklines=true]
// Independent positive-certificate checker. Does not call the search routine.
// For each recorded E, verifies32distinct vectors of a5-space all lie inE+E.
#include <array>
#include <cstdint>
#include <fstream>
#include <iostream>
#include <map>
#include <stdexcept>
#include <string>
#include <unordered_set>
using U=uint64_t;
static void require(bool ok,const char*why){if(!ok)throw std::runtime_error(why);}
static bool word(std::ifstream&f,unsigned n,U&v){
    v=0;for(unsigned k=0;k<n;k++){int c=f.get();if(c==EOF){require(k==0,"truncated word");return false;}v|=U((unsigned char)c)<<(8*k);}return true;
}
static void header(std::ifstream&f,const char*expected){char b[8];f.read(b,8);require(bool(f)&&std::string(b,8)==std::string(expected,8),"bad header");}
int main(int argc,char**argv){
    require(argc>=4,"usage: verifier BASES SURVIVORS CERT1 [CERT...] ");
    std::ifstream bf(argv[1]);unsigned nb;bf>>nb;
    std::map<unsigned,std::array<unsigned,48>>point;
    for(unsigned k=0;k<nb;k++){
        unsigned id,h,nk,ns;bf>>std::dec>>id>>h>>nk>>ns;std::array<U,9>D;U skip;
        for(U&v:D)bf>>std::hex>>v;for(unsigned j=0;j<nk+ns;j++)bf>>std::hex>>skip;
        require(bool(bf),"truncated bases");require(D[0]==((1ULL<<48)-1),"ones generator missing");
        std::array<unsigned,48>P{};for(unsigned j=0;j<48;j++)for(unsigned i=0;i<8;i++)P[j]|=((D[i+1]>>j)&1)<<i;point[id]=P;
    }
    std::ifstream sf(argv[2],std::ios::binary);header(sf,"F48LFT1\n");
    std::unordered_set<U>pending;U id,g;uint64_t expected=0;
    while(word(sf,4,id)){require(word(sf,8,g),"truncated survivor");require(id<65536&&g<(1ULL<<48),"key range");require(pending.insert((id<<48)|g).second,"duplicate survivor");expected++;}
    uint64_t checked=0;
    for(int arg=3;arg<argc;arg++){
        std::ifstream cf(argv[arg],std::ios::binary);header(cf,"F48FLT1\n");
        while(word(cf,4,id)){
            require(word(cf,8,g),"truncated extension");std::array<unsigned,5>B;for(unsigned&i:B){U x;require(word(cf,2,x),"truncated basis");require(x<512,"basis range");i=x;}
            require(pending.erase((id<<48)|g)==1,"extra or duplicate certificate");
            auto it=point.find(id);require(it!=point.end(),"unknown base");
            std::array<unsigned,48>E;std::array<bool,512>seen{};
            for(unsigned j=0;j<48;j++){E[j]=it->second[j]|(((g>>j)&1)<<8);require(!seen[E[j]],"repeated candidate point");seen[E[j]]=true;}
            std::array<bool,512>diff{};
            for(unsigned x:E)for(unsigned y:E)diff[x^y]=true;
            seen.fill(false);
            for(unsigned mask=0;mask<32;mask++){
                unsigned x=0;for(unsigned k=0;k<5;k++)if(mask>>k&1)x^=B[k];
                require(!seen[x],"dependent certificate basis");require(diff[x],"certificate vector outside E+E");seen[x]=true;
            }
            checked++;
        }
    }
    require(pending.empty(),"survivors lack certificates");require(checked==expected,"coverage count");
    std::cout<<"PASS: "<<checked<<" distinct candidate sets; all32vectors of every5-space verified inE+E; exactcoverage, noextras.\n";
}
\end{lstlisting}

From \path{support/d9_order48/}, the complete driver commands are
\begin{verbatim}
python3 reproduce_fuglede48_exclusion.py --verify-only
python3 reproduce_fuglede48_exclusion.py
\end{verbatim}
The first checks the positive certificates. The second additionally
regenerates the complete extension families and compares the independently
generated retained cosets. The counts are $26,582$ feasible prescribed
forms, $10,294,816$ extension representatives, $5,748,704$ representatives
of weight divisible by four, and $490,816$ representatives satisfying the
quadratic bound. Every retained representative has a verified five-space;
none remains unresolved.

\subsection{The tile-to-spectral direction}
The directory \path{support/d9_tile_to_spectral/} contains the three
finite calculations used for tiles of size $32$.
First, two independent orbit algorithms cover all $50,955,971$
three-dimensional subspaces of the eleven-dimensional quotient of Boolean
functions by affine functions. Both give $494$ orbits. For each orbit
representative, a support clique is checked by direct Walsh sums.
Second, all $134,688$ perfect-code kernel configurations are generated;
$134,398$ have direct support-clique certificates, and all $290$ exceptions
are explicitly mapped to $T_*$. Third, two exact-cover searches recover
the same $12,096$ normalized complements of $T_*$. Their $1,788$
difference profiles give $1,599,366$ unordered pairs, including the
diagonal. Every pair is covered by at least one of $75$ explicitly
recorded four-dimensional subgroup complements.
\begin{verbatim}
python3 audit_tile32_9.py --enumerate-fibers
python3 reproduce_tile32_9.py --full
\end{verbatim}
The first command independently checks certificates and enumerates the
exceptional fibers. The second rebuilds the orbit, kernel, and profile
calculations. In particular, the proof does not assume that two spectral
fibers automatically have a common spectrum.

\subsection{\texorpdfstring{The partial $64$-point results}{The partial 64-point results}}
The periodic-spectrum certificates are in
\path{support/d9_size64_reductions/}. Their $332$ canonical models agree
with the order-$32$, rank-at-most-eight extraction used in dimension eight,
but the certificates have a different purpose: they require all complement
differences to lie in Fourier support. There are $304$ subgroup
certificates and $28$ general set certificates. The checker validates the
raw matrix, both orientations, the coordinate maps, and every support
condition. Individual ordinary tiling complements would not suffice here.
The short fiber-weight enumeration includes every partition satisfying
the stated divisibility, lower-bound, and ambient-capacity constraints,
and records a positive subgroup basis whenever the weight-only criterion
succeeds.

The proof for the seven unequal four-fiber patterns uses the small
self-dual-code classification~\cite{Pless,HMdata}. Its source files are
included, and a separate extraction checker compares their code lists
with the retained models. The following is the complete positive
odd-coset checker. Its coverage count is decisive: the minimum-distance
four models occur in lengths $8,12,16,16,16$, and their odd cosets have
respectively $8,32,128,128,128$ weight-one or weight-three representatives.
\begin{lstlisting}[language=Python,basicstyle=\ttfamily\scriptsize,breaklines=true]
#!/usr/bin/env python3
"""Check all small self-dual models and positive odd-coset coverings.

Classification completeness remains external (Pless 1972; Harada-Munemasa).
The source list includes all models of length12/16, and the unique min4
length8 model. Models of minimum2 do not enter the cap lemma.
"""
import itertools
import json
from pathlib import Path

ROOT=Path(__file__).resolve().parent
data=json.loads((ROOT/'selfdual_small_codes.json').read_text())
out=[]
for index,item in enumerate(data['codes']):
    n=item['n'];rows=item['rows'];C=[0]
    for row in rows:C += [x^row for x in C]
    assert len(C)==len(set(C))==2**(n//2)
    assert all((a&b).bit_count()%2==0 for a in rows for b in rows)
    assert (1<<n)-1 in C
    minimum=min(x.bit_count() for x in C if x)
    if minimum<4:continue
    def syndrome(v):
        return sum(((v&r).bit_count()%2)<<i for i,r in enumerate(rows))
    leaders={}
    for weight in [1,3]:
        for loc in itertools.combinations(range(n),weight):
            v=sum(1<<j for j in loc)
            leaders.setdefault(syndrome(v),v)
    # Every odd coset is distinct from every even coset, and exactly half
    # the cosets are odd because every codeword is even.
    assert len(leaders)==2**(n//2-1)
    assert all(v.bit_count() in [1,3] for v in leaders.values())
    assert all(syndrome(v)==s for s,v in leaders.items())
    out.append(dict(source_index=index,length=n,minimum_distance=minimum,
                    odd_cosets_certified=len(leaders),leaders=leaders))
assert sorted(x['length'] for x in out)==[8,12,16,16,16]
(ROOT/'small_selfdual_coset_certificates.json').write_text(json.dumps(out,indent=2)+'\n')
print('PASS: five min-distance-four self-dual models; every odd coset has a weight-one/three leader.')
\end{lstlisting}
The main proof uses this self-dual-code route. An alternative Boolean
classification check is retained in the archive as supplementary material;
it is not a dependency of the theorem, since the original full thesis
table underlying that alternative was not recovered.

\subsection{The eighth-density example}
The following complete checker needs only the Python standard library.
It verifies spectrality in both orientations, the Boolean convolution
identity, the complete list of support planes, the absence of any
three-space extending a support plane, and the eight-point tiling
complement. Thus this example does not depend on any classification.
\begin{lstlisting}[language=Python,basicstyle=\ttfamily\scriptsize,breaklines=true]
#!/usr/bin/env python3
"""Exact counterexample to the Fourier-support-cube version of an eighth-density lemma.

Python 3.10+, standard library only. Run without -O.
Integers encode binary vectors; addition is XOR.

B and E below are a 64-point spectral pair in F_2^9. Nevertheless,
supp(W_B) contains NO three-dimensional linear subspace. This disproves
the proposed support-cube lemma even when B is spectral. It does NOT
disprove spectral-implies-tile, or existence of a subgroup complement
outside that Fourier support. In fact an eight-point complement is verified.

The 32-point factor was obtained from the canonical model H43323_rows
in the earlier Hadamard packet. No classification theorem, source file,
third-party package, or numerical tolerance is needed to check this example.
"""
from itertools import combinations
from time import perf_counter

L = [0,128,120,248,68,196,60,188,34,162,90,218,153,25,225,97,
     238,110,150,22,83,211,43,171,207,79,183,55,141,13,245,117]
A = [0,1,2,3,4,5,6,7,8,16,32,56,64,88,104,112,
     128,152,168,176,200,208,224,248,247,239,223,199,191,167,151,143]
T = [0,25,43,52,77,87,98,126]

def walsh(points, u):
    return sum(1 - 2*((u & x).bit_count() % 2) for x in points)

def main():
    start = perf_counter()
    B = L + [x ^ 256 for x in L]
    E = A + [x ^ 256 for x in A]
    assert len(set(B)) == len(set(E)) == 64
    assert all(0 <= x < 512 for x in B + E)
    W = [walsh(B, u) for u in range(512)]
    assert all(W[x ^ y] == 0 for x, y in combinations(E, 2))
    # Also check the transposed orthogonality directly.
    assert all(walsh(E, x ^ y) == 0 for x, y in combinations(B, 2))
    assert all(sum(W[t]*W[s ^ t] for t in range(512)) == 512*W[s]
               for s in range(512))
    S = {u for u in range(512) if W[u] != 0}
    assert len(S) == 29 and all(u < 256 for u in S)
    planes = {tuple(sorted((0,x,y,x ^ y)))
              for x,y in combinations(sorted(S - {0}), 2) if x ^ y in S}
    assert len(planes) == 56
    # Every 3-space contains a plane P and is P union (z+P) for z outside P.
    extensions = [(P,z) for P in planes for z in S - set(P)
                  if all(z ^ p in S for p in P)]
    assert not extensions
    assert len(set(T)) == 8 and all(x ^ y in S for x in T for y in T)
    assert len({e ^ t for e in E for t in T}) == 512
    assert any(x ^ y not in T for x in T for y in T)
    print('PASS: size-64 spectral pair in F_2^9; Boolean convolution identity.')
    print('PASS: Fourier support has 29 points, 56 planes, and no 3-space.')
    print('PASS: explicit eight-point non-subspace complement tiles E.')
    print('Support:', sorted(S))
    print('Complement:', T)
    print('Seconds:', round(perf_counter()-start, 3))

if __name__ == '__main__':
    main()
\end{lstlisting}
The subgroup complements stated in Proposition~\ref{b9:counterexample}
can also be verified by the following additional code, using the same
lists \texttt{A} and \texttt{L}. The products are in $\F_2^9$, with the
ninth bit encoded by $256$.
\begin{lstlisting}[language=Python,basicstyle=\ttfamily\small,breaklines=true]
def span(basis):
    result = [0]
    for v in basis:
        result += [x ^ v for x in result]
    return result
E = A + [x ^ 256 for x in A]
B = L + [x ^ 256 for x in L]
for points, basis in [(E, [25, 42, 76]), (B, [1, 2, 4])]:
    K = span(basis)
    assert len(set(K)) == 8
    assert len({x ^ k for x in points for k in K}) == 512
\end{lstlisting}

\subsection{Reproduction and scope}
From \path{code/binary/}, the retained analytic and certificate checks
can be run together by
\begin{verbatim}
python3 verify_packet.py --tile32 --order48
\end{verbatim}
Python must be run with assertions enabled. Python $3.10$ or later and
a C++17 compiler suffice for these checks; the full database extraction
also uses zlib and NumPy. The full-generation commands above use the
published classifications as inputs, with their exact source hashes.
They do not reconstruct those classification theorems from first
principles. No timeout or unsuccessful search is used as a proof of
nonexistence. There is no exhaustive order-$64$ Hadamard classification
in this argument, and the uniform four-fiber case remains open.


\section{\texorpdfstring{Exact computations for $\F_{11}^3$}{Exact computations for F11 cubed}}\label{p11:code}
The computational steps in the odd-prime proof use finite loops and exact
arithmetic in $\F_{11}$. They do not use a numerical feasibility solver.
The source distribution contains the original programs in
\path{code/prime11/original/}, with their relative paths intact, and the
independent verification programs in \path{code/prime11/audit/checks/}.
Python~3.9 or later, Node.js and a C++17 compiler suffice for these programs.
Run Python without the \texttt{-O} option, since the verification conditions
are assertions. The complete reproduction driver is
\path{code/prime11/reproduce.py}; it executes the checks described below.

\begin{center}
\begin{tabular}{>{\raggedright\arraybackslash}p{1.1cm}>{\raggedright\arraybackslash}p{6.5cm}>{\raggedright\arraybackslash}p{5.2cm}}
\toprule
Size&Finite task&Coverage or required output\\\midrule
22&Normalized planar caps and all admissible quartic/quintic profile tuples&
No tuple survives the required Radon identities; the equal-chord cap maximum is six.\\
33&Planar profiles, harmonic masks, six parallel triples, mixed-parity graph heights&
$1,948,717$ even classes; $242,202,917$ odd kernel vectors; no full mixed solution.\\
44&Small planar direction bounds and moment-flat graph classification&
$177,156$ normalized graphs, $375$ moment-flat, $254$ with line bound four.\\
55&Planar direction bounds, cubic profiles and two-orientation grids&
$36$ affine grid types; no admissible grid; six cubic-layer types excluded.\\
66,77&Six-point directions and small integer incidence checks&
$36$ masks, $1,566$ extra-point triples; no compatible triple.\\
\bottomrule
\end{tabular}
\end{center}

\subsection{Why the enumerations are exhaustive}
For a planar noncollinear set, an affine transformation sends a chosen
triangle to $(0,0),(1,0),(0,1)$. Its direction mask then contains
$0,-1,\infty$. The mask searches enumerate every remaining choice of
directions and every compatible extension of that triangle; insertion
is rejected only for a forbidden direction or line occupancy. The
profile searches enumerate polynomial coefficients, using precisely the
Vandermonde kernels and coordinate translations proved in the text.

For moment-flat graphs, subtracting a vertical affine function gives
$f(0)=f(1)=0$. The linear moments then give
$f(x)=\sum_{j=2}^7c_j(x^j-x)$. A nonzero vertical scaling makes its
first nonzero coefficient one. The $177,156$ candidates therefore cover
every graph under consideration. The proportional-fourth-moment test
checks equality of the complete candidate list with the homothety list;
it does not infer a classification from examples.

For cardinality $33$, affine height subtraction fixes the constant and
two odd linear degrees of freedom. Every nonzero even part has a unique
projective representative. For each representative, the program solves
the mixed equations by exact Gaussian elimination and enumerates its
entire odd kernel. A second implementation evaluates the antipodal-pair
moments directly. Both retain $21,611$ quadratic solutions, of which four
pass the cubic and quartic equations and none passes degree five.
The pure-odd height search is unnecessary: the determinant lemma supplies
the final contradiction analytically.

\subsection{The small direction and incidence verifier}
The following complete program checks the stronger six-point direction
lemma used in both the $66$ and $77$ cases, and the subsequent arithmetic.
The numerical counts are checks on exhaustive loops, not their inputs.
\begin{lstlisting}[language=Python,basicstyle=\ttfamily\scriptsize,breaklines=true]
"""Independent small finite checks for the size-66 and size-77 audit.

The direction census uses all 36 masks, no orbit representatives or
precomputed candidate partitions. All arithmetic is exact.
"""
from itertools import combinations, product
from fractions import Fraction
import json

p = 11
triangle = ((0, 0), (1, 0), (0, 1))
def slope(a, b):
    x, y = (b[0]-a[0]) % p, (b[1]-a[1]) % p
    return p if x == 0 else y * pow(x, -1, p) % p

nodes = triples = edges = 0
for ab in combinations(range(1, 10), 2):
    directions = {0, 10, 11, *ab}
    candidates = [v for v in product(range(p), repeat=2)
                  if v not in triangle and
                  all(slope(v, t) in directions for t in triangle)]
    nodes += len(candidates)
    edges += sum(slope(a,b) in directions for a,b in combinations(candidates,2))
    for a,b,c in combinations(candidates,3):
        triples += 1
        assert not all(slope(v,w) in directions for v,w in ((a,b),(a,c),(b,c)))
assert triples == 1566

# Check the integer minimization behind degree-eight rigidity.
assert min(n*n-157*n+6384 for n in range(134)) == 222
assert 222 > 8 * 15

# Reconstruct all affine-residue occupancy profiles, retaining mass 66 or77.
profile_counts = {}
for mass in (66,77):
    counts = {'uniform':0,'constant_zero':0,'nonconstant':0}
    for a,b in product(range(p), repeat=2):
        residues = [(a*t+b)%p for t in range(p)]
        zeros = [t for t,n in enumerate(residues) if n == 0]
        for mask in range(1 << len(zeros)):
            values = residues[:]
            for j,t in enumerate(zeros):
                if mask & (1 << j): values[t] = p
            if sum(values) != mass: continue
            kind = 'nonconstant' if a else ('uniform' if b else 'constant_zero')
            counts[kind] += 1
    profile_counts[mass] = counts
assert profile_counts[66] == {'uniform':1,'constant_zero':462,'nonconstant':110}
assert profile_counts[77] == {'uniform':1,'constant_zero':330,'nonconstant':0}

# All admissible c,k pairs from the two incidence inequalities, including k=0.
possibilities = []
for c,k in product(range(24),range(10)):
    u,z = 64+2*c,64+3*c
    if u*u-169*u+7182 <= 10*k and z*z-169*z+7182 <= 18*k and k != 8:
        possibilities.append((c,k))
assert possibilities == [(7,9),(8,7),(8,9),(9,5),(9,6),(9,7),(9,9),(10,9)]

union_bounds = {}
for k in (5,6,7,9):
    bounds = [min(Fraction(12*k)-Fraction(k*(k-1),t),
                  11*t+1+(k-t)*(12-t)) for t in (2,3,4)]
    union_bounds[k] = int(max(bounds))
assert union_bounds == {5:53,6:61,7:70,9:85}
surviving = []
for c,k in possibilities:
    residual = 0 if k == 9 else 1+12*(9-k-1)
    if 64+3*c <= union_bounds[k]+residual:
        surviving.append((c,k))
assert surviving == [(7,9)]

# The new short size-77 ending is already impossible for every published
# allowed secant occupancy a=2,3,4.
assert all((55-11*a)%4 for a in (2,3,4))

print(json.dumps({'direction_masks':36,'candidate_vertices_counted_with_masks':nodes,
                  'candidate_edges':edges,'candidate_triples':triples,
                  'compatible_triples':0,'degree8_incidence_minimum':222,
                  'profile_counts':profile_counts,'incidence_cases':possibilities,
                  'line_union_bounds':union_bounds,'last_case':surviving,
                  'all_checks_passed':True},indent=2))
\end{lstlisting}

\subsection{The shortened height computation}
The next complete driver runs every retained size-$33$ dependency and
both exhaustive mixed-height implementations. Its argument is the
absolute path of \path{code/prime11/original/}. The invoked sources,
including the two C++ checkers and their exact matrix generators, are
included in that directory.
\begin{lstlisting}[language=Python,basicstyle=\ttfamily\scriptsize,breaklines=true]
"""Reproduce the shortened size-33 computation, omitting the pure-odd search.

Usage: python3 verify33_short.py /absolute/path/to/extracted/closure_packet
The packet must contain the original output/fuglede11_* tree. This script
runs in that tree and regenerates the finite results; use a working copy.
The determinant obstruction is a mathematical final step, not a code test.
"""
from pathlib import Path
import json
import os
import shutil
import subprocess
import sys
import tempfile

ROOT = Path(sys.argv[1]).resolve()
ALG = ROOT / 'output/fuglede11_round3/algebra'
AUD = ROOT / 'output/fuglede11_round3/quadrics'
CXX = os.environ.get('CXX') or shutil.which('g++') or shutil.which('clang++')
if not CXX:
    raise SystemExit('A C++17 compiler is required.')
LOGS = ROOT / 'size33_short_audit_logs'
LOGS.mkdir(exist_ok=True)

def run(args, name):
    print(name, flush=True)
    with (LOGS / (name + '.log')).open('w') as stream:
        subprocess.run(list(map(str, args)), cwd=ROOT, check=True,
                       stdout=stream, stderr=subprocess.STDOUT)

for relative in [
    'output/fuglede11_next/size33/verify_small_profiles.py',
    'output/fuglede11_next/size33/classify_eight_cliques.py',
    'output/fuglede11_round2/size33_profiles/verify_seven_lemma.py',
    'output/fuglede11_round2/size33_profiles/verify_three_set_difference.py',
    'output/fuglede11_round2/size33_profiles/six_equal_chords.py',
    'output/fuglede11_round2/size33_profiles/six_triple_skeletons.py',
    'output/fuglede11_round2/size33_seven/verify_branch_lemmas.py',
    'output/fuglede11_round3/dependencies/audit_inputs.py',
    'output/fuglede11_round3/algebra/prepare_odd.py',
    'output/fuglede11_round3/algebra/prepare_split.py',
    'output/fuglede11_round3/quadrics/rank_parity.py',
    'output/fuglede11_round3/quadrics/make_independent_solver.py',
]:
    run([sys.executable, ROOT / relative], Path(relative).stem)

with tempfile.TemporaryDirectory(prefix='size33-short-') as temp:
    jobs = [
        (ROOT / 'output/fuglede11_next/size33/verify_ten_points.cpp',
         'ten_points', []),
        (ALG / 'solve_split.cpp', 'primary_mixed', [ALG]),
        (AUD / 'independent_solve.cpp', 'independent_mixed', []),
    ]
    for source, name, args in jobs:
        binary = Path(temp) / name
        run([CXX, '-O3', '-std=c++17', source, '-o', binary], name + '_compile')
        run([binary, *args], name)
run([sys.executable, ALG / 'check_prequintic.py'], 'four_prequintic_cases')

primary = json.loads((ALG / 'split_results.json').read_text())
independent = json.loads((AUD / 'independent_solutions.json').read_text())
assert primary['even_projective_classes'] == independent['even'] == 1948717
assert primary['mixed_ranks'] == [0, 0, 0, 0, 187, 3052, 1945478]
assert independent['checks'] == 242202917
assert primary['quadratic_survivors'] == independent['quadratic_pass'] == 21611
assert primary['cubic_survivors'] == primary['quartic_survivors'] == 4
assert primary['full_moment_survivors'] == independent['full_pass'] == 0
result = {
    'status': 'PASS',
    'even_projective_classes': 1948717,
    'independent_odd_kernel_vectors_checked': 242202917,
    'nonzero_even_moment_solutions': 0,
    'conclusion': 'Every moment-compatible lift over B33 becomes odd after an affine height change.',
    'mathematical_final_step': 'Central symmetry contradicts spectrality by the determinant lemma.',
    'pure_odd_search': 'Omitted; unnecessary for the shortened proof.'
}
(LOGS / 'result.json').write_text(json.dumps(result, indent=2) + '\n')
print(json.dumps(result, indent=2))
\end{lstlisting}

For reference, the main mixed-height enumeration is printed below.
The data it reads are generated from the explicit support and moment
matrices in Lemma~\ref{p11:33-height}; they are not assumed certificates
of nonexistence.
\begin{lstlisting}[language=C++,basicstyle=\ttfamily\fontsize{7.5}{9}\selectfont,breaklines=true,breakatwhitespace=false,columns=fullflexible]
#include <array>
#include <cstdint>
#include <cstdio>
#include <fstream>
#include <vector>
#include <chrono>
#include "split_system.inc"
using U=uint64_t;
int ec[7],eh[16],ker[8][8],ohbasis[8][16],dim;
int target[4],qq[4][8][8],oc[8],oparam[8];
U even_count=0,ranks[7]={},quad_candidates=0,cubic_candidates=0,quartic_candidates=0,full_candidates=0,geom_survivors=0,lowrank_leaves=0;
std::ofstream solutions;
std::ofstream prequintic;
int inv[11]={0,1,6,4,3,9,2,8,7,5,10};
int mod(int x){x%=11;return x<0?x+11:x;}
const int R[16][2]={{0,4},{0,5},{1,0},{1,3},{1,4},{1,5},{1,6},{1,7},{1,10},{2,4},{2,5},{3,4},{4,0},{4,3},{4,4},{4,7}};
void finish(){
 quad_candidates++;
 int oh[16]={};for(int i=0;i<dim;i++)for(int t=0;t<16;t++)oh[t]+=oparam[i]*ohbasis[i][t];for(int t=0;t<16;t++)oh[t]%=11;
 int e2[16],o2[16],e3[16],o3[16];for(int t=0;t<16;t++){e2[t]=eh[t]*eh[t]%11;o2[t]=oh[t]*oh[t]%11;e3[t]=e2[t]*eh[t]%11;o3[t]=o2[t]*oh[t]%11;}
 for(int k=0;k<4;k++){int sum=0;for(int t=0;t<16;t++)sum+=WE[k][t]*(e3[t]+3*eh[t]*o2[t]);if(sum%11)return;}
 for(int k=0;k<2;k++){int sum=0;for(int t=0;t<16;t++)sum+=WO[k][t]*(3*e2[t]*oh[t]+o3[t]);if(sum%11)return;}
 cubic_candidates++;
 int sum=0;for(int t=0;t<16;t++)sum+=e2[t]*e2[t]+6*e2[t]*o2[t]+o2[t]*o2[t];if(sum%11)return;
 for(int k=0;k<2;k++){sum=0;for(int t=0;t<16;t++)sum+=WO[k][t]*(e3[t]*oh[t]+eh[t]*o3[t]);if(sum%11)return;}
 quartic_candidates++;
 for(int j=0;j<8;j++){oc[j]=0;for(int i=0;i<dim;i++)oc[j]+=oparam[i]*ker[i][j];oc[j]%=11;}
 prequintic<<"{\"e\":[";for(int i=0;i<7;i++){if(i)prequintic<<",";prequintic<<ec[i];}prequintic<<"],\"o\":[";for(int i=0;i<8;i++){if(i)prequintic<<",";prequintic<<oc[i];}prequintic<<"]}\n";
 sum=0;for(int t=0;t<16;t++)sum+=e3[t]*e2[t]+10*e3[t]*o2[t]+5*eh[t]*o2[t]*o2[t];if(sum%11)return;
 full_candidates++;
 int pts[33][3]={};for(int t=0;t<16;t++){pts[1+2*t][0]=R[t][0];pts[1+2*t][1]=R[t][1];pts[1+2*t][2]=(eh[t]+oh[t])%11;pts[2+2*t][0]=mod(-R[t][0]);pts[2+2*t][1]=mod(-R[t][1]);pts[2+2*t][2]=mod(eh[t]-oh[t]);}
 int maxplane=0;int witness[4]={};
 for(int pivot=0;pivot<3;pivot++)for(int u=0;u<(pivot==0?11:1);u++)for(int v=0;v<(pivot<2?11:1);v++){
  int normal[3]={0,0,0};normal[pivot]=1;if(pivot==0){normal[1]=u;normal[2]=v;}if(pivot==1)normal[2]=v;
  int counts[11]={};for(auto&p:pts)counts[(normal[0]*p[0]+normal[1]*p[1]+normal[2]*p[2])%11]++;
  for(int k=0;k<11;k++)if(counts[k]>maxplane){maxplane=counts[k];for(int j=0;j<3;j++)witness[j]=normal[j];witness[3]=k;}
 }
 if(maxplane<=7)geom_survivors++;
 for(int j=0;j<8;j++){oc[j]=0;for(int i=0;i<dim;i++)oc[j]+=oparam[i]*ker[i][j];oc[j]%=11;}
 solutions<<"{\"e\":[";for(int i=0;i<7;i++){if(i)solutions<<",";solutions<<ec[i];}solutions<<"],\"o\":[";for(int i=0;i<8;i++){if(i)solutions<<",";solutions<<oc[i];}solutions<<"],\"maxplane\":"<<maxplane<<",\"plane_witness\":[";for(int i=0;i<4;i++){if(i)solutions<<",";solutions<<witness[i];}solutions<<"]}\n";
}
void search_odd(int j,const std::array<int,4>& vals){
 if(j==dim){lowrank_leaves++;bool ok=true;for(int k=0;k<4;k++)if(vals[k]!=target[k])ok=false;if(ok)finish();return;}
 int lin[4]={};for(int k=0;k<4;k++)for(int i=0;i<j;i++)lin[k]+=2*qq[k][i][j]*oparam[i];for(int k=0;k<4;k++)lin[k]%=11;
 for(int v=0;v<11;v++){oparam[j]=v;std::array<int,4> next;bool ok=true;for(int k=0;k<4;k++){next[k]=(vals[k]+v*lin[k]+v*v*qq[k][j][j])%11;if(j==dim-1&&next[k]!=target[k]){ok=false;break;}}if(ok){if(j==dim-1)finish();else search_odd(j+1,next);}}
 oparam[j]=0;
}
void process_even(){
 even_count++;
 if(even_count%100000==0){printf("even %llu full %llu plane7 %llu ranks",(unsigned long long)even_count,(unsigned long long)full_candidates,(unsigned long long)geom_survivors);for(auto x:ranks)printf(" %llu",(unsigned long long)x);puts("");fflush(stdout);}
 int A[6][8]={};for(int k=0;k<6;k++)for(int j=0;j<8;j++){for(int i=0;i<7;i++)A[k][j]+=LM[k][i][j]*ec[i];A[k][j]%=11;}
 int rank=0,piv[6];bool ispiv[8]={};
 for(int j=0;j<8&&rank<6;j++){
  int p=rank;while(p<6&&!A[p][j])p++;if(p==6)continue;
  for(int k=0;k<8;k++)std::swap(A[p][k],A[rank][k]);int factor=inv[A[rank][j]];for(int k=j;k<8;k++)A[rank][k]=A[rank][k]*factor%11;
  for(int i=0;i<6;i++)if(i!=rank&&A[i][j]){factor=A[i][j];for(int k=j;k<8;k++)A[i][k]=mod(A[i][k]-factor*A[rank][k]);}
  piv[rank++]=j;ispiv[j]=true;
 }
 ranks[rank]++;dim=0;
 for(int j=0;j<8;j++)if(!ispiv[j]){for(int k=0;k<8;k++)ker[dim][k]=0;ker[dim][j]=1;for(int i=0;i<rank;i++)ker[dim][piv[i]]=mod(-A[i][j]);dim++;}
 for(int t=0;t<16;t++){eh[t]=0;for(int i=0;i<7;i++)eh[t]+=ec[i]*EB[i][t];eh[t]%=11;}
 for(int k=0;k<4;k++){int sum=0;for(int t=0;t<16;t++)sum+=WE[k][t]*eh[t]*eh[t];target[k]=mod(-sum);}
 for(int i=0;i<dim;i++)for(int t=0;t<16;t++){int sum=0;for(int j=0;j<8;j++)sum+=ker[i][j]*OB[j][t];ohbasis[i][t]=sum%11;}
 for(int k=0;k<4;k++)for(int i=0;i<dim;i++)for(int j=i;j<dim;j++){int sum=0;for(int t=0;t<16;t++)sum+=WE[k][t]*ohbasis[i][t]*ohbasis[j][t];qq[k][i][j]=qq[k][j][i]=sum%11;}
 if(dim==2){
  int C[4][4];for(int k=0;k<4;k++){C[k][0]=qq[k][0][0];C[k][1]=2*qq[k][0][1]%11;C[k][2]=qq[k][1][1];C[k][3]=target[k];}
  int rr=0,pp[3];for(int j=0;j<3&&rr<4;j++){int p=rr;while(p<4&&!C[p][j])p++;if(p==4)continue;for(int k=0;k<4;k++)std::swap(C[p][k],C[rr][k]);int f=inv[C[rr][j]];for(int k=j;k<4;k++)C[rr][k]=C[rr][k]*f%11;for(int i=0;i<4;i++)if(i!=rr&&C[i][j]){f=C[i][j];for(int k=j;k<4;k++)C[i][k]=mod(C[i][k]-f*C[rr][k]);}pp[rr++]=j;}
  for(int i=rr;i<4;i++)if(C[i][3])return;
  if(rr==3){int mon[3];for(int i=0;i<3;i++)mon[pp[i]]=C[i][3];if(mon[0]*mon[2]%11!=mon[1]*mon[1]%11)return;for(int u=0;u<11;u++)if(u*u%11==mon[0])for(int v=0;v<11;v++)if(v*v%11==mon[2]&&u*v%11==mon[1]){oparam[0]=u;oparam[1]=v;finish();}return;}
 }
 search_odd(0,{});
}
void enum_even(int j){if(j==7){process_even();return;}for(int v=0;v<11;v++){ec[j]=v;enum_even(j+1);}ec[j]=0;}
int main(int argc,char** argv){std::string out=argc>1?argv[1]:".";solutions.open(out+"/split_solutions.jsonl");prequintic.open(out+"/split_prequintic.jsonl");auto start=std::chrono::steady_clock::now();for(int first=0;first<7;first++){for(int j=0;j<7;j++)ec[j]=0;ec[first]=1;enum_even(first+1);}double elapsed=std::chrono::duration<double>(std::chrono::steady_clock::now()-start).count();
 std::ofstream f(out+"/split_results.json");f<<"{\"even_projective_classes\":"<<even_count<<",\"mixed_ranks\":[";for(int i=0;i<7;i++){if(i)f<<",";f<<ranks[i];}f<<"],\"quadratic_survivors\":"<<quad_candidates<<",\"cubic_survivors\":"<<cubic_candidates<<",\"quartic_survivors\":"<<quartic_candidates<<",\"full_moment_survivors\":"<<full_candidates<<",\"plane_bound_survivors\":"<<geom_survivors<<",\"elapsed_seconds\":"<<elapsed<<"}\n";
 printf("DONE even %llu quad %llu cubic %llu quartic %llu full %llu plane7 %llu time %.3f\n",(unsigned long long)even_count,(unsigned long long)quad_candidates,(unsigned long long)cubic_candidates,(unsigned long long)quartic_candidates,(unsigned long long)full_candidates,(unsigned long long)geom_survivors,elapsed);
}
\end{lstlisting}

\subsection{A small quartic certificate}
The size-$22$ special six-direction case can also be checked by the
following complete program. It reconstructs the centered profile list
and reduces the leading scalar to two square classes. In both cases the
constant terms sum to $1$, whereas the Radon identity requires $10$.
\begin{lstlisting}[language=Python,basicstyle=\ttfamily\scriptsize,breaklines=true]
"""A small certificate for the size-22 six-point-plane obstruction.

This is not a replacement for the entire size-22 checker. It independently
reconstructs the centered quartic catalogue and makes one stage transparent.
"""
from itertools import product
import json
from pathlib import Path

p=11
def values(c):
    return tuple(sum(a*pow(x,j,p) for j,a in enumerate(c))%p for x in range(p))

actual=set()
for a,b,c,d in product(range(p),repeat=4):
    if not a: continue
    coeff=(d,c,b,0,a)
    val=values(coeff)
    if max(val)<=6 and sum(val)==22:actual.add(coeff)

representatives=[(0,0,2,0,1),(2,0,8,0,1),(0,5,6,0,2),
                 (2,0,8,0,2),(4,0,5,0,2),(6,0,7,0,2)]
orbits=[{tuple(c*pow(t,j,p)%p for j,c in enumerate(rep))
         for t in range(1,p)} for rep in representatives]
assert actual==set.union(*orbits) and len(actual)==35
assert sum(map(len,orbits))==35

bad=(2,3,5,6,8,9)
w=[]
for s in bad:
    v=1
    for t in bad:
        if t!=s:v=v*(s-t)%p
    w.append(pow(v,-1,p))
assert w==[6,3,9,2,8,5]
certificate=[]
# Dilation of both plane coordinates acts on lambda by a fourth power.
# Fourth powers are exactly the nonzero squares, so 1 and 2 suffice.
assert {pow(t,4,p) for t in range(1,p)}=={1,3,4,5,9}
for lam in (1,2):
    lists=[[c for c in sorted(actual) if c[4]==lam*v%p]for v in w]
    survivors=[]
    for cs in product(*lists):
        if any(sum(c[2]*pow(s,h,p) for c,s in zip(cs,bad))%p for h in range(3)):continue
        survivors.append(cs)
    assert len(survivors)==1
    cs=survivors[0]
    assert sorted(c[0] for c in cs)==[0,0,0,4,4,4]
    assert all(c[1]==0 for c in cs)
    assert sum(c[0]for c in cs)%p==1 != 10
    certificate.append({'lambda':lam,'coefficients':cs,'constant_sum':1,'required':10})
result={'orbit_representatives':representatives,'orbit_sizes':list(map(len,orbits)),
        'orbit_maximum_values':[max(values(r)) for r in representatives],
        'bad_slopes':bad,'weights':w,'two_scale_cases':certificate}
Path(__file__).with_name('compact_quartic_certificate.json').write_text(json.dumps(result,indent=2)+'\n')
print('PASS: six quartic orbits; both scale classes force 1 = 10 in F_11.')
\end{lstlisting}


\end{document}